\documentclass[11pt]{amsart}
\usepackage{amsmath}
\usepackage{amssymb}
\usepackage{mathrsfs}
\usepackage{esint}
\usepackage{tikz}
\usepackage{comment}
\usepackage{ulem}
\usepackage{xcolor}
\usepackage{hyperref}
\usepackage{lmodern}
\usepackage{microtype}
\usepackage{amsthm,mathtools}

\theoremstyle{plain}
\newtheorem{theorem}{Theorem}[section]
\newtheorem{lemma}[theorem]{Lemma}
\newtheorem{proposition}[theorem]{Proposition}
\newtheorem{cor}[theorem]{Corollary}

\theoremstyle{definition}

\newtheorem{definition}[theorem]{Definition}

\newtheorem{example}[theorem]{Example}

\numberwithin{equation}{section}

\makeatletter
\@namedef{subjclassname@2020}{%
  \textup{2020} Mathematics Subject Classification}
\makeatother

\def\be{\begin{equation}}
\def\ee{\end{equation}}

\def\R{\mathbb R}
\def\N{\mathbb N}

\def\Z{\mathbb Z}

\newcommand{\ind}{\mathfrak{M}}
\newcommand{\dd}{\,d}

\newcommand{\Lam}{\Lambda}

\newcommand{\xirad}{\xi}
\newcommand\norm[1]{\left\lVert#1\right\rVert}
\def\epsilon{\varepsilon}

\def\({\left(}
\def\){\right)}

\newcommand\abs[1]{\left| #1 \right|}

\begin{document}

\title[Solutions to the Yamabe Equation]
{A Parametrization of Solutions to the Yamabe Equation with Isolated Singularity}

\author[Han]{Qing Han}
\address{Department of Mathematics\\
University of Notre Dame\\
Notre Dame, IN 46556, USA} 
\email{qhan@nd.edu}
\author[Hung]{Jui-Yun Hung}
\address{Mathematical Sciences Institute\\
The Australian National University\\
Canberra, ACT 2600, Australia} 
\email{JUI-YUN.HUNG@anu.edu.au}
\author[Wu]{Nan Wu}
\address{School of Mathematics \& IMS\\
Nanjing University\\
Nanjing, 210093, China} 
\email{nwu2@nju.edu.cn}


\begin{abstract}
Inspired by a pioneering work by Caffarelli, Gidas, and Spruck \cite{Caffarelli1989AsymptoticSA}, 
we study positive solutions to the Yamabe equation with nonremovable isolated singularities.  
It is known that there is a formal one-to-one correspondence between the collection of 
the solutions to the Yamabe equation with nonremovable isolated singularities 
and the collection of harmonic functions 
through formal infinite series determined by the harmonic functions. 
In this paper, we prove that these formal infinite series yield actual solutions 
to the Yamabe equation with a nonremovable isolated singularity.
These solutions to the Yamabe equation have designated asymptotic behaviors 
determined by the harmonic functions. 
In a broader context, we provide a general method of constructing all solutions near their isolated singular points
and all solutions near infinity.
\end{abstract}

\thanks{Q. Han acknowledges the support of NSF Grant DMS-2305038.}


\maketitle

\section{Introduction}\label{sec-Intro}

In a pioneering paper \cite{Caffarelli1989AsymptoticSA}, 
Caffarelli, Gidas, and Spruck studied the Yamabe equation of the form
\begin{align}
    \label{eq-Yamabe-equation}
    \Delta u + \frac{1}{4} n (n-2) u^{\frac{n+2}{n-2}} = 0 \text{ in }B_1 \setminus \{0\},
\end{align}
and proved that any positive singular solutions of \eqref{eq-Yamabe-equation} in $B_1 \setminus \{0\}$ 
are asymptotic to radial singular solutions 
of \eqref{eq-Yamabe-equation} in $\R^n \setminus\{0\}$. 
Subsequently, in \cite{Korevaar1998RefinedAF}, Korevaar, Mazzeo, Pacard, and Schoen 
studied refined asymptotics and expanded solutions to the next order. 
More recently in \cite{Han2019AsymptoticEO}, 
Han, X. Li, and Y. Li studied expansions of positive singular solutions of \eqref{eq-Yamabe-equation} 
up to arbitrary orders. 
In this paper, we provide a complete characterization of positive singular solutions 
of \eqref{eq-Yamabe-equation} in $B_1 \setminus \{0\}$
and parameterize these solutions by harmonic functions. 
In a broader context, we present a general method of constructing solutions with singularities.
Throughout this paper, we always assume $n\ge 3$.

A radial singular solution to \eqref{eq-Yamabe-equation} in $\R^n\setminus \{0\}$ is of the form
$$\psi (x) =  |x|^{\frac{2-n}{2}} \xi(-\ln|x|),$$
where $\xi$ is a smooth positive periodic function, with $0 < \xi < 1$ on $\R$, satisfying
\begin{align}
\label{eq-ode-for-radial-singular-solution}
\xi''  - \frac{1}{4}(n-2)^2 \xi + \frac{1}{4}n(n-2)\xi^{\frac{n+2}{n-2}} = 0.
\end{align}
We first recall the main estimates established in \cite{Caffarelli1989AsymptoticSA}, 
\cite{Korevaar1998RefinedAF}, and \cite{Han2019AsymptoticEO}. 

Caffarelli,  Gidas, and Spruck \cite{Caffarelli1989AsymptoticSA} 
demonstrated that positive solutions of \eqref{eq-Yamabe-equation} in $B_1\setminus\{0\}$
with a nonremovable singularity at $x=0$ blow up at 0 up to a certain rate and proved the following result. 

\smallskip 
\noindent 
{\bf Theorem A.}
{\it Let $u$ be a positive solution of \eqref{eq-Yamabe-equation} in $B_1\setminus\{0\}$
with a nonremovable singularity at $x=0$. Then, there exists a smooth 
positive periodic solution $\xi$ to \eqref{eq-ode-for-radial-singular-solution} with $0 < \xi < 1$ on $\R$ such that, 
for any $x\in B_{1/2}\setminus\{0\}$, 
\begin{align*}
\big||x|^{\frac{n-2}2}u(x)-\xi(-\ln|x|)\big|\le C|x|^\alpha,
\end{align*}
where $C$ is a positive constant and $\alpha$ is a constant in $(0,1]$.} 
\smallskip 

The function $\xi$ is referred to as the {\it leading term} of $|x|^{\frac{n-2}2}u(x)$. 
Next, we write $x=|x|\theta$ for $x\in B_1\setminus \{0\}$ and $\theta\in \mathbb S^{n-1}$. 
Korevaar, Mazzeo, Pacard, and Schoen \cite{Korevaar1998RefinedAF}
studied the expansion of solutions to the next order and proved the following result. 

\smallskip 
\noindent 
{\bf Theorem B.}
{\it Let $u$ be a positive solution of \eqref{eq-Yamabe-equation} in $B_1\setminus\{0\}$
with a nonremovable singularity at $x=0$ 
and $\xi$ be the smooth positive periodic solution to \eqref{eq-ode-for-radial-singular-solution} on $\mathbb R$ 
as in Theorem A. 
Then, for any $x\in B_{1/2}\setminus\{0\}$, 
\begin{align*}
\big||x|^{\frac{n-2}2}u(x)-\xi(-\ln|x|)-\xi_1(-\ln|x|)X_1(\theta)|x|\big|\le C|x|^\gamma,\end{align*}
where $C$ is a positive constant, $\gamma$ is a constant in $(1,2)$, 
$X_1$ is a spherical harmonic of degree 1, and $\xi_1$ is a smooth periodic function with the same period as $\xi$.}

Han, X. Li, and Y. Li  \cite{Han2019AsymptoticEO} established expansions 
up to arbitrary orders.

\smallskip 
\noindent 
{\bf Theorem C.}
{\it Let $u$ be a positive solution of \eqref{eq-Yamabe-equation} in $B_1\setminus\{0\}$
with a nonremovable singularity at $x=0$  
and $\xi$ be the smooth positive periodic solution to \eqref{eq-ode-for-radial-singular-solution} on $\mathbb R$ 
as in Theorem A. 
Then, there exist a strictly increasing sequence $\{\mu_i\}_{i\ge 1}$, 
divergent to $\infty$ and with $\mu_1=1$, 
and a sequence of bounded smooth functions $\{c_{ij}\}_{i\ge 1,0\le j\le i-1}$ in $B_{1/2}\setminus\{0\}$ such that, 
for any positive integer $\ell$ and any $x\in B_{1/2}\setminus\{0\}$, 
\begin{align}\label{eq-ch6-estimate-u-k-intro}
&\Big||x|^{\frac{n-2}2}u(x)-\xi(-\ln|x|)-\sum_{i=1}^\ell\sum_{j=0}^{i-1} 
c_{ij}(x)|x|^{\mu_i}(-\ln|x|)^j\Big|\le C_\ell|x|^{\mu_{\ell+1}} (-\ln |x|)^{\ell},
\end{align}
where $C_\ell$ is a positive constant.}

The set $\{\mu_i\}_{i\ge 1}$, 
to be defined in Definition \ref{def-index-set}, 
is referred to as the {\it index set} and 
determines the decay rates of $|x|^{\frac{n-2}2}u(x)-\xi(-\ln|x|)$ in the following pattern: 
\begin{align*}
|x|^{\mu_1},\, (-\ln|x|)|x|^{\mu_2},\, |x|^{\mu_2},\,\cdots,\,
(-\ln|x|)^{\ell-1}|x|^{\mu_\ell},\,\cdots,\, |x|^{\mu_\ell},\, \cdots.\end{align*}
We point out that the index set $\{\mu_i\}_{i\ge 1}$ is determined by the leading term $\xi$. 

Theorems A, B, and C start with the existence of 
a positive solution of \eqref{eq-Yamabe-equation} in $B_1\setminus\{0\}$,
with a nonremovable singularity at the origin. 
In this paper, we will investigate how to characterize and construct all these solutions. 
Based on Theorem C, Han and Y. Li \cite{HAN2021127} constructed positive solutions 
to \eqref{eq-Yamabe-equation} near $0$ 
with prescribed asymptotic expansions up to {\it fixed} finite orders. 
Here, we aim to construct positive solutions to \eqref{eq-Yamabe-equation} near $0$ 
with prescribed asymptotic expansions up to {\it arbitrary} orders.
To achieve this, we need to study structures of the coefficients $\{c_{ij}\}$ carefully. 

First, we improve Theorem C. 

\begin{theorem}\label{thrm-HHW-intro-expansion}
Let $u$ be a smooth positive solution of \eqref{eq-Yamabe-equation} in $B_1\setminus\{0\}$
with a nonremovable singularity at $x=0$  
and $\xi$ be the smooth positive periodic solution to \eqref{eq-ode-for-radial-singular-solution} on $\mathbb R$ 
as in Theorem A. 
Then, there exist a strictly increasing sequence $\{\mu_i\}_{i\ge 1}$, 
divergent to $\infty$ and with $\mu_1=1$, 
and a sequence of bounded smooth functions $\{c_{i}\}_{i\ge 1}$ in $B_{\epsilon}\setminus\{0\}$ such that, 
for any $x\in B_{\epsilon}\setminus\{0\}$, 
\begin{align}\label{eq-HHW-intro-expansion-infinite}
|x|^{\frac{n-2}2}u(x)=\xi(-\ln|x|)+\sum_{i=1}^\infty 
c_{i}(x)|x|^{\mu_i},
\end{align}
where the series converges uniformly and absolutely in $B_{\epsilon}\setminus\{0\}$, for some $\epsilon\in (0,1)$. 
Moreover, $t\mapsto c_{i}(e^{-t}\theta)$ is periodic in $t$
with the same period as $\xi$.
\end{theorem}


The most significant improvement of Theorem \ref{thrm-HHW-intro-expansion} over Theorem C 
is the convergence of the infinite series in \eqref{eq-HHW-intro-expansion-infinite}. 
Such a convergence implies the estimates of finite orders as in \eqref{eq-ch6-estimate-u-k-intro}. 
We also note that the monomials of $-\ln|x|$ are absent in the series in \eqref{eq-HHW-intro-expansion-infinite}
and hence $|x|^{\frac{n-2}2}u(x)-\xi(-\ln|x|)$ decays in the pattern  
$|x|^{\mu_1}, |x|^{\mu_2},\cdots, |x|^{\mu_\ell},\, \cdots.$
In particular, $c_{ij}$ in Theorem C vanishes if $j \ge 1$. 
In fact, there is a hidden improvement which plays a decisive role in the proof of the convergence
and is one of the major contributions in this paper. 

A careful analysis of the proof of Theorem C reveals a hidden structure of 
the coefficients $\{c_{i}\}$.  
These coefficients of course are determined by the solution $u$. 
Actually, they are determined uniquely by the leading term $\xi$
as in Theorem A and 
a sequence of spherical harmonics $\{X_i\}_{i\ge 1}$ with $\deg X_i=i$.
To prove the convergence of the series in \eqref{eq-HHW-intro-expansion-infinite}, 
it is essential to prove that $\{X_i\}_{i\ge 1}$ satisfies the so-called analyticity-type estimates
or, equivalently, $\{X_i\}_{i\ge 1}$ arises from a harmonic function $H$ through $H(x)=\sum_{i\ge 1} |x|^iX_i(\theta)$. 
To this end, we need to construct a sequence of families of functions convergent to $\{c_{i}\}$ in an appropriate sense. 
We point out that the harmonic function $H$, or the sequence of spherical harmonics $\{X_i\}_{i\ge 1}$,  
is uncovered through 
a family of recursive equations satisfied by $\{c_{i}\}$ and 
there does not exhibit any direct relations between $H$ and $u$.

We now study the converse existence problem and aim to construct positive solutions of \eqref{eq-Yamabe-equation} 
in $B_1\setminus\{0\}$ with prescribed asymptotic behaviors at the origin.  
In a broader context, our method can be used to construct all solutions with isolated singularity
and parameterize these solutions by harmonic functions. 

\begin{theorem}
\label{thrm-correspondence-harmonic-to-solutions-Yamabe}
Let 
$\xi$ be a smooth positive periodic solution 
to \eqref{eq-ode-for-radial-singular-solution} 
and $H$ be a harmonic function near $0$ with $H(0) = 0$. 
Then, there exists a sequence of bounded smooth functions $\{c_{i}\}_{i\ge 1}$ in $B_{\epsilon}\setminus\{0\}$ 
determined uniquely by $\xi$ and $H$, 
for some $\epsilon\in (0,1)$, such that $u$ given by \eqref{eq-HHW-intro-expansion-infinite}
is a solution of \eqref{eq-Yamabe-equation} in $B_{\epsilon}\setminus\{0\}$, 
where the series converges uniformly and absolutely in $B_{\epsilon}\setminus\{0\}$.
\end{theorem}


Theorem \ref{thrm-correspondence-harmonic-to-solutions-Yamabe} concerns the existence of solutions 
of \eqref{eq-Yamabe-equation} with prescribed asymptotics. 
The next result addresses the uniqueness. 

\begin{theorem}
\label{thrm-unique-expansion}
Let $u$ and $\widetilde u$ be solutions to \eqref{eq-Yamabe-equation} on $B_1 \setminus \{0\}$ 
with nonremovable singularity at $0$. 
Assume that for any positive constant $\mu$, there exists a positive constant $C_{\mu}$ such that, for any $x \in  B_1\setminus\{0\}$,  
\begin{align*}
|u(x)-\widetilde u(x)| \leq C_{\mu} |x|^{\mu}.\end{align*} 
Then, $u = \widetilde u$ on $B_1 \setminus \{0\}$.
\end{theorem}

Theorem \ref{thrm-unique-expansion} can be interpreted as the unique continuation at the isolated singularity.
A similar result was established by Mazzeo \cite{Mazzeo1991UniqueCA} 
for Laplacian operators on asymptotic hyperbolic manifolds using Carleman's inequality.
The operator related to the equation \eqref{eq-Yamabe-equation} is the Euclidean Laplacian in cylindrical coordinates,
and thus falls outside the framework in \cite{Mazzeo1991UniqueCA}.
In fact, a straightforward computation shows that if 
$u$ and $\widetilde u$ are solutions to \eqref{eq-Yamabe-equation} and
$w = |x|^{(n-2)/2} (u - \widetilde u)$, then $w$ satisfies  
\begin{align}
\label{eq-linear-eq-unique-conti-Euclidean}
\Delta w - \frac{1}{|x|} \frac{(n-2)x_i}{|x|} \partial_i w 
+ \frac{c(x)}{|x|^2} w= \frac{B(x)}{|x|^{2}} w,
\end{align}
where $c$ 
is a bounded function in $B_1\setminus\{0\}$  
and $B = O(|x|)$ as $|x|\rightarrow 0$.
The goal is to prove the unique continuation property at $0$.
Due to the singular coefficient ${(n-2)x_i}/{|x|^2}$ of the first order derivatives, 
such a unique continuation property is not covered by results in 
\cite{GarofaloLin2}, \cite{Hormander}, \cite{JerisonKenig}, and \cite{Meshkov}.

Theorems \ref{thrm-HHW-intro-expansion}, \ref{thrm-correspondence-harmonic-to-solutions-Yamabe}, 
and \ref{thrm-unique-expansion} establish a one-to-one and onto correspondence between 
the collection of the positive solutions $u$ of \eqref{eq-Yamabe-equation} in $B_1\setminus\{0\}$
with a nonremovable singularity at $x=0$ and the collection of harmonic functions $H$ with $H(0)=0$.  
We can study the regularity of the map $H\mapsto u$ in appropriate function spaces. 

We now briefly discuss the proof of
Theorem \ref{thrm-correspondence-harmonic-to-solutions-Yamabe}. 
We consider $X_1=0$ and $X_1\neq 0$ separately.
Refer to Theorem \ref{thrm-convergence-of-the-expansion-from-optimal-term}
and Theorem \ref{thrm-convergence-of-the-expansion-from-optimal-term-improved}
for a more detailed version for \eqref{eq-Yamabe-equation} in cylindrical coordinates.
Let $\{X_i\}_{i\ge 1}$ be the sequence of spherical harmonics as given in 
Theorem \ref{thrm-correspondence-harmonic-to-solutions-Yamabe}
and $\{c_{\ell}\}$ be the sequence of functions satisfying the recursive equations as described before. 
Note that $\{c_{\ell}\}$ is determined uniquely by $\xi$ and $\{X_i\}_{i\ge 1}$.
We now form an infinite series 
\begin{align}\label{eq-intro-infinite-series}\sum_{\ell=1}^\infty c_{\ell}(x)|x|^{\mu_\ell},\end{align}
and aim to study whether this series converges. 

With the given sequence of spherical harmonics $\{X_i\}_{i\ge 1}$ 
in Theorem \ref{thrm-correspondence-harmonic-to-solutions-Yamabe}, 
each $c_\ell$ satisfies a linear equation involving $\xi$ and prior $c_1, \cdots, c_{\ell-1}$. 
We will derive a recursive estimate of $c_\ell$ 
in terms of $c_1, \cdots, c_{\ell-1}$ in an appropriate Sobolev norm and then investigate whether such 
recursive estimates permit us to prove the convergence of the series in \eqref{eq-intro-infinite-series}. 
In the proof of recursive estimates, we encounter 
the so-called {\it small divisors}. 
The difference of terms from two related sequences appear in the denominators in the recursive estimates. 
One sequence is the index set $\{\mu_\ell\}_{\ell\ge 1}$ 
as in Theorem \ref{thrm-HHW-intro-expansion}, 
and another sequence $\{\rho_j\}_{j\ge 1}$ consists of the exponential decay orders 
of the homogeneous linearization of the Yamabe equation 
\eqref{eq-Yamabe-equation} at the leading term $\xi$. 
To avoid small divisors, we need to prove the positivity of a
certain infimum, specifically,  
\begin{align}\label{eq-def-delta}
\inf\big\{|\mu_\ell-\rho_j|;\, \mu_\ell\neq \rho_j\text{ for }\ell, j\ge 1\big\}>0.\end{align}

The difficulty to prove the positivity of such an infimum is obvious. 
In the set where the infimum is taken in \eqref{eq-def-delta}, 
$\rho_j$ and $\mu_\ell$ vary in two different sequences.   
Although we exclude $\mu_\ell= \rho_j$ in the set, it is possible that 
$\mu_\ell$ can be arbitrarily close to some $\rho_j$.   
Moreover, there are no explicit expressions of $\rho_j$ if $\xi$ is a positive nonconstant periodic solution to 
\eqref{eq-ode-for-radial-singular-solution}. 
Unfortunately, the infimum in \eqref{eq-def-delta} can be zero. 
In fact, the zero infimum occurs frequently for many dimensions, 
as Example \ref{exa-small-divisors} demonstrates. 
In \eqref{eq-def-delta}, we include all possible choices of $\mu_\ell$ and $\rho_j$. 
It is more practical to attempt to prove the positivity of the infimum if it is taken in a smaller set.

The proof of the positivity of an appropriately restricted infimum constitutes the most technical preparation in this paper. 
To achieve this, we employ  
multi-indices with countably many components to reformulate the series 
\eqref{eq-intro-infinite-series}. 
Specifically, denote by $e_i \in \Z_+^\infty$ the integer sequence 
whose $i$-th component is $1$ and all other components are $0$, and then define
$$\ind = \Z_+\text{-}\text{span}\{ e_i;\,i\geq 1\} \setminus\{0\}.$$
In other words, $\ind$ consists of all nonzero sequences $\alpha = \{\alpha_i\}_{i\geq 1}$ 
such that $\alpha_i \in \Z_+$ for any $i\geq 1$ and $\alpha_i = 0$ for all but finitely many $i$. 
With such multi-indices, we can carry out similar computations 
in the multi-index setting to reformulate the series in \eqref{eq-intro-infinite-series} as 
\begin{align}\label{eq-intro-infinite-series-reformulated-multi}
\sum_{\alpha\in \ind} c_\alpha(x)|x|^{\mu_\alpha},\end{align}
where $\{\mu_\alpha\}_{\alpha\in \ind}$ is the index set $\{\mu_\ell\}_{\ell\ge 1}$ 
reparameterized by $\ind$. 
Refer to Section \ref{sec-Reparametrizations} for details. 

At first glance, it seems much more complicated to parameterize the series 
by infinite dimensional multi-indices as in \eqref{eq-intro-infinite-series-reformulated-multi}
than by positive integers as in \eqref{eq-intro-infinite-series}. 
The fact is that such a reparameterization has significant advantages. 
Among them, there is a much clearer description 
of $c_\alpha$ 
in the multi-index setting. 
In particular, when we expand $c_\alpha$ in terms of spherical harmonics, 
we are able to characterize the highest degree of spherical harmonics in $c_\alpha$ 
in terms of $\alpha$ explicitly. This is hard to achieve in the integer index setting. 
Furthermore, the corresponding set as in \eqref{eq-def-delta} where the infimum is taken 
has nice structures. In the infinite dimensional multi-index setting, we are able to prove 
that such an infimum is indeed positive under 
an additional assumption that the degree-1 spherical harmonic is absent from the 
collection $\{X_i\}_{i\ge 1}$; namely, $X_1=0$.  
The small divisors never appear in our recursive estimates. 
Refer to Lemma \ref{lemma-comparison-of-mu-and-rho} for details.

We emphasize that we do not impose any assumptions 
on the collection of the exponential decay orders $\{\rho_j\}_{j\ge 1}$ such as the non-resonance condition. 
Any assumptions on $\{\rho_j\}_{j\ge 1}$ are probably hard to verify 
partly due to the lack of explicit expressions of $\rho_j$. 

With Theorem \ref{thrm-correspondence-harmonic-to-solutions-Yamabe} proved under the additional condition $X_1=0$, 
we next consider the general case. 
We will employ a class of conformal transforms to eliminate such a condition. 
More precisely, with appropriate conformal transforms, we reduce the general case 
to the case with the vanishing condition. 
The conformal structure of the equation \eqref{eq-Yamabe-equation} plays an essential role. 
Such conformal transforms were applied to radial solutions $\xi$ of \eqref{eq-ode-for-radial-singular-solution} 
in the derivation of bases of the projection of the linearized equation onto the first eigenspace, 
as in \cite{Korevaar1998RefinedAF}. 
To prove Theorem \ref{thrm-correspondence-harmonic-to-solutions-Yamabe} in the general case, 
we apply these conformal transforms to arbitrary solutions and also to formal solutions.



Estimates as in Theorems A, B, and C were derived for other equations. 
For example, \eqref{eq-Yamabe-equation} has fully nonlinear counterparts, 
the $\sigma_k$-Yamabe equations for $2 \leq k \leq n$. 
For these equations, 
Z.-C. Han, Y.-Y. Li, and Teixeira \cite{HanLi2010sigma_k} proved lower order expansions 
and Q. Han, X. Li, and Y. Li  \cite{Han2019AsymptoticEO} established expansions 
up to arbitrary orders. 
Refer to \cite{GuoLiWan2020CPDE, GuoLiuWan2022JGA, GuoWanYang2021CalVar} 
for similar results.
We will focus on \eqref{eq-Yamabe-equation} in this paper. 

A related problem is the asymptotic analysis of geometric equations in exterior domains.  
Refer to \cite{BL, CaLi2003, FeMaMi1999, Han-Wang2024, HongYuan2020, LLY2020, Schoen1983}, 
for example. 

The paper is organized as follows. 
In Section \ref{sec-admissible-bases}, 
we rewrite \eqref{eq-Yamabe-equation} in cylindrical coordinates and study a family of linear ODEs. 
In Section \ref{sec-Linear-Equations}, we characterize exponential decay solutions 
to the linearized equation. 
In Section \ref{sec-Asymp-Expansions}, we provide an alternative proof of Theorem C 
using our characterization in Section \ref{sec-Linear-Equations}. 
In Section \ref{sec-unique-conti}, we establish a version of the unique continuation 
and prove Theorem \ref{thrm-unique-expansion}. 
In Section \ref{sec-Reparametrizations}, we reparameterize the index set and 
formal power series by multi-indices of infinite dimension
and prove several key technical estimates. 
In Section \ref{sec-Convergence-of-the-Expansions}, 
we prove the convergence of the formal expansions \eqref{eq-intro-infinite-series-reformulated-multi}
to solutions of \eqref{eq-Yamabe-equation}
under an appropriate vanishing condition. 
In Section \ref{sec-Conformal-Transforms},  
we employ a class of conformal transforms to eliminate the vanishing condition 
and prove Theorems \ref{thrm-HHW-intro-expansion} and \ref{thrm-correspondence-harmonic-to-solutions-Yamabe}.

We now summarize several key results leading to the proof 
of Theorems \ref{thrm-HHW-intro-expansion} and \ref{thrm-correspondence-harmonic-to-solutions-Yamabe}. 
Proposition \ref{prop:SA} concerns the strict subadditivity of the exponential decay orders. 
Lemma \ref{lemma-comparison-of-mu-and-rho}, based on Proposition \ref{prop:SA}, 
concerns the positivity of the infimum of a specific set in terms of the infinite dimensional multi-indices
and allows us to rule out small divisors and obtain effective estimates for recursive equations as in 
Lemma \ref{lemma-estimate-L_mu_alpha^-1}. 
Lemma \ref{lemma-<alpha,rho>=rho_i} concerns the finiteness of the number of roots to 
an algebraic equation involving exponential decay orders 
and allows us to reduce our study to a special case with the help of appropriately constructed conformal transforms. 
Lemmas \ref{lemma-comparison-of-mu-and-rho} and \ref{lemma-<alpha,rho>=rho_i}
constitute the most important technical preparations in this paper. 
Their validity demonstrates that it is essential to introduce 
the infinite dimensional multi-indices.

\section{Admissible Bases and Exponential Decay Orders}\label{sec-admissible-bases}

In this section, we formulate the Yamabe equation in cylindrical coordinates and 
study a sequence of linear ordinary differential equations associated 
with its linearization. 
We prove the crucial strict subadditivity of the exponential decay orders.
Proposition \ref{prop:SA} constitutes a preliminary technical preparation in this paper. 

Following \cite{Caffarelli1989AsymptoticSA}, \cite{Korevaar1998RefinedAF}, 
and \cite{Han2019AsymptoticEO},
we introduce cylindrical coordinates $(t,\theta) \in [0,\infty) \times \mathbb{S}^{n-1}$ given by
$t = -\ln |x|$ and $\theta = {x}/{|x|}.$
Let $u$ be a positive solution of \eqref{eq-Yamabe-equation} in $B_1\setminus \{0\}$ 
with a nonremovable singularity at the origin and $\xi$
be a smooth positive periodic solution of \eqref{eq-ode-for-radial-singular-solution} 
such that Theorem A holds. 
Set $v(t,\theta) = |x|^{\frac{n-2}{2}}u(x).$
A straightforward calculation transforms \eqref{eq-Yamabe-equation} to
\begin{align}
\label{eq-Yamabe-cylindrical}
\partial_{tt} v + \Delta_\theta v - \frac{1}{4} (n-2)^2 v + \frac{1}{4} n (n-2) v^{\frac{n+2}{n-2}} = 0.
\end{align}
Define $\varphi(t,\theta) = v(t,\theta) - \xi(t).$
It is easy to verify that $\varphi$ satisfies
\begin{align}
\label{eq-linearized-Yamabe-at-xi}
    \partial_{tt} \varphi + \Delta_\theta \varphi - \frac{1}{4}(n-2)^2 \varphi 
    + \frac{1}{4} n(n+2)\xi^{\frac{4}{n-2}} \varphi = F(\varphi),
\end{align}
where $F$ is given by 
\begin{align}
\label{eq-RHS-of-linearized-Yamabe}
    F(s) = -\frac{1}{4} n (n-2) \Big[ (\xi + s)^{\frac{n+2}{n-2}} - \xi^{\frac{n+2}{n-2}} 
    - \frac{n+2}{n-2} \xi^{\frac{4}{n-2}} s \Big].
\end{align}

We now fix some notations. 
Let $\{\lambda_i \}_{i\ge 0}$ be the sequence of distinct nonnegative eigenvalues 
of $-\Delta_\theta$, i.e., $\lambda_i= i(i+n-2)$ for $i\ge 0$.
Note that we do not count multiplicity. 
In addition, we set, for $i\ge 0$, 
\begin{align*}
\text{$\mathscr{E}_i=$ the eigenspace of $-\Delta_\theta$ 
corresponding to $\lambda_i$.}
\end{align*}
In other words, 
$\mathscr{E}_i$ is the collection of spherical harmonics of degree $i$.
We recall a well-known fact for spherical harmonics. 
Refer to Chapter 5 in \cite{HarmonicFcnTheory}.

\begin{lemma}
\label{lemma-decomposition-of-spherical-harmonics}
Let $X_{k}\in \mathscr{E}_k$ and $X_{l}\in \mathscr{E}_l$. 
Then,
$$X_{k} X_{l}=\sum_{0\le j\le (k+l)/2} Y_{k+l-2j},$$
for some $Y_{k+l-2j}\in \mathscr{E}_{k+l-2j}$, for $0\le j\le (k+l)/2$.
\end{lemma}


All positive periodic solutions of \eqref{eq-ode-for-radial-singular-solution} 
form a one-parameter family of functions,
up to translations in $t$,
and there is a unique positive constant solution given by 
\begin{align}\label{eq-positive-constant-solu} 
\xi(t)\equiv \Big(\frac{n-2}{n}\Big)^{\frac{n-2}4}.\end{align}
Throughout the paper, we fix a smooth positive periodic solution $\xi$ 
to \eqref{eq-ode-for-radial-singular-solution} on $\R$
and all periodic functions are assumed to have the same period as $\xi$. Note that $0 <\xi < 1$. 
For simplicity, we write
\begin{align}
\label{eq-linear-operator-of-Yamabe}
   \mathcal L = \partial_{tt} + \Delta_\theta - \frac{1}{4}(n-2)^2 + \frac{1}{4}n(n+2)\xi^{\frac{4}{n-2}}.
\end{align}
Next, we project $\mathcal L$ onto the eigenspaces $\mathscr E_i$. Specifically, for each $i \geq 0$, 
any $X_i \in \mathscr{E}_i$, and any $\psi = \psi(t) \in C^2(\R)$, define $L_i$ by 
$\mathcal L(\psi X_i) = (L_i\psi) X_i.$
By \eqref{eq-linear-operator-of-Yamabe}, we have
\begin{align}
\label{eq-ode-projection-operator}
    L_i \psi = \psi'' + \Big[ \frac{1}{4}n(n+2)\xi^\frac{4}{n-2} - \Big(i+\frac{1}{2}(n-2)\Big)^2\Big] \psi.
\end{align}
Our main goal in this section is to construct an appropriate basis for $\mathrm{Ker}(L_i)$ for each $i\ge 0$ 
and analyze the associated order of exponential growth and decay. 

If $\xi$ is the positive constant solution as in \eqref{eq-positive-constant-solu}, 
then $L_{i}$ in \eqref{eq-ode-projection-operator} reduces to
\begin{align*} 
L_{i} \psi=\psi^{\prime \prime}-\big(i^2+(n-2)(i-1)\big) \psi.
\end{align*}
It is easy to check that $\operatorname{Ker}(L_{i})$ has a basis given by 
$\psi_{0}^{+}=\cos (\sqrt{n-2} t)$,  $\psi_{0}^{-}=\sin (\sqrt{n-2} t)$ for $i=0$, 
and by $\psi_{i}^{+}=e^{-\rho_{i} t}$, $ \psi_{i}^{-}=e^{\rho_{i} t}$ for $i \geq 1$, 
where
\begin{align}\label{eq-i-Yamabe-linearization-constant-decay-rate}
\rho_{i}=\big(i^2+(i-1)(n-2)\big)^{1/2}.
\end{align}
However, if $\xi$ is a positive nonconstant periodic solution, 
it is much harder to study $\mathrm{Ker}(L_i)$, except for $i=0$ or 1. 
In fact, a basis of $\mathrm{Ker}(L_i)$ can be constructed explicitly for $i=0$ and 1 
but only abstractly by the Floquet theory for $i\ge 2$. 
Refer to \cite{Han2019AsymptoticEO}, \cite{Korevaar1998RefinedAF}, \cite{Mazzeo1996ConstantSC}, 
and \cite{Mazzeo1994ModuliSO} for details, 
or to \cite{HanLi2010sigma_k} for a more general setting. 

In the following, we construct a basis of $\mathrm{Ker}(L_i)$ by a direct method, 
that enables us to prove a crucial strict subadditivity of the exponential orders as in Proposition \ref{prop:SA}.  

Let $\xirad=\xirad(t)$ be a positive nonconstant $T$-periodic solution of \eqref{eq-ode-for-radial-singular-solution}. 
For brevity, set
\begin{equation*}
  a=\frac12(n-2),\quad 
  V=\frac14n(n+2)\xirad^{\frac4{n-2}}.
\end{equation*}
Then, $V$ is a smooth, positive, and $T$-periodic function on $\mathbb R$.  
In the following, we write 
 $$\Lambda_1=(a+1)^2.$$ 
Introduce a 
spectral parameter $\Lambda\ge \Lambda_1$ and consider
\begin{equation}\label{eq:continuous-lambda-equation}
  \Psi''-(\Lambda-V)\Psi=0\quad\text{on }\mathbb R.
\end{equation}

\begin{lemma}\label{lem:smooth-positive-lambda-new}
For any $\Lambda\ge \Lambda_1$, \eqref{eq:continuous-lambda-equation} has a basis of solutions 
given by 
\begin{equation}\label{eq:continuous-floquet} 
  \Psi^+_\Lambda(t)=e^{-\rho(\Lambda)t}P^+_\Lambda(t), 
  \quad \Psi^-_\Lambda(t)=e^{\rho(\Lambda)t}P^-_\Lambda(t),
\end{equation}
for some positive constant $\rho(\Lambda)$ 
and some smooth, positive, and $T$-periodic functions $P^\pm_\Lambda$ on $\mathbb R$.
Moreover, $\Lambda\mapsto \rho(\Lambda)$
is smooth and strictly increasing on $(\Lambda_1,\infty)$, with $\rho(\Lambda_1)=1$.
\end{lemma}



\begin{proof} The proof consists of four steps. 

{\it Step 1.} We prove the desired result for $\Lambda=\Lambda_1$, with $\rho(\Lambda_1)=1$
and
\begin{equation}\label{eq:floquet-p-pm-i=1}
P_{\Lambda_1}^+=a\xirad-\xirad',\quad 
P_{\Lambda_1}^-=a\xirad+\xirad'.
\end{equation} 
The existence was proved in \cite{Mazzeo1994ModuliSO}. 
In fact, with \eqref{eq:floquet-p-pm-i=1} given, 
we can verify directly that $\Psi_{\Lambda_1}^+$ and $\Psi_{\Lambda_1}^-$
are solutions of \eqref{eq:continuous-lambda-equation} with $\Lambda=\Lambda_1$. 
It is obvious that $\Psi_{\Lambda_1}^+$ and $\Psi_{\Lambda_1}^-$ are linearly independent. 
We now prove that $P_{\Lambda_1}^+$ and $P_{\Lambda_1}^-$  are positive on $\mathbb R$.

Recall that the equation \eqref{eq-ode-for-radial-singular-solution} has a conserved energy given by 
\begin{equation*} 
E=(\xirad')^2+a^2\big(\xirad^{\frac{2n}{n-2}}-\xirad^2\big).
\end{equation*}
Set $m=\min\xirad>0$.  At a minimum point, $\xirad'=0$ and $\xirad''\ge0$.  
By \eqref{eq-ode-for-radial-singular-solution}, 
we get 
$m^{\frac{4}{n-2}}\le (n-2)/n$, and hence $E<0$ by a simple computation. 
Therefore, $(\xirad')^2 <a^2\xirad^2$ on $\mathbb R$ 
and consequently $P_{\Lambda_1}^\pm>0$ on $\mathbb R$.


{\it Step 2.}
 Let $\Lambda$ be a constant. We consider the equation 
\begin{equation}\label{eq:continuous-riccati}
  r_\Lambda'=r_\Lambda^2-\Lambda+V.
\end{equation}
We claim that for any $\Lambda\ge \Lambda_1$, there exists a unique $T$-periodic solution $r_\Lambda$ of 
\eqref{eq:continuous-riccati} with a positive average $\rho(\Lambda)$ over $[0,T]$; namely 
\begin{equation}\label{eq:continuous-rho-def}
  \rho(\Lambda)=\frac1T\int_0^T r_\Lam\dd t.
\end{equation}
  
At $\Lambda=\Lambda_1$, Step 1 yields \eqref{eq:continuous-floquet} for $\Lambda=\Lambda_1$, 
with $\rho(\Lambda_1)=1$ and 
the positive $T$-periodic functions $P_{\Lambda_1}^\pm$ given by \eqref{eq:floquet-p-pm-i=1}.
Define
\[
  r_{\Lambda_1}=-\frac{(\Psi_{\Lambda_1}^{+})'}{\Psi^+_{\Lambda_1}}
  =\rho(\Lambda_1)-\frac{(P_{\Lambda_1}^{+})'}{P_{\Lambda_1}^+}.
\]
Then, $r_{\Lambda_1}$ is $T$-periodic and satisfies \eqref{eq:continuous-riccati} 
and \eqref{eq:continuous-rho-def} with $\Lambda=\Lambda_1$. 

For any constant $\Lambda$ and smooth $T$-periodic function $r$, set  
\[
  \mathcal F(\Lambda,r)=r'-r^2+\Lambda-V. 
\]
Here, we consider $\mathcal F$ as a map
$\mathcal F=\mathcal F(\Lambda,r): \R\times C^{1,\alpha}(\mathbb{R} / T \mathbb{Z})
  \mapsto
  C^{0, \alpha}(\mathbb{R} / T \mathbb{Z}).$
Suppose $r_\Lam$ is a $T$-periodic solution of $\mathcal F(\Lambda,r_\Lam)=0$. 
The linearization of $\mathcal F$ in $r_\Lam$ is given by 
\[
  D_r\mathcal F(\Lambda,r_\Lam)w=w'-2r_\Lam w.
\]
If $\rho(\Lam)>0$, we then prove that 
$D_r\mathcal F(\Lambda,r_\Lam): C^{1, \alpha}(\mathbb{R} / T \mathbb{Z}) 
\rightarrow C^{0, \alpha}(\mathbb{R} / T \mathbb{Z})$ 
is a Banach-space isomorphism.
Indeed, the homogeneous equation
$w'-2r_\Lam w=0$
yields $w(t+T)=e^{2T\rho(\Lam)}w(t).$
Since $e^{2T\rho(\Lam)}\ne1$, the only $T$-periodic homogeneous solution is zero.  
For an inhomogeneous equation, the variation of constants gives the unique $T$-periodic solution explicitly.  
Thus, the implicit function theorem for the Banach space applies 
at every $T$-periodic solution $r_\Lam$ of $\mathcal F(\Lambda,r_\Lam)=0$ 
with a positive average $\rho(\Lam)$.

We now record a useful formula. For a $T$-periodic function $f$, consider 
\begin{equation}\label{eq:w-general-ODE-branch}
  w'-2r_\Lambda w=f.
\end{equation}
If $\rho(\Lambda)>0$, the unique $T$-periodic solution $w$ is given by 
\begin{equation}\label{eq:w-positive-formula}
  w(t)=
  -\frac{1}{1-\exp\{-2T\rho(\Lambda)\}}
  {\displaystyle \int_t^{t+T}
  \exp\!\Big\{-\int_t^s2r_\Lambda(\tau)\dd\tau\Big\}f(s)\dd s}.
\end{equation}
If $f<0$ on $\mathbb R$, then $w>0$ on $\mathbb R$.

Starting from 
$\Lambda=\Lambda_1$, by the implicit function theorem, 
we know that for some $\varepsilon>0$, there exists a unique $C^{\infty}$ map
\begin{align*}
\Lambda\in  (\Lambda_1-\varepsilon, \Lambda_1+\varepsilon)  
\mapsto r_{\Lambda} \in C^{1, \alpha}(\mathbb{R} / T \mathbb{Z}), 
\end{align*}
such that
$\mathcal{F}\left(\Lambda, r_{\Lambda}\right)=0,$
or equivalently, \eqref{eq:continuous-riccati} holds. 
We should emphasize that $r_\Lambda \in C^{1, \alpha}(\mathbb{R} / T \mathbb{Z})$
depends on $\Lambda$ smoothly. 
For $\Lambda\in[\Lambda_1, \Lambda_1+\varepsilon)$, set $u=\partial_\Lambda r_\Lambda.$
Since $r_\Lambda$ depends on $\Lambda$ smoothly, 
we conclude that $\rho(\Lambda)$ is a smooth function of $\Lambda$
and that $u$ is a well-defined function in $C^{1, \alpha}(\mathbb{R} / T \mathbb{Z})$.
Differentiating \eqref{eq:continuous-riccati} with respect to $\Lambda$ 
shows that $u$ satisfies \eqref{eq:w-general-ODE-branch} with $f=-1<0$. 
A similar expression for 
$u$ as in \eqref{eq:w-positive-formula} shows that $u>0$ on $\mathbb R$. 
With \eqref{eq:continuous-rho-def}, 
we get $\rho'(\Lambda)=\frac1T\int_0^T u\dd t>0$.
Hence, $\rho(\Lam)$ is strictly increasing as a function of $\Lam$ and
$\rho(\Lam)\ge 1$ by $\rho(\Lambda_1)=1$, along any continuation of the interval $[\Lambda_1, L)$.

It remains to check that the continuation can extend to $L=\infty$.  
Fix a compact interval $[\Lambda_1,L_0]$ along the extension and take any $\Lambda\in[\Lambda_1,L_0]$.  
At any critical point of $r_\Lambda$, 
the equation \eqref{eq:continuous-riccati} yields
$r_\Lambda^2=\Lambda-V\le L_0$,
and hence, $\|r_\Lambda\|_{L^\infty([0,T])}\le \sqrt L_0.$
The equation \eqref{eq:continuous-riccati} then provides a uniform bound of $r_\Lambda'$.  
Thus, if a maximal continuation interval had a finite endpoint $\Lambda_*$, 
a sequence $\Lambda_j\to\Lambda_*$ would have a subsequence 
for which $r_{\Lambda_j}$ converges in $C^1$ to a $T$-periodic solution $r_{\Lambda_*}$
of \eqref{eq:continuous-riccati} with $\Lambda=\Lambda_*$.  
With $\rho(\Lambda_j)\to \rho(\Lambda_*)$, we have $\rho(\Lambda_*)\ge 1$. 
The linearization at $(\Lambda_{*},r_{\Lambda_*})$ is therefore invertible, 
so the implicit function theorem extends the interval past $\Lambda_*$, leading to a contradiction.  
Hence, the extension exists for every $\Lambda\ge\Lambda_1$.
We conclude that for any $\Lambda\ge \Lambda_1$ there exists a unique $T$-periodic solution $r_\Lambda$ of 
\eqref{eq:continuous-riccati} with a positive average $\rho(\Lambda)$ over $[0,T]$.
Moreover, $\rho(\Lambda)$ is smooth and strictly increasing on $[\Lam_1, \infty)$.

{\it Step 3.} 
For any $\Lambda>\Lambda_1$,  define
\[
  \Psi^+_\Lambda(t)=\exp\!\Big\{-\int_0^t r_\Lambda(s)\dd s\Big\}>0.
\]
Then, 
$r_\Lambda=-(\Psi^{+}_{\Lambda})'/\Psi^+_\Lambda$,
and \eqref{eq:continuous-riccati} yields 
that $\Psi^{+}_{\Lambda}$ is a solution of  \eqref{eq:continuous-lambda-equation}.  
Since $r_\Lambda$ is $T$-periodic, then
\begin{align}\label{eq-periodic-condition-Psi}
  \Psi^+_\Lambda(t+T)=e^{-T\rho(\Lambda)}\Psi^+_\Lambda(t),
\end{align}
and hence, the function $P^+_\Lambda(t)=e^{\rho(\Lambda)t}\Psi^+_\Lambda(t)$
is positive, $T$-periodic, and normalized by $P^+_\Lambda(0)=1$. 
In conclusion, we have a desired solution
$\Psi_\Lambda^+(t)
=
e^{-\rho(\Lambda)t}P_\Lambda^+(t).
$

{\it Step 4.} 
Define, for any $t\in\mathbb R$,
\begin{equation}\label{eq:integral I lam}
I_\Lambda(t)
=\int_{-\infty}^{t}
\frac{ds}{(\Psi_\Lambda^+(s))^2}.
\end{equation}
The integral in \eqref{eq:integral I lam} converges since
$\Psi_\Lambda^+(s)\geq c_\Lambda e^{-\rho(\Lambda)s}$
and $\rho(\Lambda)>0$. 
Set
\begin{equation}\label{eq:psi lam-}
\Psi_\Lambda^-
=
\Psi_\Lambda^+I_\Lambda.
\end{equation}
A direct computation shows that 
$\Psi_\Lambda^-$ is also a solution of \eqref{eq:continuous-lambda-equation}. 
By \eqref{eq:integral I lam}, \eqref{eq:psi lam-},
and a simple computation, we conclude that 
the Wronskian $W$ of $\Psi_{\Lambda}^{+}$ and $\Psi_{\Lambda}^{-}$ satisfies
\begin{equation*}
W(\Psi_{\Lambda}^{+}, \Psi_{\Lambda}^{-}) 
 =\Psi_{\Lambda}^{+}(\Psi_{\Lambda}^{-})^{\prime}
-(\Psi_{\Lambda}^{+})^{\prime} \Psi_{\Lambda}^{-} =1. 
\end{equation*}
Then, $\Psi_{\Lambda}^{+}$ and $\Psi_{\Lambda}^{-}$ are linearly independent and form a basis 
of solutions of \eqref{eq:continuous-lambda-equation}. By \eqref{eq:psi lam-}, we have
$\Psi_\Lambda^-(t)>0.$

By \eqref{eq-periodic-condition-Psi}, 
we get $I_\Lambda(t+T)=e^{2\rho(\Lambda)T}I_\Lambda(t)$
and consequently $\Psi_\Lambda^-(t+T)=e^{\rho(\Lambda)T}\Psi_\Lambda^-(t)$.
Therefore, the function $P^-_\Lambda(t)=e^{-\rho(\Lambda)t}\Psi^-_\Lambda(t)$
is positive and $T$-periodic. 
In conclusion, we have a desired solution
$\Psi_\Lambda^-(t)
=
e^{\rho(\Lambda)t}P_\Lambda^-(t)$. 
\end{proof}

In terms of $\rho(\Lambda)$ and $P_\Lambda^+$, we can write 
\begin{equation}\label{eq:def of r Lam in mono}
r_\Lambda
=
-\frac{(\Psi^{+}_{\Lambda})'}{\Psi^+_\Lambda}
=
\rho(\Lambda)-\frac{(P^{+}_{\Lambda})'}{P^+_\Lambda}.
\end{equation}
The characterization of $\rho(\Lambda)$ and $r_\Lambda$ given by \eqref{eq:continuous-rho-def} 
and \eqref{eq:def of r Lam in mono}, respectively, 
will play an important role in our subsequent study.

In the next result, we derive upper and lower bounds of $\rho(\Lambda)$. 

\begin{lemma}
\label{lem:bounds-rho}
Let $\rho(\Lambda)$ be given in Lemma \ref{lem:smooth-positive-lambda-new}. Then, 
\begin{equation}\label{eq:rho-less-sqrt-Lambda}
\rho(\Lambda)\le \sqrt{\Lambda-V_{\min}}\quad\text{for any }\Lambda>\Lambda_1,
\end{equation}
and 
\begin{equation}\label{eq:rho-greater-sqrt-Lambda}
\rho(\Lambda)\geq\sqrt{\Lambda-V_{\max}}\quad\text{for any }\Lambda>V_{\max},
\end{equation}
where $V_{\min}=\min_{[0,T]}V$ and $V_{\max}=\max_{[0,T]}V$.
Moreover, $\rho(\Lam)\to\infty$ as $\Lam\to\infty$.
\end{lemma}

\begin{proof} Let $r_\Lambda$ be the $T$-periodic solution of 
\eqref{eq:continuous-riccati} and $\rho(\Lambda)$ be characterized by \eqref{eq:continuous-rho-def}. 
Averaging \eqref{eq:continuous-riccati} over $[0,T]$ yields
\[
\frac1T\int_0^T r_\Lambda^2\,dt
=
\Lambda-\overline V,
\]
where $\overline V$ is the average of $V$ on $[0,T]$, which is positive. 
Then, the Cauchy inequality implies 
\[
\rho(\Lambda)^2
=
\Big(\frac1T\int_0^T r_\Lambda\,dt\Big)^2
\leq
\frac1T\int_0^T r_\Lambda^2\,dt
=
\Lambda-\overline V
\le\Lambda-V_{\min}.
\]
By $\rho(\Lambda)>0$, we obtain \eqref{eq:rho-less-sqrt-Lambda}. 

We next prove the pointwise estimate
\begin{equation}\label{eq:r-less-sqrt-Lambda}
r_\Lambda(t)<\sqrt{\Lambda}
\qquad\text{for any }t\in\mathbb R.
\end{equation}
Indeed, at any point $t_0$ satisfying
$r_\Lambda(t_0)=\sqrt{\Lambda}$,
the equation \eqref{eq:continuous-riccati} gives
$r_\Lambda'(t_0)=V(t_0)>0$.
On the other hand, the average of $r_\Lambda$ satisfies
$\rho(\Lambda)<\sqrt{\Lambda}$ by \eqref{eq:rho-less-sqrt-Lambda}. 
If $r_\Lambda$ were above $\sqrt{\Lambda}$ somewhere, 
by the periodicity, there would exist a $t_*$ such that $r_\Lambda(t_*)=\sqrt{\Lambda}$ and $r_\Lambda'(t_*)\le 0$.
This is impossible and hence proves \eqref{eq:r-less-sqrt-Lambda}.
We will need \eqref{eq:r-less-sqrt-Lambda} later. 

Take an arbitrary $\Lambda>V_{\max}$. If $r_\Lambda(t_0)=0$ for some $t_0$, then
$r_\Lambda'(t_0)=-\Lambda+V(t_0)<0.$
If $r_\Lambda$ were 0 somewhere, 
by the periodicity, there would exist a $t_*$ such that $r_\Lambda(t_*)=0$ and $r_\Lambda'(t_*)\ge 0$. 
This is impossible.
Therefore,
$r_\Lambda$ has no zero for such $\Lambda$.
Since its average is positive, it follows that
$r_\Lambda>0$ on $\mathbb{R}$.
Let $t_\Lambda$ be a minimum point of $r_\Lambda$. Then,
$r_\Lambda'(t_\Lambda)=0$,
and \eqref{eq:continuous-riccati} yields
\[
r_\Lambda(t_\Lambda)^2
=
\Lambda-V(t_\Lambda)
\geq
\Lambda-V_{\max}.
\]
By $r_\Lambda(t_\Lambda)>0$, we get $r_\Lambda(t_\Lambda)\geq\sqrt{\Lambda-V_{\max}}$
and hence, 
\[
\rho(\Lambda)
=
\frac1T\int_0^T r_\Lambda\,dt
\geq
\min r_\Lambda
\geq
\sqrt{\Lambda-V_{\max}}.
\]
This proves \eqref{eq:rho-greater-sqrt-Lambda}. 
\end{proof}

\begin{lemma}\label{lem:F-concavity}
Let $\rho(\Lambda)$ be given in Lemma \ref{lem:smooth-positive-lambda-new}. 
Then,  
$\rho(\Lambda)^2$ is a concave function of $\Lambda$ for $\Lambda \geq \Lambda_1$.  
\end{lemma}

\begin{proof} Let $r_\Lambda$ be the $T$-periodic solution of 
\eqref{eq:continuous-riccati} and $\rho(\Lambda)$ be characterized by \eqref{eq:continuous-rho-def}. 
Set  $u=\partial_\Lambda r_\Lambda$ and $v=\partial_\Lambda^2 r_\Lambda$.
Differentiating \eqref{eq:continuous-riccati} with respect to $\Lambda$ leads to 
\begin{equation}\label{eq:u-v-eq}
u'-2r_\Lambda u=-1,
\quad  v'-2r_\Lambda v=2u^2.
\end{equation}
Moreover, by \eqref{eq:continuous-rho-def}, we get
\begin{equation}\label{eq:rho-derivatives}
  \rho'(\Lambda)=\frac{1}{T}\int_0^T u \dd t,
  \quad
  \rho''(\Lambda)=\frac{1}{T}\int_0^T v \dd t.
\end{equation}
Here, derivatives of $\rho$ are taken with respect to $\Lambda$.

Let $z$ be the $T$-periodic solution of the adjoint equation
\begin{equation}\label{eq:z-eq}
  -z'-2r_\Lambda z=-1,
\end{equation}
or equivalently,
$z'+2r_\Lambda z=1$.
The solution $z$ is positive.  In fact, $z$ is given by 
\[
  z(t)=
  \frac{1}{\displaystyle\exp\!\big\{2T\rho(\Lambda)\big\}-1}
  {\displaystyle\int_t^{t+T}\exp\!\Big\{\int_t^s2r_\Lambda(\tau)\dd\tau\Big\}\dd s},
\]
where the denominator is positive by $\rho(\Lambda)>0$.

Set $L=\partial_t-2r_\Lambda.$
Then $L^*=-\partial_t-2r_\Lambda$, and \eqref{eq:z-eq} means $L^*z=-1$.  
An integration by parts over one period gives
\begin{equation}\label{eq:avg-y-fz}
  \int_0^T y\dd t=-\int_0^TyL^*z\dd t=-\int_0^T zLy\dd t.
\end{equation}
Applying \eqref{eq:avg-y-fz} to the first equation in \eqref{eq:u-v-eq}, written as $Lu=-1$, yields
\begin{equation*}
  \int_0^T u\dd t=\int_0^T z\dd t.
\end{equation*}
Similarly, applying \eqref{eq:avg-y-fz} to the second equation in \eqref{eq:u-v-eq}, written as
$Lv=2u^2$,
yields
\begin{equation}\label{eq:avg-v}
  \int_0^T v\dd t=-{2}\int_0^T {u^2z}\dd t.
\end{equation}
Finally, divide \eqref{eq:z-eq} 
by $-z$ to get
\[
  \frac{z'}{z}+2r_\Lambda=\frac1z.
\]
Integrating over one period gives
\begin{equation*}
  \int_0^T{\frac1z}\dd t=2T\rho(\Lambda).
\end{equation*} 
By the Cauchy inequality, we get 
\begin{equation}\label{eq:cauchy-u}
 \Big(\int_0^T{u}\dd t\Big)^2
  \le \Big(\int_0^T{u^2z}\dd t\Big)\Big(\int_0^T{\frac1z}\dd t\Big)
  ={2T\rho(\Lambda)}\int_0^T{u^2z}\dd t.
\end{equation}
By \eqref{eq:rho-derivatives} and \eqref{eq:avg-v}, we obtain
\begin{align*}
  (\rho^2)''(\Lambda)
  &=2(\rho'(\Lambda))^2+2\rho(\Lambda)\rho''(\Lambda) \\
  &=\frac{2}{T^2} \Big(\int_0^T{u}\dd t\Big)^2+\frac{2\rho(\Lambda)}{T}\int_0^T{v}\dd t \\
  &=\frac{2}{T^2}\Big(\int_0^T{u}\dd t\Big)^2-\frac{4\rho(\Lambda)}{T}\int_0^T{u^2z}\dd t,
\end{align*}
which is nonpositive by \eqref{eq:cauchy-u}. This proves the desired concavity.
\end{proof}

We now prove a technical consequence. 
Its importance will be apparent 
in the proof of the crucial strict subadditivity  
as in Proposition \ref{prop:SA}.

\begin{cor}\label{cor:base-gap} Let $\rho(\Lambda)$ be given in Lemma \ref{lem:smooth-positive-lambda-new} 
and set $\Lambda_i=(a+i)^2$ for $i\ge 1$. 
Then, $\rho(\Lambda_4)<2\rho(\Lambda_2)$. 
\end{cor}

\begin{proof} Write $F_i=F(\Lambda_i)=\rho(\Lambda_i)^2$ for convenience. 
Then, $F_1=1$ by $\rho(\Lambda_1)=1$.  
The concavity of $F$, proved in Lemma \ref{lem:F-concavity}, yields
\begin{equation*}
  \frac{F_4-F_2}{\Lambda_4-\Lambda_2}
  \le
  \frac{F_2-F_1}{\Lambda_2-\Lambda_1}.
\end{equation*}
Hence,
\begin{equation}\label{eq:F4-bound}
  F_4\le F_2+R(F_2-1),
\end{equation}
where
\begin{align}\label{eq:R-def}
  R=\frac{\Lambda_4-\Lambda_2}{\Lambda_2-\Lambda_1}
  =\frac{(a+4)^2-(a+2)^2}{(a+2)^2-(a+1)^2}
  =\frac{4a+12}{2a+3}.
\end{align}
We now prove $F_4<4F_2$.  

First, suppose $a\ge 3/2$.  Then $R\le3$, and \eqref{eq:F4-bound} yields 
$F_4\le F_2+3(F_2-1)<4F_2$.
Next, suppose $0<a<3/2$.  Then $R>3$.  By \eqref{eq:F4-bound}, we will prove $F_2+R(F_2-1)<4F_2$
or equivalently,
\begin{equation}\label{eq:R-ineq-needed}
  (R-3)F_2<R.
\end{equation}
By 
\eqref{eq:rho-less-sqrt-Lambda}, 
we have $F_2=\rho(\Lambda_2)^2<\Lambda_2=(a+2)^2.$
It suffices to prove
\[
  (R-3)(a+2)^2<R,
\]
which follows from \eqref{eq:R-def}.
Hence \eqref{eq:R-ineq-needed} holds, and again $F_4<4F_2$.
\end{proof}

In the next result, we study properties of $\rho(\Lambda)$ as a function of $\sqrt{\Lambda}$.

\begin{lemma}
\label{lem:rho-properties-sqrt-Lam}
Let $\rho(\Lambda)$ be given in Lemma \ref{lem:smooth-positive-lambda-new}. 
Then, for any $\Lambda>\Lambda_1$,  
\begin{align}\label{eq:der-rho-1&2}
\frac{d}{d\sqrt{\Lambda}}\rho(\Lambda)>1
\quad\text{and}\quad
\frac{d^2}{d(\sqrt{\Lambda})^2}\rho(\Lambda)<0.
\end{align}
Moreover,  $\rho(\Lambda)-\sqrt{\Lambda}$ is strictly increasing 
and $\rho(\Lambda)$ is strictly concave as functions of $\sqrt{\Lambda}$, for $\Lambda>\Lambda_1$. 
In addition, for any $\Lambda>\Lambda_1$, 
\begin{align}\label{eq:rho-greater-sqrt-Lambda-improved} 
\rho(\Lambda)>\sqrt{\Lambda}-a.\end{align}
\end{lemma}

\begin{proof} Let $r_\Lambda$ be the $T$-periodic solution of 
\eqref{eq:continuous-riccati} and $\rho(\Lambda)$ be characterized by \eqref{eq:continuous-rho-def}.
By Lemma \ref{lem:smooth-positive-lambda-new} and \eqref{eq:def of r Lam in mono}, 
$r_\Lambda$ depends smoothly on $\Lambda$ and hence also 
smoothly on $\sqrt{\Lambda}$, for $\Lambda>\Lambda_1$. 

Set
\begin{equation}\label{eq:def of u lam in mono-0}
u_\Lambda=\frac{\partial r_\Lambda}{\partial\sqrt{\Lambda}}, \quad 
v_\Lambda=\frac{\partial^2r_\Lambda}{\partial(\sqrt{\Lambda})^2}.
\end{equation}
By differentiating \eqref{eq:continuous-rho-def} with respect to
$\sqrt{\Lambda}$, we obtain
\begin{equation}\label{eq:expression-derivatives-rho}
\frac{d}{d\sqrt{\Lambda}}\rho(\Lambda)=\frac1T\int_0^T u_\Lambda\,dt, \quad
\frac{d^2}{d(\sqrt{\Lambda})^2}\rho(\Lambda)=\frac1T\int_0^T v_\Lambda(t)\,dt.
\end{equation}
Differentiating \eqref{eq:continuous-riccati} with respect to
$\sqrt{\Lambda}$ gives
\begin{align} \label{eq:u-general-Lambda}
u_\Lambda'-2r_\Lambda u_\Lambda=-2\sqrt{\Lambda},\quad 
v_\Lambda'-2r_\Lambda v_\Lambda=2\bigl(u_\Lambda^2-1\bigr).
\end{align}

A simple rearrangement of the first equation in \eqref{eq:u-general-Lambda} shows that 
that $u_\Lambda-1$ satisfies \eqref{eq:w-general-ODE-branch} 
with $f=-2\bigl(\sqrt{\Lambda}-r_\Lambda\bigr)<0$ by \eqref{eq:r-less-sqrt-Lambda}. 
A similar expression for $u_\Lambda-1$ as in \eqref{eq:w-positive-formula} yields that $u_\Lambda-1>0$ on $\mathbb R$,
or equivalently,
\begin{equation}\label{eq:u-greater-one-general-Lambda}
u_\Lambda>1
\quad\text{on }\mathbb R.
\end{equation}
Combining with the first identity in \eqref{eq:expression-derivatives-rho}, we obtain 
the first inequality in \eqref{eq:der-rho-1&2}. 
By $\sqrt{\Lambda_1}-a=1$, we have \eqref{eq:rho-greater-sqrt-Lambda-improved}
for any $\Lambda>\Lambda_1$. 

Next, the second equation in \eqref{eq:u-general-Lambda} shows that
$v_\Lambda$ satisfies \eqref{eq:w-general-ODE-branch} 
with $f=2\bigl(u_\Lambda^2-1\bigr)>0$ by \eqref{eq:u-greater-one-general-Lambda}. 
A similar expression for $v_\Lambda$ as in \eqref{eq:w-positive-formula} yields that $v_\Lambda<0$ on $\mathbb R$.
Combining with the second identity in \eqref{eq:expression-derivatives-rho}, we obtain 
the second inequality in \eqref{eq:der-rho-1&2}.
\end{proof}


Next, we establish uniform  estimates of solutions 
$\Psi_{\Lambda}^\pm$ of \eqref{eq:continuous-lambda-equation}.
The upper bound of $\rho(\Lam)$ in 
\eqref{eq:rho-less-sqrt-Lambda} 
and the lower bound in \eqref{eq:rho-greater-sqrt-Lambda-improved}
play an important role in the proof.

\begin{lemma}\label{lem:uniform-floquet-factor-bounds}
For any $\Lambda\geq\Lambda_1$, 
let $\Psi_{\Lambda}^\pm$ be solutions of \eqref{eq:continuous-lambda-equation}
in the form in \eqref{eq:continuous-floquet}, 
for $\rho(\Lam)$ as in Lemma \ref{lem:smooth-positive-lambda-new} 
and some smooth, positive, and $T$-periodic functions $P_\Lambda^\pm$, 
satisfying 
\begin{equation}\label{eq:point-normalization-Ppm}
P_\Lambda^\pm(0)=1.
\end{equation}
Then, 
\begin{equation}\label{eq:uniform-Ppm-bounds}
e^{-aT}\leq P_\Lambda^\pm(t)\leq e^{aT}
\quad
\text{for any }t\in\mathbb R.
\end{equation}
Moreover, the Wronskian $W_{\Lambda}=W(\Psi_{\Lambda}^{+}, \Psi_{\Lambda}^{-})$ 
of $\Psi_{\Lambda}^{+}$ and $ \Psi_{\Lambda}^{-}$ satisfies 
\begin{equation}\label{eq:Main result for W Lam}
2e^{-2 a T}\rho(\Lambda) \leq W_{\Lambda} \leq 2e^{2 a T}\rho(\Lambda),
\end{equation}
and, in particular, 
\begin{equation}\label{eq:Coro for W Lam}
    W_\Lambda \geq 2 e^{-2aT}.
\end{equation}
\end{lemma}

\begin{proof} The proof consists of three steps. 
Throughout the proof, we take any $\Lambda\geq\Lambda_1$. 

{\it Step 1. We prove} 
\begin{equation}\label{eq:uniform-floquet-ratio}
\max_{\mathbb R}P_\Lambda^\pm 
\leq
\exp\!\big\{T\bigl(\sqrt{\Lambda}-\rho(\Lambda)\bigr)\big\}\min_{\mathbb R}P_\Lambda^\pm.
\end{equation}
We first consider $P_\Lambda^-$. Set
$m_\Lambda^-=\min_{\mathbb R}P_\Lambda^->0,$
and choose $\tau_\Lambda^-\in\mathbb R$ such that
$P_\Lambda^-(\tau_\Lambda^-)=m_\Lambda^-$.
Then,
$(P_\Lambda^-)'(\tau_\Lambda^-)=0.$
For $s\geq0$, define
\begin{equation}\label{eq:Y decay shifting}
Y_-(s)=\Psi_\Lambda^-(\tau_\Lambda^-+s).
\end{equation}
By
$\Psi_\Lambda^-(t)=e^{\rho(\Lambda)t}P_\Lambda^-(t)$, we have
\begin{equation}\label{eq:Y-growing-initial-data}
Y_-(0)>0,
\quad
Y_-'(0)=\rho(\Lambda)Y_-(0),
\end{equation}
and
\begin{equation*}
Y_-''(s)=\bigl(\Lambda-V(\tau_\Lambda^-+s)\bigr)Y_-(s).
\end{equation*}
Set
\[
H(s)=Y_-(0)e^{\sqrt{\Lambda}s}-Y_-(s).
\]
Then,
\[
H''-\Lambda H
=V(\tau_\Lambda^-+s)Y_-(s)>0,
\]
and, by \eqref{eq:rho-less-sqrt-Lambda} and \eqref{eq:Y-growing-initial-data},
\[
H(0)=0,
\quad
H'(0)=\bigl(\sqrt{\Lambda}-\rho(\Lambda)\bigr)Y_-(0)>0.
\]
By a straightforward argument,  we get $H(s)>0$ for any $s>0$
and hence, $Y_-(s)\leq Y_-(0)e^{\sqrt{\Lambda}s}$ for any $s\ge 0$. 
With \eqref{eq:Y decay shifting}, we have
\begin{equation*} 
P_\Lambda^-(\tau_\Lambda^-+s)
\leq
m_\Lambda^-e^{(\sqrt{\Lambda}-\rho(\Lambda))s}
\quad \text{for any } s\geq0.
\end{equation*}
The periodicity of $P_\Lambda^-(\tau_\Lambda^-+s)$ then yields \eqref{eq:uniform-floquet-ratio} for $P_\Lam^-$.

The proof of \eqref{eq:uniform-floquet-ratio} for $P_\Lambda^+$ is similar and omitted.
We point out that the proof of \eqref{eq:uniform-floquet-ratio} does not need \eqref{eq:point-normalization-Ppm}. 
Now we assume the normalization \eqref{eq:point-normalization-Ppm}. 

{\it Step 2. We prove \eqref{eq:uniform-Ppm-bounds}.} 
 Note that \eqref{eq:rho-greater-sqrt-Lambda-improved}  implies 
$\sqrt{\Lambda}-\rho(\Lambda)\leq a.$
Substituting this in \eqref{eq:uniform-floquet-ratio}, we get
\begin{equation*} 
\max_{\mathbb R}P_\Lambda^\pm
\leq e^{aT} \min_{\mathbb R}P_\Lambda^\pm.
\end{equation*}
Set
$m_\Lambda^\pm=\min_{\mathbb R}P_\Lambda^\pm$
and 
$M_\Lambda^\pm=\max_{\mathbb R}P_\Lambda^\pm$.
By \eqref{eq:point-normalization-Ppm}, we have
\[
m_\Lambda^\pm\leq1\leq M_\Lambda^\pm,
\quad
M_\Lambda^\pm\leq e^{aT}m_\Lambda^\pm,
\]
and hence
\[
M_\Lambda^\pm\leq e^{aT},
\quad
m_\Lambda^\pm\geq e^{-aT}.
\]
This proves \eqref{eq:uniform-Ppm-bounds}. 

{\it Step 3. We prove \eqref{eq:Main result for W Lam} and \eqref{eq:Coro for W Lam}.} 
Since the equation \eqref{eq:continuous-lambda-equation}
has no first-order term, the Wronskian $W_\Lambda$ is a constant. Define 
$Q_{\Lambda}={\Psi_{\Lambda}^{-}}/{\Psi_{\Lambda}^{+}}$.
By the normalization \eqref{eq:point-normalization-Ppm}, we have
$Q_\Lam(0)=1$ and  $Q_{\Lam}(T)=e^{2\rho(\Lam)T}$.
A direct differentiation gives $Q_{\Lambda}^{\prime}={W_{\Lambda}}/{(\Psi_{\Lambda}^{+})^2}.$
A simple integration over $[0,T]$ yields
\[
e^{2 \rho(\Lam) T}-1=Q_{\Lambda}(T)-Q_{\Lambda}(0)
=W_{\Lambda} \int_0^T \frac{d t}{(\Psi_{\Lambda}^{+}(t))^2}
=W_{\Lambda} \int_0^T \frac{e^{2\rho(\Lambda) t} d t}{(P_{\Lambda}^{+}(t))^2},
\]
where we used \eqref{eq:continuous-floquet} for the last equality. 
By \eqref{eq:uniform-Ppm-bounds} for $P_\Lam^+$ 
and a simple evaluation of the integral $\int_0^T e^{2\rho(\Lambda) t} d t$, 
we obtain \eqref{eq:Main result for W Lam}. 
Then, \eqref{eq:Coro for W Lam} follows with the help of $\rho(\Lam)\geq  \rho(\Lam_1)=1$.
\end{proof}

We now define 
\begin{equation}\label{eq:continuous-defect}
\delta(\Lambda)=\rho(\Lambda)-\sqrt{\Lambda}+a.
\end{equation}

\begin{lemma}
\label{lem:continuous-defect-monotonicity}
Let $\rho(\Lambda)$ be given in Lemma \ref{lem:smooth-positive-lambda-new} 
and $\delta(\Lambda)$ be defined by \eqref{eq:continuous-defect}. 
Then, 
$\delta(\Lambda_1)=0$, 
$0<\delta(\Lambda)<a$ for any $\Lambda>\Lambda_1$, and 
$\delta(\Lam)\to a$ as $\Lam\to\infty$. 
Moreover, $\delta(\Lambda)$ is strictly increasing and strictly concave as a function of
$\sqrt{\Lambda}$, for $\Lambda>\Lambda_1$. 
\end{lemma}

\begin{proof} First, $\rho(\Lambda_1)=1$ by Lemma \ref{lem:smooth-positive-lambda-new}. 
With $\Lambda_1=(a+1)^2$, we get $\delta(\Lambda_1)=0.$
We now take $\Lambda>\Lambda_1$. 
It is obvious that \eqref{eq:rho-less-sqrt-Lambda} implies 
$\delta(\Lambda)=\rho(\Lambda)-\sqrt{\Lambda}+a<a.$
Next, by \eqref{eq:der-rho-1&2}  
and \eqref{eq:continuous-defect}, we have easily 
\begin{equation*}
\frac{d}{d\sqrt{\Lambda}}\delta(\Lambda)
>0 
\quad\text{and}\quad
\frac{d^2}{d(\sqrt{\Lambda})^2}\delta(\Lambda)
<0. 
\end{equation*}
Then, $\delta(\Lambda)$ is strictly increasing and strictly concave 
as a function of $\sqrt{\Lambda}$. In particular, we obtain
$\delta(\Lambda)>\delta(\Lambda_1)=0$.
Last, by \eqref{eq:rho-less-sqrt-Lambda}
and \eqref{eq:rho-greater-sqrt-Lambda}, we have, for any $\Lambda>V_{\max}$,
\[
a-
\frac{V_{\max}}
{\sqrt{\Lambda}+\sqrt{\Lambda-V_{\max}}}
\leq
\delta(\Lambda)
<a.
\]
This implies $\delta(\Lambda)\to a$ as $\Lambda\to\infty$.
\end{proof}

In the rest of this section, 
we return to the operator $L_i$ as in \eqref{eq-ode-projection-operator}.
For each $i\ge 1$, we can write $L_i\psi=0$ as 
\begin{equation}\label{eq:mode-equation}
  \psi''-\bigl(\Lambda_i-V\bigr)\psi=0,
\end{equation}
where
\begin{equation}\label{eq:Lam_i-def} 
  \Lambda_i=(i+a)^2.
\end{equation}
We point out that \eqref{eq:mode-equation} is \eqref{eq:continuous-lambda-equation} 
with $\Lambda=\Lambda_i$ given by \eqref{eq:Lam_i-def}. 

First, we characterize the kernels of $L_i$. 

\begin{lemma}
\label{lemma-kernel-of-projection-ode}
Let $L_i$ be as in \eqref{eq-ode-projection-operator}, for $i\geq 0$.
\begin{itemize}
\item[(i)] For $i = 0$, $\mathrm{Ker} (L_0)$ has a basis $\psi_0^+ = p_0^+$ 
and $\psi_0^- = ctp_0^+ + p_0^-$, for some smooth periodic functions $p_0^+$ 
and $p_0^-$ on $\R$, and some constant $c \in \R$.
\item[(ii)] For $i\ge 1$, \eqref{eq:mode-equation} has a basis given by 
\begin{equation}\label{eq:floquet-pm}
  \psi_i^+(t)=e^{-\rho_i t}p_i^+(t),
  \quad
  \psi_i^-(t)=e^{\rho_i t}p_i^-(t),
\end{equation}
where $\rho_i$ is a positive constant and  $p_i^\pm$ are smooth, 
positive, and periodic functions on $\mathbb R$.
Moreover, $\{\rho_i\}_{i\ge 1}$ is strictly increasing, with $\rho_1=1$ and $\rho_i\to\infty$ as $i\to\infty$. 
\end{itemize}
In addition, all periodic functions in $\mathrm{(i)}$ and $\mathrm{(ii)}$ have the same period as $\xi$.
\end{lemma}

Lemma \ref{lemma-kernel-of-projection-ode}, except the positivity of $p_i^\pm$ for $i\ge 1$, 
is known. 
Refer to \cite{Korevaar1998RefinedAF}, \cite{Mazzeo1996ConstantSC}, 
and \cite{Mazzeo1994ModuliSO} for details, 
or to \cite{HanLi2010sigma_k} for a more general setting. 
The present version is formulated after Lemma 2.1 and Lemma 2.2 in \cite{Han2019AsymptoticEO}. 

\begin{proof} 
Let $\rho(\Lambda)$, $P^\pm_\Lambda$, and $\Psi^\pm_\Lambda$ 
be given by Lemma \ref{lem:smooth-positive-lambda-new}.
With $\Lambda=\Lambda_i$ as in \eqref{eq:Lam_i-def}, we set 
$\rho_i=\rho(\Lambda_i)$, $p_i^\pm=P_{\Lambda_i}^\pm$, and $\psi_i^\pm=\Psi_{\Lambda_i}^\pm.$
Then, the desired result follows from 
Lemma \ref{lem:smooth-positive-lambda-new} and 
Lemma \ref{lem:bounds-rho}. 
\end{proof}

Throughout this paper, we write $\{\rho_i\}_{i\geq 1}$ 
as the increasing sequence in Lemma \ref{lemma-kernel-of-projection-ode}(ii).
In the next result, we establish explicit lower and upper bounds for $\rho_i$.

\begin{proposition}
\label{prop-lower-upper-bound-for-rho_i}
There holds, for any $i \geq 2$,
\begin{align}
\label{eq-lower-upper-bound-rho_i-no-xi}
\Big(i+\frac12(n-2)\Big)^2 - \frac{1}{4}n(n+2)  < \rho_i^2 < \Big(i+\frac12(n-2)\Big)^2.
\end{align}
\end{proposition}

Proposition \ref{prop-lower-upper-bound-for-rho_i} follows from Lemma \ref{lem:bounds-rho}
with $\Lambda=\Lambda_i$. 
We point out that \eqref{eq:rho-greater-sqrt-Lambda} is applicable since 
$\Lambda_i>V_{\max}$ for any $i\ge 2$ and that $0<\xi<1$ on $\mathbb R$. 

In the next result, we construct a sequence of bases
$\{\psi_{i}^+,\psi_{i}^-\}_{i\ge 0}$ 
with uniform estimates. 
Denote by $W_i = \psi_i^+ (\psi_i^-)' - \psi_i^-(\psi_i^+)'$ the Wronskian of $\psi_i^+$ and $\psi_i^-$.

\begin{lemma}
\label{lemma-admissible-sequence-of-bases}
Let $L_i$ be as in \eqref{eq-ode-projection-operator}, for $i\geq 0$.
\begin{itemize}
\item[(i)] 
For $i=0$, $\mathrm{Ker}(L_0)$ has a basis given by 
$\psi_0^+=p_0^+$ and $ \psi_0^-=ctp_0^++p_0^-,$
for some smooth periodic functions $p_0^+$ and $p_0^-$ on $\mathbb R$,  
and some constant $c$. 
\item[(ii)] 
For $i\ge 1$, $\mathrm{Ker}(L_i)$ has a basis given by \eqref{eq:floquet-pm}
for some smooth periodic functions $p_i^\pm$ on $\R$ such that, 
for any $i \geq 1$,
$$c_0 \le p_i^{\pm} \le C_0, \quad c_0 \le {W_i}/{\rho_i} \le C_0\quad\text{on }\R,$$
for some positive constants $c_0$ and $C_0$ depending only on 
$n$ and $\xi$.
\end{itemize}
In addition, all periodic functions in $\mathrm{(i)}$ and $\mathrm{(ii)}$ have the same period as $\xi$.
\end{lemma}

Lemma \ref{lemma-admissible-sequence-of-bases}(ii) follows from 
Lemma \ref{lem:uniform-floquet-factor-bounds} with $\Lambda=\Lambda_i$. 
Any sequence of bases for $\{\mathrm{Ker} (L_i)\}_{i\geq 0}$ 
satisfying Lemma \ref{lemma-admissible-sequence-of-bases}
will be called a sequence of {\it admissible bases} for $\{\mathrm{Ker} (L_i)\}_{i\geq 0}$. 
The uniform bounds in (ii) play a significant role in our study.

Define 
\begin{equation}\label{eq:delta-Delta-def}
  \delta_i=\rho_i-i, \quad \Delta_i=\delta_{i+1}-\delta_i. 
\end{equation}

\begin{lemma}\label{lem:delta-Delta}
Let $\{\delta_i\}_{i\ge 1}$ and $\{\Delta_i\}_{i\ge 1}$ be defined by \eqref{eq:delta-Delta-def}. 
Then, 
\begin{itemize}
\item[(i)] 
$\{\delta_i\}_{i\ge 1}$ is strictly increasing, $\delta_1=0$, 
$\delta_i< a$ for any $i\ge 1$, and $\delta_i\to a$ as $i\to\infty$; 
\item[(ii)]  $\{\Delta_i\}_{i\ge 1}$ is a positive and strictly decreasing sequence.
\end{itemize}
\end{lemma}

\begin{proof} Let $\rho(\Lambda)$ be given by Lemma \ref{lem:smooth-positive-lambda-new} 
and $\delta(\Lambda)$ be defined by \eqref{eq:continuous-defect}. 
With $\Lambda=\Lambda_i$ as in \eqref{eq:Lam_i-def}, we have 
\begin{equation}\label{eq:delta i identity}
\delta(\Lambda_i)=\rho(\Lambda_i)-\sqrt{\Lambda_i}+a
=\rho_i-(i+a)+a=\delta_i.
\end{equation}
Lemma \ref{lem:continuous-defect-monotonicity} implies the desired result (i).  



The positivity $\Delta_i>0$ follows from (i). 
It remains to prove $\Delta_i>\Delta_{i+1}$, which is equivalent to 
\begin{align}\label{eq:delta-discrete-concave}
\delta_{i+2}-\delta_{i+1}<\delta_{i+1}-\delta_i.
\end{align}
By Lemma \ref{lem:continuous-defect-monotonicity}, 
$\delta(\Lambda)$ is strictly concave as a function of $\sqrt{\Lambda}$. 
With the increasing sequence $\Lambda_i$ as in \eqref{eq:Lam_i-def}, we have
\begin{equation}\label{eq:Concave Jensen inequality}
\frac{\delta(\Lambda_{i+2})-\delta({\Lambda}_{i+1})}{\sqrt{\Lambda_{i+2}}-\sqrt{\Lambda_{i+1}}}
<\frac{\delta(\Lambda_{i+1})-\delta({\Lambda}_{i})}{\sqrt{\Lambda_{i+1}}-\sqrt{\Lambda_{i}}}.
\end{equation}
Since $\sqrt{\Lambda_{i+1}}-\sqrt{\Lambda_i}=(a+i+1)-(a+i)=1$, 
\eqref{eq:Concave Jensen inequality} is then equivalent to
\[
\delta(\Lambda_{i+2})-\delta({\Lambda}_{i+1})<\delta(\Lambda_{i+1})-\delta({\Lambda}_{i}).
\]
By \eqref{eq:delta i identity}, we have $\delta({\Lambda}_{i})=\delta_i$, 
and hence \eqref{eq:delta-discrete-concave}. 
\end{proof}

We are ready to prove the crucial strict subadditivity of $\{\delta_i\}_{i\ge 1}$.

\begin{proposition}\label{prop:SA}
Let $\{\delta_i\}_{i\ge 1}$ be defined by \eqref{eq:delta-Delta-def}. 
Then, for any integer $N\ge2$ and any collection of integers $j_1,\ldots,j_N\ge2$,  
\begin{equation}\label{eq:SA}
  \sum_{r=1}^N\delta_{j_r}>
  \delta_{j_1+\cdots+j_N}.
\end{equation}
\end{proposition}

\begin{proof}
First, Corollary \ref{cor:base-gap} shows
$2\rho_2>\rho_4$, which is equivalent to $2\delta_2>\delta_4.$
We next prove the pairwise form
\begin{equation}\label{eq:pairwise-SA}
  \delta_p+\delta_q>\delta_{p+q},
\end{equation}
for any $p,q\ge2$. 
Assume without loss of generality that $p\le q$.  
Recall that $ \Delta_i=\delta_{i+1}-\delta_i$ as in \eqref{eq:delta-Delta-def}. 
Since $q\ge2$ and the sequence $\Delta_\ell$ is decreasing by Lemma \ref{lem:delta-Delta}(ii), then
\begin{align*}
  \delta_{p+q}-\delta_q
  =\sum_{\ell=q}^{p+q-1}\Delta_\ell 
  \le \sum_{\ell=2}^{p+1}\Delta_\ell
  =\delta_{p+2}-\delta_2.
\end{align*}
A similar argument yields 
\begin{align*}
  \delta_{p+2}-\delta_p
  \le \delta_{2+2}-\delta_2=\delta_{4}-\delta_2.
\end{align*}
Hence, 
\begin{align*}
  \delta_{p+q}-\delta_p-\delta_q
  \le \delta_{p+2}-\delta_p-\delta_2
  \le \delta_{4}-\delta_2-\delta_2=\delta_{4}-2\delta_2.
\end{align*}
We then obtain \eqref{eq:pairwise-SA}.

The general case 
follows by induction.  
For $N=2$, \eqref{eq:SA} is \eqref{eq:pairwise-SA}.  
Suppose \eqref{eq:SA} holds for $N-1$.  
By the induction hypothesis for $N-1$ and \eqref{eq:pairwise-SA}, 
we have 
\[
  \sum_{r=1}^N\delta_{j_r}
  >\delta_{j_1+\cdots+j_{N-1}}+
  \delta_{j_N}
  >\delta_{j_1+\cdots+j_N}.
\]
This proves \eqref{eq:SA}. 
\end{proof}

To end this section, we discuss briefly the case 
that $\xi$ is the unique positive constant solution of \eqref{eq-ode-for-radial-singular-solution}, 
given by \eqref{eq-positive-constant-solu}. 
In this case, the sequence of the exponential decay orders $\{\rho_i\}_{i\ge 1}$ is expressed explicitly by 
\eqref{eq-i-Yamabe-linearization-constant-decay-rate}.
It is straightforward to verify that the function $h(s)=\big(s^2+(n-2)(s-1)\big)^{1/2}-s$ 
is strictly increasing and strictly concave on $[1,\infty)$. 
Then, Lemma \ref{lem:delta-Delta} follows easily. 
We can also prove directly $2\rho_2>\rho_4$.
Then, Proposition \ref{prop:SA} holds.

\section{Linear Equations}\label{sec-Linear-Equations}

In this section, we study linear equations associated with \eqref{eq-linearized-Yamabe-at-xi}. 
First, we discuss a class of special solutions to the linear equations. 
Then, we derive decay estimates for the linear equations. 

We fix a smooth positive periodic solution $\xi$ to \eqref{eq-ode-for-radial-singular-solution}, 
and let $\mathcal L$ be given by \eqref{eq-linear-operator-of-Yamabe}
and $L_i$ given by \eqref{eq-ode-projection-operator}, for $i\ge 0$. 
For each $i \ge 0$, we fix a basis $\{\psi_i^+, \psi_i^-\}$ of $\mathrm{Ker} (L_i)$ 
as in Lemma \ref{lemma-kernel-of-projection-ode},  
and let $W_i$ be the Wronskian of $\psi_i^+$ and $\psi_i^-$, 
which is constant.

For $i\ge 0$, we first consider 
\begin{align}\label{eq-linearization-ODE-i-nonconstant}L_i\psi=f\quad\text{on }[t_0, \infty).\end{align}
It is convenient to introduce the following terminology. 

\begin{definition}\label{def-generalized-polynomial}
A smooth function $r$ on $\R$ is called a {\it generalized polynomial} 
if there exists a nonnegative integer $\ell$ such that
\begin{align}\label{eq-polynomial-nonconstant}r(t) = \sum_{k=0}^\ell r_k(t) t^k,\end{align}
for some smooth periodic functions $r_0, \cdots, r_\ell$ on $\R$ with the same period. 
The smallest possible $\ell$ is called the degree of $r$ and the common period of $r_0, \cdots, r_\ell$ 
is called the period of $r$.
\end{definition}

We now consider a special class of $f$ in \eqref{eq-linearization-ODE-i-nonconstant}. 
Let $\mu > 0$ be a constant and $r$ be a generalized polynomial as in \eqref{eq-polynomial-nonconstant}. 
Consider 
\begin{align}\label{eq-special-f-nonconstant} f(t) = r(t) e^{-\mu t}.\end{align} 
First, for $i = 0$, we write $\psi_0^+ = p_0^+$ and $\psi_0^- = c t p_0^+ + p_0^-$ 
as in Lemma \ref{lemma-kernel-of-projection-ode}(i), and define 
\begin{align}\label{eq-def-L_0^-1}\begin{split}
L_0^{-1} [f](t) &= \frac{p_0^+(t)}{W_0}\int_t^\infty p_0^-(s) f(s) d s 
- \frac{p_0^-(t)}{W_0}\int_t^\infty p_0^+(s) f(s)d s\\
&\,\,\,\,\,\,\,+\frac{c p_0^+(t)}{W_0} \int_{t}^\infty \int_s^\infty p_0^+(\tau) f(\tau) d \tau d s. 
\end{split}
\end{align}
Next, for $i \geq 1$, define, for $\mu > \rho_i$,
\begin{align}
\label{eq-def-L_i^-1}
L_i^{-1} [f](t) = 
\frac{\psi_i^+(t)}{W_i}\int_{t}^\infty \psi_i^-(s) f(s) d s 
-  \frac{\psi_i^-(t)}{W_i}\int_{t}^\infty \psi_i^+(s) f(s) d s.
\end{align}
We replace the first integral by $-\int_{-\infty}^t$ if $\mu < \rho_i$ and 
by $-\int_0^t$ if $\mu = \rho_i$. 
It is obvious that the map $f\mapsto L_i^{-1}[f]$ is linear.

In the next result, we characterize the integrals 
in \eqref{eq-def-L_0^-1} and \eqref{eq-def-L_i^-1} 
for $f$ given by \eqref{eq-special-f-nonconstant}. 
For simplicity, we consider just one term in \eqref{eq-special-f-nonconstant}. 

\begin{lemma}
\label{lemma-computation-of-periodic-exp-integral-general}
Let $\mu > 0$ be a constant, $k \ge 0$ be an integer, and $r$ be a smooth periodic function on $\R$. 
Then,
\begin{itemize}
\item[(i)] $\displaystyle \int_0^t r(s) s^k d s = c t^{k+1} + \sum_{\ell=0}^kr_\ell(t) t^\ell$, 
where $c$ is a constant and $r_0,\cdots,r_k$ are smooth periodic functions 
on $\R$ with the same period as $r$;
\item[(ii)] $\displaystyle \int_{t}^\infty r(s) s^k e^{-\mu s}d s =  \sum_{\ell=0}^kr_\ell(t) t^\ell e^{-\mu t}$, 
where $r_0,\cdots,r_k$ are smooth periodic functions on $\R$ with the same period as $r$.
\end{itemize}
\end{lemma}

A similar result holds for $\int_{-\infty}^t r(s) s^k e^{\mu s}ds$.
The proof is elementary and omitted.

We now study \eqref{eq-linearization-ODE-i-nonconstant} with $f$ given by \eqref{eq-special-f-nonconstant}, i.e., 
\begin{align}\label{eq-linearization-ODE-i-nonconstant-one}(L_i\psi)(t)=r(t) e^{-\mu t}\quad\text{on }\mathbb R,\end{align}
and consider only solutions $\psi$ such that $e^{\mu t}\psi$ 
is a generalized polynomial.
For convenience, we set $\rho_0 = 0$.

\begin{lemma}
\label{lemma-computation-of-the-exp-poly-integral-nonconstant} 
Let $i, \ell\ge 0$ be integers, $\mu>0$ be a constant, 
and $r$ be a generalized polynomial of degree $\ell$ with the same period as $\xi$. 
\begin{itemize}
\item[(i)] Then, $L_i^{-1} [r(s)e^{-\mu s}](t)$
is a solution of \eqref{eq-linearization-ODE-i-nonconstant-one} 
and $e^{\mu t}L_i^{-1} [r(s)e^{-\mu s}](t)$ is a generalized polynomial, 
with degree $\ell$ for $\mu\neq \rho_i$ and at most $\ell+1$ for $\mu= \rho_i$, and 
the same period as $\xi$.
\item[(ii)] Let $\psi$ be a solution of \eqref{eq-linearization-ODE-i-nonconstant-one}
such that $e^{\mu t} \psi$ is a generalized polynomial.
Then, if $\mu\neq \rho_i$, $\psi$ is uniquely given by 
$$\psi(t)=L_i^{-1} [r(s)e^{-\mu s}](t),$$ 
and, if $\mu= \rho_i$, $\psi$ is given by
$$\psi(t)=cp_i^{+}e^{-\rho_i t} + L_i^{-1} [r(s)e^{-\mu s}](t),$$ 
for some constant $c$. 
\end{itemize}
\end{lemma}

The proof is based on explicit computations and omitted.
We can rewrite a part of Lemma \ref{lemma-computation-of-the-exp-poly-integral-nonconstant}
as follows. Let 
$p$ be a generalized polynomial with the same period as $\xi$. Then, 
for $\mu\neq \rho_i$,
\begin{align}\label{eq-identity-inverse-L-i}L_i^{-1}\, L_i [p(s)e^{-\mu s}](t)=p(t)e^{-\mu t}.\end{align}


In the rest of this section, we study 
\begin{equation}\label{eq-linearization-PDE-nonconstant}
\mathcal L\varphi=f\quad\text{on }[t_0, \infty)\times\mathbb S^{n-1}.\end{equation}
First, we consider 
a special class of $f$. 
We fix an orthogonal basis $Y_{i1},\cdots,Y_{id_i}$ of $\mathscr{E}_i$ 
in $L^2(\mathbb{S}^{n-1})$, 
for each $i\ge 0$.

\begin{definition}\label{def-generalized-polynomial-order}
A smooth function $R$ on $\R\times\mathbb{S}^{n-1}$ is called a 
{\it generalized polynomial with $C^\infty(\mathbb{S}^{n-1})$-coefficients} 
of a finite spherical degree
if there exist nonnegative integers $\ell$ and $N$ such that
\begin{align}\label{eq-expression-polynomial-nonconstant}
R(t,\theta) = \sum_{k=0}^\ell \sum_{i=0}^N \sum_{j=1}^{d_i} Y_{ij}(\theta) r_{ijk}(t) t^k,\end{align}
for some smooth periodic functions $r_{ijk}$ on $\R$ with the same period. 
The smallest possible $\ell$ is called the {\it polynomial degree} of $R$, 
the smallest possible $N$ is called the {\it spherical degree} of $R$, 
and the common period of $r_{ijk}$ is called the period of $R$.
\end{definition}

When there is no ambiguity, we will also use generalized polynomial to mean
a generalized polynomial with $C^\infty(\mathbb S^{n-1})$-coefficients of a finite spherical degree.
Note that the notion of a generalized polynomial with $C^\infty(\mathbb{S}^{n-1})$-coefficients of a finite spherical degree 
is independent of the choice of an orthogonal basis of $\mathscr{E}_i$ in $L^2(\mathbb{S}^{n-1})$, 
for $i \ge 0$. 

For the linear operator $\mathcal L$ in \eqref{eq-linear-operator-of-Yamabe}
and a constant $\mu>0$, consider the linear equation
\begin{align}
\label{eq-special-linear-equation}
\mathcal L (P e^{-\mu t}) = R e^{-\mu t},
\end{align}
where $P$ and $R$ are generalized polynomials with $C^\infty(\mathbb{S}^{n-1})$-coefficients 
of finite spherical degree and the same period as $\xi$.
We now introduce a linear operator $\mathcal L_\mu$ such that $\mathcal L (P e^{-\mu t})=e^{-\mu t}\mathcal L_\mu P.$
Then, 
\begin{align}\label{eq-linear-equation-mu-polynomial-nonconstant}
\mathcal L_\mu P=\partial_{tt} P - 2\mu \partial_t P + \Delta_\theta P + \Big(\mu^2 - \frac{1}{4}(n-2)^2 
+ \frac{1}{4}n(n+2) \xi^\frac{4}{n-2}\Big) P, 
\end{align}
and \eqref{eq-special-linear-equation} holds if and only if 
\begin{align}\label{eq-special-linear-equation-equiv}
\mathcal L_\mu P=R.\end{align}
Recall that $L_i^{-1}$ is defined in \eqref{eq-def-L_0^-1} for $i=0$ 
and in \eqref{eq-def-L_i^-1} for $i\ge 1$.

\begin{lemma}\label{lemma-polynomial_solution}
Let $\ell, N\ge 0$ be integers, 
$\mathcal L$ be given by \eqref{eq-linear-operator-of-Yamabe}, 
$\mu > 0$ be a constant, and $R$
be a generalized polynomial given by \eqref{eq-expression-polynomial-nonconstant}, 
for some smooth periodic functions $r_{ijk}$ on $\mathbb R$ with the same period as $\xi$, 
for $0\le i\le N$, $0\le k\le\ell$, and $1\le j\le d_i$. Set 
\begin{align}\label{eq-L_mu^-1} 
\mathcal L_\mu^{-1}[R](t,\theta) = e^{\mu t} \sum_{k=0}^\ell \sum_{i=0}^N \sum_{j=1}^{d_i} 
Y_{ij}(\theta) L_i^{-1}[r_{ijk}(s) s^k e^{-\mu s}](t).\end{align}
Then, $\mathcal L_\mu^{-1}[R]$ is a generalized polynomial solution of 
\eqref{eq-special-linear-equation-equiv} of polynomial degree at most $\ell+1$, spherical degree at most $N$,   
and the same period as $\xi$. Moreover, 
\begin{itemize}
\item[(i)] if $\mu \neq \rho_i$ for any $i \ge 1$, 
then $P=\mathcal L_\mu^{-1}[R]$ is the unique generalized polynomial solution of 
\eqref{eq-special-linear-equation-equiv}
with $C^\infty(\mathbb{S}^{n-1})$-coefficients of a finite spherical degree;
\item[(ii)] if $\mu = \rho_i$ for some $i \ge 1$, 
then any generalized polynomial solution $P$ of \eqref{eq-special-linear-equation-equiv}
with $C^\infty(\mathbb{S}^{n-1})$-coefficients of a finite spherical degree is given by 
$$P(t,\theta) =X_i(\theta) p_i^+(t)+ \mathcal  L_\mu^{-1}[R](t,\theta),$$ 
for some $X_i \in \mathscr{E}_i$;  
\item[(iii)]  if $\mu \neq \rho_i$ for any $1 \leq i \leq N$, 
then the polynomial degree of $\mathcal L_\mu^{-1}[R]$ is at most $\ell$.
\end{itemize}
In addition, the map $R \mapsto \mathcal L_\mu^{-1}[R]$ is linear. 
\end{lemma}

\begin{proof} 
By Lemma \ref{lemma-computation-of-the-exp-poly-integral-nonconstant}, 
it is straightforward to check 
that $\mathcal L_\mu^{-1}[R]$ is a generalized polynomial, 
of the stated polynomial degree, spherical degree, and period, 
and solves \eqref{eq-special-linear-equation-equiv}. 
It remains to characterize solutions 
to the homogeneous equation $\mathcal L_\mu P=0$, or equivalently, 
\begin{align}
\label{eq-homogeneous-poly-equation-nonconstant}
\mathcal L (P e^{-\mu t} ) = 0,
\end{align}
where $P$ is a generalized polynomial with $C^\infty(\mathbb{S}^{n-1})$-coefficients of finite spherical degree. 
We claim that either $\mu\neq \rho_i$ for any $i\ge 1$ and $P=0$ 
or $\mu=\rho_i$ for some $i\ge 1$ and $P(t, \theta)=X_i(\theta)p_i^+(t)$ for some $X_i\in\mathscr E_i$.
For simplicity, we assume that $\{Y_{ij}\}$ is orthonormal in $L^2(\mathbb{S}^{n-1})$. 
For some nonnegative integers $\ell$ and $N$, we write
$$P(t,\theta) = \sum_{i=0}^N \sum_{j=1}^{d_i} Y_{ij}(\theta) q_{ij}(t)
\quad\text{with }q_{ij}(t) = \sum_{k=0}^\ell p_{ijk}(t) t^k,$$
for some smooth periodic functions $p_{ijk}$ with the same period as $\xi$, 
for $0\le i\le N$, $0\le k\le \ell$, and $1\le j\le d_i$. 
Fix $0\leq i \leq N$ and $1 \leq j \leq d_i$. 
Multiplying \eqref{eq-homogeneous-poly-equation-nonconstant} by $Y_{ij}$ and  
integrating over $\mathbb{S}^{n-1}$, we have $L_i ( q_{ij} e^{-\mu t} ) = 0$. 
Hence,
$$q_{ij}  e^{-\mu t} = a_{ij}^+ \psi_i^+ + a_{ij}^- \psi_i^-,$$
for some constants $a_{ij}^+, a_{ij}^-$. 
For $i=0$ or $1\leq i \leq N$ but $\mu \neq \rho_i$, we have 
$a_{ij}^- = a_{ij}^+=0$ and hence $p_{ijk}=0$ for any $0\le k\le \ell$ and $1 \leq j \leq d_i$. 
For $1\leq i \leq N$ and $\mu = \rho_i$, we have 
$a_{ij}^-=0$, for any $1 \leq j \leq d_i$, 
and $ p_{ijk}=0$, for any $1\le k\le \ell$ and $1 \leq j \leq d_i$. 
Therefore, 
$$P(t,\theta) = \sum_{j=1}^{d_i} Y_{ij}(\theta) p_{ij0}(t)=p_i^+(t)\sum_{j=1}^{d_i} a_{ij}^+Y_{ij}(\theta).$$
We may write $P(t,\theta) = p_i^+(t) X_i(\theta),$
for some $X_i\in\mathscr{E}_i$. 
We then have the claim. 
\end{proof}

In (ii), $X_i$ is called the ``free part" of $P$. This notion depends on the choice 
of the bases $\{\psi_i^+,\psi_i^-\}_{i\geq 0}$ for $\{\mathrm{Ker}(L_i)\}_{i\geq 0}$.

We can rewrite a part of Lemma \ref{lemma-polynomial_solution}
as follows. Let 
$P$ be a generalized polynomial 
with $C^\infty(\mathbb{S}^{n-1})$-coefficients of 
spherical degree $N$. By \eqref{eq-identity-inverse-L-i},  if $\mu\neq \rho_i$ for any $i\le N$, then 
\begin{align}\label{eq-identity-inverse-L-mu}
\mathcal L_\mu^{-1}[\mathcal L_\mu P]=P.\end{align}


Next, we discuss decay estimates for the linear equation \eqref{eq-linearization-PDE-nonconstant}. 
Refer to Lemma A.8\cite{Han2019AsymptoticEO}.

\begin{lemma}
\label{lemma-same-decay-estimate-for-L}
Let $\gamma > 0$ be a constant, $k\geq 0$ be an integer, and $f$ be a smooth function on $[t_0,\infty)\times\mathbb{S}^{n-1}$, 
for some $t_0>0$, satisfying, for any $(t,\theta) \in [t_0,\infty)\times \mathbb{S}^{n-1}$,
$$|f(t,\theta)| \leq C_0t^k e^{-\gamma t},$$
for some positive constant $C_0>0$. Let $\varphi$ be a solution of \eqref{eq-linearization-PDE-nonconstant} 
on $[t_0,\infty)\times \mathbb{S}^{n-1}$ 
such that $\varphi(t,\theta) \rightarrow 0$ as $t \rightarrow \infty$, uniformly in $\theta \in \mathbb{S}^{n-1}$.
\begin{itemize}
\item[(i)] If $\gamma \leq \rho_1$, then,  
for any $(t,\theta) \in [t_0,\infty)\times \mathbb{S}^{n-1}$,
        $$\left|\varphi(t,\theta)\right| \leq 
        \begin{cases} 
        Ct^k e^{-\gamma t} & \text{if }  \gamma < \rho_1,\\
        C t^{k+1} e^{-\gamma t}& \text{if } \gamma = \rho_1,
        \end{cases}$$
where $C$ is a positive constant depending on $C_0$, $k$, $\gamma$, and $\varphi$. 
\item[(ii)] If $\rho_\ell < \gamma \leq \rho_{\ell+1}$ for some integer $\ell \geq 1$, 
then there exist spherical harmonics $X_i \in \mathscr{E}_i$ for $i=1, \cdots, \ell$ 
such that, for any $(t,\theta) \in [t_0,\infty)\times \mathbb{S}^{n-1}$,
        $$\Big|\varphi(t,\theta) - \sum_{i=1}^\ell X_i(\theta)p_i^+(t) e^{-\rho_i t} \Big| \leq 
        \begin{cases} 
        Ct^k e^{-\gamma t} & \text{if }  \rho_\ell < \gamma < \rho_{\ell+1},\\
        C t^{k+1} e^{-\gamma t}& \text{if } \gamma = \rho_{\ell+1},
        \end{cases}$$        
where $C$ is a positive constant depending on $C_0$, $k$, $\gamma$, $\rho_\ell$, and $\varphi$. 
\end{itemize}
\end{lemma}

\section{Asymptotic Expansions}\label{sec-Asymp-Expansions}

Starting from this section, we will study exponential decay solutions 
of the nonlinear equation \eqref{eq-linearized-Yamabe-at-xi}. 
Although all results in the following are formulated for the nonlinear term $F$ 
given by \eqref{eq-RHS-of-linearized-Yamabe}, 
the explicit expression in \eqref{eq-RHS-of-linearized-Yamabe} 
is needed only in Section \ref{sec-Conformal-Transforms}. 

We first list the notations to be used in the remaining part of the paper as follows: 
\begin{itemize}
    \item[(N1)] $\xi$ is a given smooth positive periodic solution to \eqref{eq-ode-for-radial-singular-solution}, with $0 < \xi < 1$ on $\R$;
    \item[(N2)] $\mathcal L$ is given by \eqref{eq-linear-operator-of-Yamabe} and $L_i$ given by \eqref{eq-ode-projection-operator} 
for $i \ge 0$;
    \item[(N3)] $\{\rho_i\}_{i\ge 1}$ is the sequence of positive constants in Lemma \ref{lemma-kernel-of-projection-ode}(ii), with $\rho_1=1$; 
    \item[(N4)] $\{\psi_i^+, \psi_i^-\}$ is the basis of $\mathrm{Ker} (L_i)$ in Lemma \ref{lemma-kernel-of-projection-ode} for $i\ge 0$;
    \item[(N5)] $L_i^{-1}$ is given by \eqref{eq-def-L_0^-1} and \eqref{eq-def-L_i^-1} for $i \geq 0$
and $\mathcal L_\mu^{-1}$ by \eqref{eq-L_mu^-1};
    \item[(N6)] $\{\mu_\ell\}_{\ell\geq 1}$ is the index set associated with $\mathcal L$ 
to be defined in Definition \ref{def-index-set}.
\end{itemize}
In particular, for $i \geq 1$,
$$\psi_i^+ = p_i^+ e^{-\rho_i t}\quad\text{and}\quad\psi_i^- = p_i^- e^{\rho_i t},$$
where $p_i^+$ and $p_i^-$ are positive smooth periodic functions on $\R$ with the same period as $\xi$. 
Throughout the paper, all generalized polynomials with $C^\infty(\mathbb{S}^{n-1})$-coefficients 
have finite spherical degrees and the same period as $\xi$.
We always denote by $\mathbb Z_+$ the collection of nonnegative integers.

In this section, 
we characterize asymptotic behaviors of exponential decay solutions of \eqref{eq-linearized-Yamabe-at-xi}
and derive asymptotic expansions. 
Theorem \ref{thrm-higher-order-expansion} is an improved version of Theorem C in the introduction 
and concerns the expansion up to arbitrary orders. 
Our proof of Theorem \ref{thrm-higher-order-expansion} here is different from that in \cite{Han2019AsymptoticEO}. 

We always consider a solution $\varphi$ of \eqref{eq-linearized-Yamabe-at-xi} 
on $[t_0,\infty)\times\mathbb{S}^{n-1}$, for some $t_0 \geq 0$, 
such that {\it $\varphi(t,\theta)$ converges to $0$ exponentially as $t\to\infty$, uniformly in $\theta \in \mathbb{S}^{n-1}$}, 
i.e., for any $(t,\theta)\in [t_0,\infty)\times\mathbb{S}^{n-1}$, 
\begin{equation}\label{eq-exp-convergence}
\left|\varphi(t,\theta) \right| \leq C_* e^{-\gamma t},\end{equation} 
for some positive constants $C_*$ and $\gamma$. 

Recall that $F$ is given by \eqref{eq-RHS-of-linearized-Yamabe}.
It is easy to check that, for any $|s|<\min \xi$,  
\begin{align}\label{eq-estimate-F}
|F(s)| \leq C_0 s^2,\end{align} 
for some constant $C_0>0$, depending only on $n$ and $\xi$.
We view $F(\varphi)$ as a higher order term in \eqref{eq-linearized-Yamabe-at-xi}.
In the following, we always assume $|\varphi|< \min \xi$ on $[t_0, \infty)\times\mathbb S^{n-1}$. 

For simplicity, we will use the big $O$ notation: 
$$\text{$h = O(g) \Longleftrightarrow |h(t,\theta)| \leq Cg(t)$ for any $(t,\theta)\in [t_0,\infty)\times\mathbb{S}^{n-1}$},$$ 
for some constant $C > 0$.

From 
Sections \ref{sec-admissible-bases} and \ref{sec-Linear-Equations}, 
we know that $\{\rho_i\}_{i\ge 1}$ provides all exponential decay orders 
for the homogeneous linear equation $\mathcal L\varphi=0$. 
With $F$ given by \eqref{eq-RHS-of-linearized-Yamabe}, 
the nonlinear equation $\mathcal L\varphi=F$ in \eqref{eq-linearized-Yamabe-at-xi}
can be considered as a perturbation of $\mathcal L\varphi=0$.
For the exponential decay orders of solutions of \eqref{eq-linearized-Yamabe-at-xi}, 
it is reasonable to expect that all $\rho_i$ will remain and linear combinations of $\{\rho_i\}_{i\ge 1}$
with nonnegative integer coefficients will appear due to the nonlinearity. 
It is natural to introduce the following concept.
Refer to (2.9) in \cite{Han2019AsymptoticEO}.

\begin{definition}\label{def-index-set}
The {\it index set} $\mathcal I$ associated 
with the operator $\mathcal L$ in \eqref{eq-linear-operator-of-Yamabe} is the collection 
of all nontrivial linear combinations of $\{\rho_i\}_{i \geq 1}$ with nonnegative integer coefficients, i.e., 
\begin{align*}
\mathcal I=\Big\{\sum_{i\geq 1}m_i\rho_i ;\, m_i 
\in \mathbb Z_+ \text{ with finitely many }m_i > 0\Big\} \setminus\{0\}.
\end{align*}
Write $\mathcal I=\{\mu_\ell\}_{\ell\ge 1}$ for a strictly increasing sequence of positive numbers. 
\end{definition}

Lemma \ref{lem:delta-Delta}(i) implies $\rho_i>i$ for $i\ge 2$. 
Then, $\mu_1 = \rho_1=1$, $\mu_2=2\rho_1=2$, 
and $\mu_3 = \min \{3\rho_1, \rho_2\}$.

We now demonstrate by 
formal computation that 
$\mathcal I$ provides all decay orders of 
the exponential decay solutions of \eqref{eq-linearized-Yamabe-at-xi}. 
Assume $\varphi$ is a solution of \eqref{eq-linearized-Yamabe-at-xi} on $[t_0,\infty)\times\mathbb{S}^{n-1}$, 
for some $t_0 \ge 0$, 
and has a formal expansion given by 
\begin{equation}\label{eq-formal-expansion}
\varphi(t,\theta)= \sum_{\ell=1}^\infty P_\ell(t,\theta) e^{-\mu_\ell t},\end{equation}
where $\{P_\ell\}_{\ell \geq 1}$ is a sequence of generalized polynomials with $C^\infty(\mathbb{S}^{n-1})$-coefficients. 
Next, we write $F(\varphi)$ in \eqref{eq-linearized-Yamabe-at-xi} formally as 
\begin{equation}\label{eq-formal-expansion-F}
F(\varphi)=\sum_{j=2}^\infty a_j \varphi^j,\end{equation}
where $a_j$ are smooth periodic functions on $\R$ with the same period as $\xi$, for $j \ge 2$. 
By a formal term-by-term differentiation in \eqref{eq-formal-expansion}, a simple substitution, and rearrangements, 
we find that each $P_\ell$ satisfies, for $\ell = 1$,
\begin{equation}\label{eq-recursive-equation-Asymp-expans-1}\mathcal L(P_1 e^{-\mu_1 t}) = 0,\end{equation}
and, for $\ell \geq 2$,
\begin{equation}\label{eq-recursive-equation-Asymp-expans}
\mathcal L(P_\ell e^{-\mu_\ell t})
= \sum_{j = 2}^\ell \sum_{\mu_{k_1} + \cdots + \mu_{k_j} = \mu_\ell} a_j P_{k_1} \cdots P_{k_j} e^{-\mu_\ell t}.
\end{equation}
We point out that the second summation is over all possible $(k_1,\cdots,k_j) \in \N^j$. 

The formal computation above illustrates that the index set provides the correct decay orders. 
It also introduces the recursive equations \eqref{eq-recursive-equation-Asymp-expans}. 
No results to be proved are based on the formal expansion \eqref{eq-formal-expansion}.

\begin{definition}\label{def-sequence-P-determined-by-X}
Let 
$\mathcal{X} = \{X_i\}_{i\geq 1}$ be a sequence of spherical harmonics 
with $X_i \in \mathscr{E}_i$. 
The sequence of generalized polynomials $\{P_\ell(\mathcal{X})\}_{\ell \geq 1}$
is said to be  {\it determined by $\mathcal{X}$} if 
\begin{equation}\label{eq-P-1} P_1(\mathcal{X})(t,\theta)=X_1(\theta)p_1^+(t),\end{equation}
and $P_\ell=P_\ell(\mathcal{X})$ is defined inductively for $\ell\ge 2$ by the following: 
\begin{itemize}
\item[(i)] if $\mu_\ell \neq \rho_i$ for any $i \geq 2$, then 
\begin{align*}
P_{\ell}=\mathcal L_{\mu_\ell}^{-1}\Big[ \sum_{j = 2}^\ell \sum_{\mu_{k_1} + \cdots + \mu_{k_j} = \mu_\ell} 
a_j P_{k_1} \cdots P_{k_j} \Big];
\end{align*}
\item[(ii)] if $\mu_\ell = \rho_i$ for some $i \geq 2$ and $\rho_i \notin \mathbb Z_+\text{-span}\{\rho_1,\cdots,\rho_{i-1}\}$, 
then 
$${ P_\ell(t,\theta)  =  X_i(\theta)p_i^+(t)};$$
\item[(iii)] if $\mu_\ell=\rho_i$ for some $i \geq 2$ and $\rho_i \in \mathbb Z_+\text{-span}\{\rho_1,\cdots,\rho_{i-1}\}$, then 
$$P_\ell(t,\theta) = X_i(\theta)p_i^+(t)
+\mathcal L_{\mu_\ell}^{-1}\Big[ \sum_{j = 2}^\ell \sum_{\mu_{k_1} + \cdots + \mu_{k_j} = \mu_\ell} 
a_j P_{k_1} \cdots P_{k_j} \Big].$$ 
\end{itemize}
\end{definition} 



We point out that the sum inside $\mathcal L_{\mu}^{-1}$ in Case (i) and Case (iii) 
is explicitly expressed in terms of $P_j$ for $1\le j<\ell$ 
and is a generalized polynomial of a finite spherical degree by Lemma \ref{lemma-decomposition-of-spherical-harmonics}, 
since each $P_{k_s}$ is a generalized polynomial of a finite spherical degree. 


\begin{proposition}\label{prop-polynomial-coefficients-new}
Let $\mathcal{X} = \{X_i\}_{i\geq 1}$ be a sequence of spherical harmonics with $X_i \in \mathscr{E}_i$ and
$\{P_\ell=P_\ell(\mathcal{X})\}_{\ell \geq 1}$ be the sequence of generalized polynomials 
determined by $\mathcal{X}$ as in Definition \ref{def-sequence-P-determined-by-X}.  
Then, $\{P_\ell(\mathcal{X})\}_{\ell\geq 1}$ satisfies 
\eqref{eq-recursive-equation-Asymp-expans-1}-\eqref{eq-recursive-equation-Asymp-expans}, 
and the polynomial degree of $P_\ell(\mathcal{X})$ is at most $\ell-1$ for any $\ell\ge 1$.
\end{proposition}

\begin{proof} First, by a direct computation, 
\eqref{eq-recursive-equation-Asymp-expans-1}-\eqref{eq-recursive-equation-Asymp-expans}
hold trivially. We point out that, in case (ii), 
if $\mu_\ell = \rho_i$ for some $i \geq 2$ and 
$\rho_i \notin \mathbb Z_+\text{-span}\{\rho_1,\cdots,\rho_{i-1}\}$, 
then the summation in the right-hand side of \eqref{eq-recursive-equation-Asymp-expans} 
is over the empty set.


We next characterize the polynomial degree of $P_\ell$. 
We will prove by induction that $P_\ell$ has a polynomial degree at most $[\mu_\ell]-1$,
where $[\mu_\ell]$ is the largest integer smaller than or equal to $\mu_\ell$.
Since $[\mu_\ell] \le \mu_\ell \le \ell$, this implies the desired result.
By \eqref{eq-P-1}, $P_1$ has polynomial degree $0 = [\mu_1]-1$, since $\mu_1 = 1$.
Assume that $P_j$ has a polynomial degree at most $[\mu_j]-1$ for $j = 1,\cdots, \ell$, for some $\ell \ge 1$.
If $\mu_{\ell+1} = \rho_i$ for some $i \ge 2$ and 
$\rho_i \notin \Z_+\text{-span}\{\rho_1,\cdots, \rho_{i-1}\}$, 
then, by (ii), $P_{\ell+1}$ has polynomial degree $0$.
Otherwise, by (i), (iii), and \eqref{eq-recursive-equation-Asymp-expans},
the polynomial degree of $P_{\ell+1}$ is determined by
$$R_{\ell+1} = \sum_{j = 2}^{\ell+1} \sum_{\mu_{k_1} + \cdots + \mu_{k_j} = \mu_{\ell+1}} a_j P_{k_1} \cdots P_{k_j}.$$
By the induction hypothesis, if $\mu_{k_1} + \cdots + \mu_{k_j} = \mu_{\ell+1}$, then
$a_j P_{k_1}\cdots P_{k_j}$ has a polynomial degree at most 
$[\mu_{k_1}] + \cdots + [\mu_{k_j}]-j \le \mu_{\ell+1}-2, $
since $j \ge 2$.
Hence, $R_{\ell+1}$ has a polynomial degree at most $[\mu_{\ell+1}]-2$.
By Lemma \ref{lemma-polynomial_solution} and \eqref{eq-recursive-equation-Asymp-expans},
we have
$$\deg P_{\ell+1} \le \deg R_{\ell+1}+1 \le [\mu_{\ell+1}]-1.$$
This finishes the proof by induction and implies the desired conclusion.
\end{proof}



The next theorem 
concerns 
expansions of $\varphi$ up to arbitrary orders.

\begin{theorem}\label{thrm-higher-order-expansion}
Assume $\varphi$ is a solution to \eqref{eq-linearized-Yamabe-at-xi} on $[t_0,\infty)\times\mathbb{S}^{n-1}$, with $t_0>0$, 
satisfying \eqref{eq-exp-convergence} for some $\gamma, C_*>0$.  
Then, there exists a sequence $\mathcal{X} = \{X_i\}_{i\geq 1}$ of spherical harmonics 
with $X_i \in \mathscr{E}_i$
such that, for any $\ell \ge 1$
and any $(t,\theta) \in [t_0,\infty)\times\mathbb{S}^{n-1}$, 
\begin{align}\label{eq-higher-order-expansion-inequality}
\Big|{\varphi(t,\theta) - \sum_{j=1}^\ell P_j(\mathcal{X})(t,\theta) e^{-\mu_j t}}\Big| \leq C_{\ell} t^\ell e^{-\mu_{\ell+1} t},
\end{align}
where $\{P_\ell(\mathcal{X})\}_{\ell \geq 1}$ is the sequence of generalized polynomials 
determined by $\mathcal{X}$ as in Definition \ref{def-sequence-P-determined-by-X}, 
and $C_{\ell}$ is a positive constant depending only on $n$, $\xi$, $\varphi$, $t_0$, and $\ell$. 
\end{theorem}

\begin{proof}
All estimates in the proof hold for 
any $t\ge t_0$ and any $\theta\in \mathbb S^{n-1}$. 
For an arbitrarily fixed large integer $K$, we expand $F$ up to order $K$ and write 
\begin{align}\label{eq-RHS-expanded-K}
F(\varphi)= \sum_{j=2}^K a_j \varphi^j + F_K, 
\end{align}
where $a_j$ are smooth periodic functions on $\R$ with the same period as $\xi$, for $j \geq 2$, and
$$|F_K| \le C|\varphi|^{K+1},$$ 
for some positive constant $C>0$, depending only on $n$, $\xi$, and $K$.
We first prove that there exists a sequence $\mathcal X = \{X_i\}_{i\ge 1}$ 
with $X_i \in \mathscr E_i$ such that,
for any $\ell \ge 1$,
\begin{align}
\label{eq-higher-order-expansion-inequality-mu}
\varphi = \sum_{j=1}^\ell P_j(\mathcal X) e^{-\mu_j t} + O(e^{-\mu t}),
\end{align}
for some $\mu \in (\mu_\ell,\mu_{\ell+1})$,
where $\{P_\ell(\mathcal{X})\}_{\ell \geq 1}$ is the sequence of generalized polynomials  
determined by $\mathcal{X}$. 
Once we have \eqref{eq-higher-order-expansion-inequality-mu} for any $\ell \ge 1$ and some
$\mu \in (\mu_\ell,\mu_{\ell+1})$, we have \eqref{eq-higher-order-expansion-inequality} for any $\ell$
since the polynomial degree of $P_{\ell+1}(\mathcal X)$ is at most $\ell$,
by Proposition \ref{prop-polynomial-coefficients-new}.

We will prove \eqref{eq-higher-order-expansion-inequality-mu} by induction.
First, by assumption, \eqref{eq-exp-convergence} holds for some positive constants $C_*$ and $\gamma$. 
Replacing $\gamma$ by a smaller constant if necessary, we assume that 
$2^k\gamma < \rho_1 < 2^{k+1}\gamma < 2\rho_1$ for some $k \ge 1$. 
Now, $F(\varphi) = O(e^{-2\gamma t})$ by \eqref{eq-estimate-F}. From our assumption on $k$, $2 \gamma < \rho_1$. 
Therefore, by Lemma \ref{lemma-same-decay-estimate-for-L}(i), $\varphi = O(e^{-2\gamma t})$, 
and hence $F(\varphi) = O(e^{-4 \gamma t})$. 
A simple iterative argument yields
$\varphi = O(e^{-2^k \gamma t})$ and $ F(\varphi) = O(e^{-2^{k+1} \gamma t})$.
Then, by Lemma \ref{lemma-same-decay-estimate-for-L}(ii), 
there exists a spherical harmonic $X_1 \in \mathscr{E}_1$ such that
$$\varphi = X_1(\theta) p_1^+(t) e^{-\rho_1 t} + O(e^{-2^{k+1} \gamma t}).$$
This is \eqref{eq-higher-order-expansion-inequality-mu}
for $\ell = 1$ with $\mu = 2^{k+1}\gamma$. 
We point out that this implies $F(\varphi) = O(e^{-2\rho_1 t})$.

Now, assume \eqref{eq-higher-order-expansion-inequality-mu} holds for some $\ell\ge 1$ and 
some generalized polynomials $P_1, \cdots, P_\ell $ 
as in Definition \ref{def-sequence-P-determined-by-X}. 
For each $1\le i\le \ell $, $P_i$ satisfies 
\eqref{eq-recursive-equation-Asymp-expans}, with $i$ replacing $\ell$ there. 
For simplicity, 
consider $\varphi_\ell$ determined by 
$$\varphi =\varphi_\ell+ \sum_{i=1}^\ell P_i(t,\theta) e^{-\mu_i t}.$$ 
By the induction hypothesis, we get $\varphi_\ell = O(e^{-\mu t})$ for some
$\mu_\ell < \mu < \mu_{\ell+1}$. 
Now substitute $\varphi$ as above in \eqref{eq-linearized-Yamabe-at-xi}. 
With \eqref{eq-RHS-expanded-K}, we have 
\begin{equation}\label{eq-bootstrapping-eq}
\mathcal L \varphi_\ell + \sum_{i=1}^\ell \mathcal L (P_i(t,\theta) e^{-\mu_i t})=F,
\end{equation}
where 
\begin{align}\label{eq-RHS-expanded-K-2}
F =  \sum_{j = 2}^K a_j\Big(\varphi_\ell + \sum_{i=1}^\ell P_i(t,\theta) e^{-\mu_i t}\Big)^j + F_K.
\end{align}
Here we choose sufficiently large $K>\ell$. 
Note that 
$$a_j\Big(\varphi_\ell + \sum_{i=1}^\ell P_i(t,\theta) e^{-\mu_i t}\Big)^j 
= \sum_{s = 0}^j \left(\begin{matrix}j\\s\end{matrix}\right) a_j \varphi_\ell^s \cdot
\Big(\sum_{i=1}^\ell P_i(t,\theta)e^{-\mu_i t}\Big)^{j-s} .$$
The summation corresponding to $s\ge 1$ has a decay order $e^{-(\mu_\ell + \rho_1 + \epsilon) t}$
for any $\epsilon > 0$ small, and hence an order $e^{-(\mu_{\ell+1}+\epsilon) t}$
by $\mu_{\ell+1} \leq \mu_\ell +\rho_1$.
By expanding the expression for $s = 0$ and
substituting these estimates in \eqref{eq-RHS-expanded-K-2}, we get 
$$F = \sum_{i = 2}^{\ell+1}\Big(\sum_{j = 2}^i \sum_{\mu_{k_1} + \cdots + \mu_{k_j} = \mu_i} 
a_j P_{k_1} \cdots P_{k_j} \Big)e^{-\mu_i t} 
+ O(e^{-(\mu_{\ell+1} + \epsilon) t}),$$
for any $\epsilon>0$ small. 
Since $P_i$ satisfies 
\eqref{eq-recursive-equation-Asymp-expans}, with $i$ replacing $\ell$ there for each $1\le i\le \ell$, 
we can reduce \eqref{eq-bootstrapping-eq} to 
\begin{align*}
\mathcal L \varphi_\ell = \sum_{j = 2}^{\ell+1} \sum_{\mu_{k_1} + \cdots + \mu_{k_j} 
= \mu_{\ell+1}} a_j P_{k_1} \cdots P_{k_j} e^{-\mu_{\ell+1} t} + O(e^{-(\mu_{\ell+1} + \epsilon) t}).
\end{align*}
By Lemma \ref{lemma-polynomial_solution}, let $Q_{\ell+1}$ be the generalized polynomial 
given by 
\begin{align*}
Q_{\ell+1} 
=\mathcal L_{\mu_{\ell+1}}^{-1}\Big[ \sum_{j = 2}^{\ell+1} \sum_{\mu_{k_1} + \cdots + \mu_{k_j} = \mu_{\ell+1}} 
a_j P_{k_1} \cdots P_{k_j}\Big].
\end{align*} 
Then,  
$$\mathcal L\big(\varphi_\ell - Q_{\ell+1}(t,\theta) e^{-\mu_{\ell+1} t}\big) = O(e^{-(\mu_{\ell+1} + \epsilon )t}),$$ 
for any $\epsilon>0$ small. 
We choose $\epsilon$ so that $(\mu_{\ell+1},\mu_{\ell+1} + \epsilon]$ does not contain any $\rho_i$ for $i \geq 1$, 
and let $m$ be the largest integer such that $\rho_m < \mu_{\ell+1}+\epsilon$. 
By Lemma \ref{lemma-same-decay-estimate-for-L}(ii), 
there exist spherical harmonics $Y_i \in \mathscr{E}_i$ for $i = 1,\cdots,m$,
such that 
$$\varphi_\ell - Q_{\ell+1} (t,\theta) e^{-\mu_{\ell+1} t} - \sum_{i = 1}^m Y_i(\theta)p_i^+(t) e^{-\rho_i t} 
= O(e^{-(\mu_{\ell+1} +\epsilon) t}).$$ 
We first consider $\mu_{\ell+1} \neq \rho_i$ for any $i \geq 1$. Then, $\rho_m \leq \mu_\ell$. 
Since $Q_{\ell+1}(t,\theta)e^{-\mu_{\ell+1} t}$ and $\varphi_\ell$ are $ O(e^{-(\mu_\ell + \epsilon) t})$, 
we must have $Y_1, \cdots, Y_m = 0$. Set $P_{\ell+1}=Q_{\ell+1}$. 
Next, we consider $\rho_m=\mu_{\ell+1}$. Then, the same argument shows that $Y_1, \cdots, Y_{m-1} = 0$. 
Set $X_m = Y_m$ and
$$P_{\ell+1}(t,\theta)  = X_m(\theta)p_m^+(t) + Q_{\ell+1}(t,\theta).$$ 
In both cases, we obtain
$$\varphi_\ell - P_{\ell+1}(t,\theta) e^{-\mu_{\ell+1} t} = O(e^{-(\mu_{\ell+1} +\epsilon) t}),$$ 
for some $\epsilon > 0$ small. Finally, write 
$$\varphi_{\ell+1} = \varphi_\ell - P_{\ell+1} (t,\theta)e^{-\mu_{\ell+1} t} 
= \varphi - \sum_{i=1}^{\ell+1} P_i(t, \theta) e^{-\mu_i t}.$$ 
Then, $\varphi_{\ell+1}= O(e^{-(\mu_{\ell+1} +\epsilon) t})$ for some $\epsilon>0$ small. 
This finishes the proof of \eqref{eq-higher-order-expansion-inequality-mu} by induction. 
\end{proof}

\begin{definition}\label{def-determined-behavior}  
A solution $\varphi$ to \eqref{eq-linearized-Yamabe-at-xi} on $[t_0, \infty)\times \mathbb{S}^{n-1}$ is said to 
have its {\it asymptotic behavior as $t\to\infty$ determined by $\mathcal{X} = \{X_i\}_{i \geq 1}$} with $X_i \in \mathscr E_i$
if $\varphi$ satisfies \eqref{eq-higher-order-expansion-inequality}, 
for the sequence of generalized polynomials $\{P_\ell(\mathcal{X})\}_{\ell \geq 1}$ 
determined by $\mathcal{X}$ as in Definition \ref{def-sequence-P-determined-by-X}. 
\end{definition} 

Theorem \ref{thrm-higher-order-expansion} asserts that 
any solution $\varphi$ of \eqref{eq-linearized-Yamabe-at-xi} 
with $\varphi$ converging to $0$ exponentially as $t\to\infty$, uniformly in $\theta \in \mathbb{S}^{n-1}$,
has its asymptotic behavior as $t\to\infty$ determined by some sequence $\{X_i\}_{i \geq 1}$ with $X_i \in \mathscr{E}_i$.
We will 
improve the estimate \eqref{eq-higher-order-expansion-inequality}
to convergence in Subsections \ref{subsec-analyticity-estimates} and \ref{subsec-classifications}, 
with main results given by 
Theorems \ref{thm:actual-solution-series-improved} and \ref{thrm-higher-order-expansion-improved}.

We will also study the converse existence problem. 
For a given sequence of spherical harmonics $\mathcal{X} = \{X_i\}_{i \geq 1}$ with $X_i \in \mathscr E_i$, 
find a solution $\varphi$ to \eqref{eq-linearized-Yamabe-at-xi} on some $[t_0, \infty)\times \mathbb{S}^{n-1}$ 
such that its asymptotic behavior as $t\to\infty$ is determined by $\mathcal{X}$
as in Definition \ref{def-determined-behavior}. 
In this formulation, the sequence of spherical harmonics $\mathcal{X} = \{X_i\}_{i \geq 1}$
is treated as ``free data" and the asymptotic behavior is considered as a ``boundary condition" near infinity. 
We will study the existence in Subsections \ref{subsec-Constructions} and 
\ref{subsec-classifications}, with main results given by Theorems \ref{thrm-convergence-of-the-expansion-from-optimal-term} 
and \ref{thrm-convergence-of-the-expansion-from-optimal-term-improved}.

\section{The Unique Continuation at Infinity}
\label{sec-unique-conti}

In 
Section \ref{sec-Asymp-Expansions}, 
we derived asymptotic expansions for exponential decay solutions of \eqref{eq-linearized-Yamabe-at-xi}. 
In Theorem \ref{thrm-higher-order-expansion}, the formal expansions associated with solutions 
provide good approximations. In this section, we prove the uniqueness of such 
formal expansions. 

First, we improve the first order estimate in 
\eqref{eq-higher-order-expansion-inequality} to derivatives of arbitrary orders.

\begin{theorem}
\label{thrm-first-decay-C^m-norm}
Assume $\varphi$ is a smooth solution 
of \eqref{eq-linearized-Yamabe-at-xi} on $[t_0,\infty)\times\mathbb{S}^{n-1}$, 
satisfying \eqref{eq-exp-convergence} for some $\gamma, C_*>0$.  
Then, for any $k\ge 0$ 
and any $(t,\theta) \in [t_0,\infty)\times \mathbb{S}^{n-1}$,
\begin{align}
\label{eq-first-decay-C^m-estimate}
|\nabla^k\varphi (t,\theta)| \leq C_k e^{-t},
\end{align}
for some constant $C_k > 0$.
\end{theorem}

\begin{proof}
By \eqref{eq-higher-order-expansion-inequality} with $\ell=1$, we have, 
for any $(t,\theta) \in [t_0,\infty)\times \mathbb{S}^{n-1}$,
\begin{align}
\label{eq-thrm-first-decay-C^m-norm-C^0-estimate}
|\varphi (t,\theta)| \leq C_0 e^{-t},
\end{align}
for some constant $C_0 > 0$. 
This is \eqref{eq-first-decay-C^m-estimate} with $k = 0$. 
In view of \eqref{eq-RHS-of-linearized-Yamabe}, set 
\begin{align}\label{eq-def-G-RHS-of-Yamabe}
G(\zeta,s) = -\frac{1}{4}n(n-2) \Big[ (\zeta+s)^{\frac{n+2}{n-2}} - \zeta^{\frac{n+2}{n-2}} 
- \frac{n+2}{n-2}\zeta^{\frac{4}{n-2}} s \Big].\end{align}
Then, $F(\varphi)(t,\theta) = G\big(\xi(t), \varphi(t,\theta)\big).$
Note that, for any $\epsilon > 0$, $G$ is smooth 
in $(\epsilon,\infty)\times(-\epsilon,\infty)$ and $G(\zeta,0) = \partial_s G(\zeta,0) = 0$.
In particular, we can write
\begin{align}
\label{eq-thrm-first-decay-C^m-norm-F(phi) = B phi}
F(\varphi)(t,\theta) = B(t,\theta) \varphi(t,\theta),
\end{align}
where
\begin{align*}
B = \int_0^1 \partial_s G(\xi,s\varphi) ds.
\end{align*}
For $t \geq t_0+1$, applying an interior estimate to \eqref{eq-linearized-Yamabe-at-xi}, we get, 
for any $t \geq t_0+1$ and $\theta \in \mathbb{S}^{n-1}$,
$$|\nabla\varphi (t,\theta)| \leq C \big\{|\varphi|_{L^\infty ((t-1,t+1)\times \mathbb{S}^{n-1})}  
+  |F(\varphi)|_{L^\infty ((t-1,t+1)\times \mathbb{S}^{n-1})} \big\},$$
for some constant $C>0$. By \eqref{eq-thrm-first-decay-C^m-norm-C^0-estimate} 
and \eqref{eq-thrm-first-decay-C^m-norm-F(phi) = B phi}, we have,  
for any $(t,\theta) \in [t_0+1,\infty)\times \mathbb{S}^{n-1},$
$$|\nabla\varphi (t,\theta)| \leq C_1 e^{-t},$$
for some constant $C_1 > 0$. 
This estimate obviously holds for $(t,\theta) \in [t_0,t_0+1]\times\mathbb{S}^{n-1}$.

The estimate \eqref{eq-first-decay-C^m-estimate} for $k \ge 2$ follows from induction and standard
interior $C^k$-estimates for elliptic equations.
\end{proof}

In the study of nonlinear equations, it is convenient to adopt certain Sobolev norms. 
Fix $m \ge 1$. 
Consider an $H^{m}$-inner product 
on $\mathbb S^{n-1}$ given by, for any $u,v \in H^{m}(\mathbb{S}^{n-1})$, 
\begin{align}
\label{eq-H^m-inner-product-with-constant}
\langle u,v \rangle_{H^{m}(\mathbb{S}^{n-1})} 
=  C(m,n) \sum_{i=0}^{m} \langle \nabla_\theta^i u, \nabla_\theta^i v\rangle_{L^2(\mathbb{S}^{n-1})},
\end{align}
for some positive constant $C(m,n)$ to be chosen. 
It is known that $\Delta_{\theta}$ is self-adjoint in $H^{m}(\mathbb{S}^{n-1})$. 
More specifically, for any $u,v \in C^\infty(\mathbb{S}^{n-1})$, 
$$\langle \Delta_{\theta} u, v\rangle_{H^{m}(\mathbb{S}^{n-1})} = \langle  u,  \Delta_{\theta} v\rangle_{H^{m}(\mathbb{S}^{n-1})}.$$
In fact, for any $u,v \in C^\infty(\mathbb{S}^{n-1})$,
\begin{align}
\label{eq-formula-H^m-2-point-homogeneous}
\langle u,v\rangle_{H^m(\mathbb S^{n-1})} = C(m,n) \langle p(\Delta_\theta) u, v\rangle_{L^2(\mathbb S^{n-1})},
\end{align}
for some polynomial $p$.
For example, $p(t) = 1-t$ for $m = 1$ and
$p(t) = t^2 + (n-3)t + 1$ for $m = 2$.
Refer to Theorem 11\cite{Helgason1959}.
We should point out that the Laplace-Beltrami operator may not be self-adjoint on
$C^\infty(M)$ with respect to $\langle\cdot,\cdot\rangle_{H^m(M)}$, $m \ge 2$,
for a general Riemannian manifold $(M,g)$.

By the Sobolev embedding, if $m > \frac{n-1}{2}$,
we have the Banach algebra property. 
We now choose the constant $C(m,n) > 1$ such that, for any $u,v \in H^{m}(\mathbb{S}^{n-1})$, 
\begin{align}
\label{eq-Banach-algebra-property-without-constant}
\|u v\|_{H^{m}(\mathbb{S}^{n-1})} \leq \|u\|_{H^{m}(\mathbb{S}^{n-1})}
\|v\|_{H^{m}(\mathbb{S}^{n-1})}.
\end{align}
By \eqref{eq-formula-H^m-2-point-homogeneous}, we have 
$\langle X,Z\rangle_{H^{m}(\mathbb{S}^{n-1})} = 0$ if and only if $\langle X,Z\rangle_{L^2(\mathbb S^{n-1})} = 0$, 
for any $X \in \mathscr E_i$
and $Z \in C^\infty(\mathbb S^{n-1})$. 
For brevity, we write, for any $u,v \in H^{m}(\mathbb{S}^{n-1})$,  
\begin{align}
\label{eq-H^*1-inner-product}
\langle u, v\rangle_m = \langle u,v\rangle_{H^{m}(\mathbb S^{n-1})},\quad
\|u\|_{m}=\|u\|_{H^{m}(\mathbb{S}^{n-1})}.
\end{align}
Throughout this paper, we fix an integer
\begin{align}
\label{eq-choice-of-m_*}
m_* > 3+ \frac{n-1}{2}.
\end{align}



Next, we derive a uniform estimate for the linearized equation.
It is a modified version of Lemma \ref{lemma-same-decay-estimate-for-L}, 
and the proof is adapted from \cite{Han2019AsymptoticEO}. 
We first introduce some notations. 
Recall the lower and upper bounds for $\rho_i$ in \eqref{eq-lower-upper-bound-rho_i-no-xi}. 
For simplicity, we set, for any $i \ge 2$,
\begin{align}
\label{eq-tilde-rho_i}
\widetilde \rho_i^2 = \Big(i+\frac12(n-2)\Big)^2 - \frac{1}{4}n(n+2) 
= \lambda_i - \frac{1}{4}n(n+2) + \frac{1}{4}(n-2)^2.
\end{align}
Then, $\widetilde \rho_i<\rho_i$ for any $i\ge 2$ by \eqref{eq-lower-upper-bound-rho_i-no-xi} 
and $\rho_i < \widetilde \rho_{i+1}$ for any $i \geq c_n$ for some $c_n>0$, depending only on $n$. 

\begin{lemma}
\label{lemma-same-decay-estimate-for-L-explicit-constants}
Let  $\{\psi_i^+,\psi_i^-\}_{i\geq 0}$ be a sequence of admissible bases for $\{\mathrm{Ker}(L_i)\}_{i\geq 0}$
as in Lemma \ref{lemma-admissible-sequence-of-bases}. 
Take $\gamma \in (\rho_\ell, \widetilde\rho_{\ell+1})$ for some $\ell \geq c_n$.
Assume $f$ is a smooth function on $[t_0,\infty)\times\mathbb{S}^{n-1}$, for some $t_0\ge0$, satisfying, for any $t \geq t_0$,
\begin{align}
\label{eq-decay-for-nonhomogeneous-term-f}
    \|f(t,\cdot)\|_{m_*} \leq C_0 e^{-\gamma t},
\end{align}
for some constant $C_0>0$.
Let $\varphi$ be a 
solution of
\eqref{eq-linearization-PDE-nonconstant} on $[t_0,\infty)\times \mathbb{S}^{n-1}$
such that 
\begin{align}\label{eq-linear-uniform-i}\|\varphi(t_0,\cdot)\|_{m_*} \leq C_1 e^{-\alpha t_0},\end{align}
for some constants $C_1>1$ and $0 <\alpha < \gamma$, 
and, for any $t \geq t_0$, 
\begin{align}\label{eq-linear-uniform-ii}\|\varphi(t,\cdot)\|_{m_*} \leq C_\nu e^{-\nu t},\end{align}
for some constants $\nu \in (\rho_\ell,\gamma)$ and $C_\nu > 0$.  
Then, for any $t \geq t_0$,
\begin{align}
\label{eq-explicit-same-decay}
    \|\varphi(t,\cdot)\|_{m_*} \leq C_*\max\{1, (\gamma - \rho_\ell)^{-1}, (\widetilde \rho_{\ell+1}^2 - \gamma^2)^{-1} , 
    e^{(\gamma - \alpha)t_0}\} \max\{C_0,C_1\} e^{-\gamma t},
\end{align}
for some constant $C_*>0$ depending only on $n$, $\xi$,  
and the bases $\{\psi_i^+,\psi_i^-\}_{i\geq 0}$.
\end{lemma}

\begin{proof}
For each $i \ge 0$, write $d_i = \dim \mathscr{E}_i$ and let $\{Y_{ij}\}_{j=1}^{d_i}$ be a basis of 
$\mathscr{E}_i$ that is orthonormal with respect to $\langle\cdot,\cdot\rangle_{m_*}$. 
Then, $\{Y_{ij}\}$ forms an orthonormal basis of $H^{m_*}(\mathbb{S}^{n-1})$ 
with respect to $\langle\cdot,\cdot\rangle_{m_*}$. 
Now, for any $i \ge 0$ and $j = 1,\cdots,d_i$, define
$$\varphi_{ij}(t) = \langle \varphi(t,\cdot),Y_{ij} \rangle_{m_*}, \,\,\,\,\,\,f_{ij}(t) = \langle f(t,\cdot),Y_{ij}\rangle_{m_*}.$$
By differentiating under integral signs and \eqref{eq-linearization-PDE-nonconstant}, 
we get, for any $i \ge 0$ and $j = 1,\cdots,d_i$,  
$$L_i \varphi_{ij} = f_{ij}.$$

Now, by standard ODE results, there exist constants $a_{ij},b_{ij}$ for $0\le i \le \ell$ and $j = 1,\cdots,d_i$, 
such that,  for $0 \le i \le \ell$ and $j = 1,\cdots,d_i$,
$$\varphi_{ij} = \varphi_{ij*} + a_{ij} \psi_i^+ + b_{ij} \psi_i^-,$$
where, for $i = 0$ and $j=1$, 
\begin{align}
\label{eq-phi_01}\begin{split}
\varphi_{01*}(t)  &= \frac{p_0^+(t)}{W_0}\int_t^\infty p_0^-(s) f_{01}(s) ds 
- \frac{p_0^-(t)}{W_0}\int_t^\infty p_0^+(s) f_{01}(s) ds\\
&\,\,\,\,\,\,\,+\frac{c p_0^+(t)}{W_0} \int_{t}^\infty \int_s^\infty p_0^+(\tau) f_{01}(\tau)  d\tau ds,     
\end{split}\end{align}
and, for $1 \le i \le \ell$ and $j = 1,\cdots,d_i$,
\begin{align}
\label{eq-phi_ij}
\varphi_{ij*}(t) = \frac{\psi_i^+(t)}{W_i}\int_{t}^\infty \psi_i^-(s) f_{ij}(s) ds 
-  \frac{\psi_i^-(t)}{W_i}\int_{t}^\infty \psi_i^+(s) f_{ij}(s) ds.
\end{align}
Fix $0 \leq i \leq \ell$ and $j = 1,\cdots,d_i$. By \eqref{eq-decay-for-nonhomogeneous-term-f}, 
$\varphi_{ij*} = O(e^{-\gamma t})$; and by \eqref{eq-linear-uniform-ii},
$\varphi_{ij} =O( e^{-\nu t})$. 
It is easy to check $a_{ij} = b_{ij} = 0$, and hence 
$\varphi_{ij} = \varphi_{ij*}$ for $0 \leq i \leq \ell$ and $1 \leq j \leq d_i$. Now, write 
$$\varphi = \sum_{i=0}^\ell \sum_{j=1}^{d_i} \varphi_{ij}(t) Y_{ij}(\theta) + \widehat\varphi(t,\theta).$$

We first estimate the summation. 
Fixing $t \geq t_0$, since $\{Y_{ij}\}$ is orthonormal with respect to $\langle\cdot,\cdot\rangle_{m_*}$, we have
$$\Big\|\sum_{i=0}^\ell \sum_{j=1}^{d_i} \varphi_{ij}(t) Y_{ij}(\cdot)  \Big\|_{m_*} 
= \Big( \sum_{i=0}^\ell \sum_{j=1}^{d_i} \varphi_{ij}(t)^2 \Big)^{\frac{1}{2}}.$$
By \eqref{eq-decay-for-nonhomogeneous-term-f} and \eqref{eq-phi_01}, it is obvious that 
\begin{align}
\label{eq-the-same-decay-finite-part-0}
|\varphi_{01}(t)| \leq C C_0 e^{-\gamma t},
\end{align}
for some universal constant $C >0$. We point out that $\gamma >1$. Next, for simplicity, we write
$$\sum_{i,j} = \sum_{i=1}^\ell \sum_{j=1}^{d_i}.$$
By \eqref{eq-phi_ij} and the Minkowski integral inequality, we have
\begin{align*}
     \Big(\sum_{i,j} \Big( \int_t^\infty \frac{1}{W_i} p_i^+(t) p_i^-(s)e^{\rho_i(s-t)}  f_{ij}(s) ds \Big)^2 \Big)^{\frac{1}{2}}
     \leq C \int_t^\infty \Big( \sum_{i,j}  e^{2\rho_{i} (s-t)}f_{ij}(s)^2 \Big)^{\frac{1}{2}} ds,
\end{align*}
for some universal constant $C>0$. Note that $s-t \geq 0$ and
$$\Big(\sum_{i,j} f_{ij}(s)^2 \Big)^{\frac{1}{2}} \leq \|f(s,\cdot)\|_{m_*}.$$
Hence, the last integral is bounded by
\begin{align}
\label{eq-the-same-decay-finite-part-first-half}
 \int_t^\infty e^{\rho_\ell (s-t)} \|f(s,\cdot)\|_{m_*} ds\leq C_0 (\gamma - \rho_\ell)^{-1} e^{-\gamma t}.
\end{align}
Similarly, we have
\begin{align}
\label{eq-the-same-decay-finite-part-second-half}
     \Big(\sum_{i,j} \Big( \int_t^\infty \frac{1}{W_i} p_i^-(t) p_i^+(s)e^{\rho_i(t-s)}  f_{ij}(s) ds \Big)^2 \Big)^{\frac{1}{2}}
     \leq C_0 C  e^{-\gamma t},
\end{align}
for some universal constant $C>0$. We point out that $t-s \leq 0$ in this case. 
By \eqref{eq-the-same-decay-finite-part-0}, \eqref{eq-the-same-decay-finite-part-first-half}, 
and \eqref{eq-the-same-decay-finite-part-second-half}, we have
\begin{align}
\label{eq-thrm-explicit-same-decay-estimate-for-the-finite-part}
    \Big\|\sum_{i=0}^\ell \sum_{j=1}^{d_i} \varphi_{ij}(t) Y_{ij}(\cdot)  \Big\|_{m_*} 
    \leq C_* C_0 \max \{1,(\gamma - \rho_\ell)^{-1}\} e^{-\gamma t},
\end{align}
where $C_* > 0$ is a universal constant.

We next estimate the remainder $\widehat \varphi$. 
The proof is similar to the corresponding part in the proof of Lemma A.8\cite{Han2019AsymptoticEO}. 
We only point out the difference.
For simplicity, set
$$c(t) = \frac{1}{4}n(n+2)\xi^{\frac{4}{n-2}}(t) - \frac{1}{4}(n-2)^2. $$
Then,
$\mathcal L = \partial_{tt} + \Delta_\theta + c.$
Write
\begin{align*}
\widehat f = f - \sum_{i=0}^{\ell} \sum_{j=1}^{d_i} f_{ij}(t)Y_{ij}(\theta).
\end{align*}
It is straightforward to check that $\mathcal L \widehat \varphi = \widehat f.$
Set $y(t) = \|\widehat\varphi(t,\cdot)\|_{m_*}.$
Since $\Delta_\theta$ is self-adjoint with respect to $\langle\cdot,\cdot\rangle_{m_*}$, we have 
\begin{align*} 
\langle \Delta_\theta \widehat\varphi(t,\cdot),\widehat\varphi(t,\cdot)\rangle_{m_*} 
= -\sum_{i=\ell+1}^\infty \sum_{j=1}^{d_i} \lambda_i \varphi_{ij}(t)^2\le -\lambda_{\ell+1} y(t)^2.
\end{align*}
By proceeding similarly as in the proof of Lemma A.8\cite{Han2019AsymptoticEO}, we conclude that
if $y(t) > 0$, then 
$$L_{\ell+1} y \geq -C_0 e^{-\gamma t}.$$
A simple calculation yields
$$L_{\ell+1} (C e^{-\gamma t}) = C(\gamma^2 + c(t) - \lambda_{\ell+1}) e^{-\gamma t} 
\leq C(\gamma^2 - \widetilde\rho_{\ell+1}^2) e^{-\gamma t},$$
where $\widetilde \rho_{\ell+1}$ is given by \eqref{eq-tilde-rho_i}, with $i$ replaced by $\ell+1$. 
Hence, if $C\geq C_0(\widetilde\rho_{\ell+1}^2 - \gamma^2)^{-1}$, 
then $L_{\ell+1} (Ce^{-\gamma t} - y) \leq 0$ 
provided $y(t)>0$. On the other hand, by \eqref{eq-linear-uniform-i},
if $C\geq  C_1 e^{(\gamma-\alpha) t_0}$, 
then $C e^{-\gamma t_0} \geq y(t_0)$. Take
$$C = \max \{ C_0 (\widetilde\rho_{\ell+1}^2 - \gamma^2)^{-1}, C_1 e^{(\gamma - \alpha) t_0}\}.$$
We point out that $c(t) - \lambda_{\ell+1} < 0$. 
As in the proof of Lemma A.8\cite{Han2019AsymptoticEO},
an application of the maximum principle to $L_{\ell+1}$ yields
\begin{align}
\label{eq-thrm-explicit-same-decay-estimate-for-the-infinite-part}
\| \widehat \varphi (t,\cdot) \|_{m_*} = y(t) \leq C e^{-\gamma t}.
\end{align}
By \eqref{eq-thrm-explicit-same-decay-estimate-for-the-finite-part} 
and \eqref{eq-thrm-explicit-same-decay-estimate-for-the-infinite-part}, 
we conclude 
the desired estimate in \eqref{eq-explicit-same-decay}.
\end{proof}

Now, we prove the main result in this section. 

\begin{theorem}\label{thrm-uniqueness-exponential-decay-sol}
Let 
$\varphi, \widetilde\varphi$ be 
solutions of \eqref{eq-linearized-Yamabe-at-xi} 
on $[t_0,\infty) \times \mathbb{S}^{n-1}$, 
satisfying \eqref{eq-exp-convergence} for some $\gamma, C_*>0$.  
Assume that for any positive constant $\mu$, there exists a positive constant $C_{\mu}$ such that, 
for any $(t,\theta) \in  [t_0,\infty)\times \mathbb{S}^{n-1}$,  
\begin{align}\label{eq-infinite-vanishing-t}|\varphi(t,\theta) - \widetilde \varphi(t,\theta)| \leq C_{\mu} e^{- \mu t}.\end{align} 
Then, $\varphi = \widetilde \varphi$ in $[t_0,\infty) \times \mathbb{S}^{n-1}$. 
\end{theorem}

\begin{proof} 
Set $w = \varphi-\widetilde\varphi$. Then, 
$w$ satisfies
\begin{align}
\label{eq-linear-eq-unique-conti}
\mathcal Lw = B w \quad\text{on }[t_0,\infty)\times\mathbb{S}^{n-1},
\end{align}
where 
\begin{align}
\label{eq-expression-B-difference-linear-eq}
B = \int_{0}^1 \partial_s G(\xi,\widetilde \varphi + sw ) ds 
= \int_0^1 \int_0^1 \partial_s^2 G(\xi,\tau\widetilde\varphi + \tau  sw )d\tau (\widetilde \varphi + s w) ds,
\end{align}
and $G$ is given by \eqref{eq-def-G-RHS-of-Yamabe}.
By Theorem \ref{thrm-first-decay-C^m-norm},  
we have, for any $t \geq t_0$,
\begin{align}
\label{eq-uniqueness-exponential-decay-sol-estimate-for-B}
\|B(t,\cdot)\|_{m_*} \leq A e^{-t},
\end{align}
where $A\ge 1$ is a positive constant. 
By \eqref{eq-infinite-vanishing-t} and an argument similar to the 
proof of Theorem \ref{thrm-first-decay-C^m-norm}, we have
$|\nabla^k w| = O(e^{-\mu t})$ for any $k\ge 0$ and $\mu > 0$,
and in particular, for any $\mu > 0$,
\begin{align}
\label{eq-uniqueness-exponential-decay-sol-arbitrary-decay-Hm-norm}
\|w(t,\cdot)\|_{m_*} = O(e^{-\mu t}).
\end{align}
Hence, \eqref{eq-linear-uniform-ii} holds for $w$, for any $\nu>0$.

Let $\widetilde \rho_i$ be as in \eqref{eq-tilde-rho_i} for $i \geq 2$.
A simple computation shows that
$$\widetilde\rho_{k+1} - \rho_{k}\to 1\,\text{ and }\,   \rho_k - \widetilde \rho_k\to 0\quad\text{as }k\to\infty.$$
Note that, by \eqref{eq-lower-upper-bound-rho_i-no-xi}, $\rho_\ell/\ell \to 1$ as $\ell\to \infty$. 
In particular, $\ell$ is large if $\rho_\ell$ is large. 
Hence, there exists $k_0>0$ such that, if $\rho_{\ell+1} > k_0/2$, then $\rho_\ell < \widetilde\rho_{\ell+1}$, 
$\widetilde\rho_{\ell+1} - \rho_{\ell} \geq \frac{9}{10}$, and $\rho_\ell - \widetilde \rho_\ell <\frac{1}{10}$.
By \eqref{eq-uniqueness-exponential-decay-sol-arbitrary-decay-Hm-norm}, we get, for any $t \geq t_0$, 
$$\|w (t, \cdot)\|_{m_*} \leq A_0 e^{-\frac{k_0}{2} t},$$
for some constant $A_0 \ge1$.
Hence, by \eqref{eq-Banach-algebra-property-without-constant} 
and \eqref{eq-uniqueness-exponential-decay-sol-estimate-for-B}, we obtain, for any $t \geq t_0$,  
$$\|B(t,\cdot)w(t,\cdot)\|_{m_*} \leq \|B(t,\cdot)\|_{m_*}\|w(t,\cdot)\|_{m_*} \leq A_0A e^{-(\frac{k_0}{2}+1) t}.$$
Next, consider $\frac{2k_0 + 3}{4} \in (\rho_\ell,\rho_{\ell+1}]$ for some $\ell$. 
By our choice of $k_0$, we have 
$\widetilde\rho_{\ell+1} - \rho_\ell \geq \frac{9}{10},$ and 
$\rho_{\ell+1} - \widetilde\rho_{\ell+1} \leq \frac{1}{10}.$ 
We claim that there exists a constant 
$\gamma \in (\rho_\ell,\widetilde\rho_{\ell+1})$ such that
\begin{align}
\label{eq-gamma-condition}
\frac{k_0+1}{2} \leq \gamma \leq \frac{k_0}{2}+1, \quad  
\min \{\gamma - \rho_\ell, \widetilde\rho_{\ell+1} - \gamma\} \geq \frac{1}{10}.
\end{align}
First, assume $\frac{2k_0 + 3}{4} - \rho_\ell \leq \frac{1}{5}$. Then, for  $\gamma = \frac{2k_0 + 3}{4} +\frac{1}{10}$,
$$\widetilde\rho_{\ell+1} - \gamma = (\widetilde\rho_{\ell+1} - \rho_\ell) 
-\Big(\frac{2k_0 +3}{4} + \frac{1}{10} -  \rho_\ell  \Big) \geq \frac{6}{10},$$
and
$\gamma - \rho_\ell \geq \frac{1}{10}$.
Second, consider $\rho_{\ell+1} - \frac{2k_0 + 3}{4} \leq \frac{1}{5}$. 
Then, for $\gamma = \frac{2k_0 + 3}{4} - \frac{1}{5},$
$$\widetilde \rho_{\ell+1} - \gamma = \Big(\rho_{\ell+1} - \frac{2k_0+3}{4} + \frac{1}{5}\Big) 
- (\rho_{\ell+1} - \widetilde \rho_{\ell+1}) \geq \frac{1}{5} - \frac{1}{10} = \frac{1}{10},$$
and 
$$\gamma - \rho_\ell  = (\rho_{\ell+1} - \rho_\ell) - \Big( \rho_{\ell+1} - \frac{2k_0+3}{4} +\frac{1}{5} \Big) \geq \frac{5}{10}.$$
Note that $\rho_{\ell+1} > \widetilde \rho_{\ell+1}$.
Finally, suppose $\rho_\ell + \frac{1}{5} \leq \frac{2k_0 + 3}{4} \leq \rho_{\ell+1}-\frac{1}{5}$. 
Then, for $\gamma = \frac{2k_0+3}{4}$,
$$\widetilde \rho_{\ell+1} - \gamma = \Big( \rho_{\ell+1} - \frac{2k_0+3}{4} \Big) 
- (\rho_{\ell+1} - \widetilde \rho_{\ell+1}) \geq \frac{1}{5} - \frac{1}{10} = \frac{1}{10},$$
and
$\gamma - \rho_\ell \geq \frac{1}{5}.$
This finishes the proof of \eqref{eq-gamma-condition}.

By applying Lemma \ref{lemma-same-decay-estimate-for-L-explicit-constants}
to \eqref{eq-linear-eq-unique-conti}, 
with $C_0 = A_0 A$, $C_1 = A_0$, and $\alpha = \frac{k_0}{2}$, we have, for any $t \geq t_0$, 
$$\|w(t,\cdot)\|_{m_*} \leq C_* \max \{1,10, e^{t_0}\} A_0 A e^{-\gamma t}. $$
We point out that $C_0 \geq C_1$, $\min \{\gamma - \rho_\ell, \widetilde\rho_{\ell+1} - \gamma\} \geq \frac{1}{10}$, 
and $\gamma - \alpha \leq 1$. Since $\gamma \geq \frac{k_0+1}{2}$, we obtain, 
for any $(t,\theta) \in [t_0,\infty)\times \mathbb{S}^{n-1}$,
$$\|w(t,\cdot)\|_{m_*} \leq A_0 e^{-\frac{k_0}{2}t} (C_* A e^{t_0}) e^{-\frac{t}{2}}.$$

We now claim, for any $k \geq 1$ and any $(t,\theta) \in [t_0,\infty)\times \mathbb{S}^{n-1}$,
\begin{equation}\label{eq-uniqueness-induction}
\|w(t,\cdot)\|_{m_*} \leq A_0 e^{-\frac{k_0}{2} t} (C_* A e^{t_0})^k e^{-\frac{k}{2} t}.
\end{equation} 
By induction, assume \eqref{eq-uniqueness-induction} holds for $k-1$. Then, for $t \geq t_0$,
$$\|B(t,\cdot)w(t,\cdot)\|_{m_*} \leq A_0 e^{-\frac{k_0}{2} t} (C_*e^{t_0})^{k-1} A^k e^{-(\frac{k-1}{2}+1) t}.$$ 
Again, we consider $\frac{k_0 + k-1 }{2} + \frac{3}{4} \in (\rho_{\ell}, \rho_{\ell+1}]$ for some $\ell$, 
and choose $\gamma \in (\rho_\ell,\widetilde\rho_{\ell+1})$ such that 
$$\frac{k_0 + k}{2} \leq \gamma \leq \frac{k_0+k+1}{2}, \quad
\min \{\gamma - \rho_\ell, \widetilde\rho_{\ell+1} - \gamma\} \geq \frac{1}{10}.$$ 
The proof is similar to that of \eqref{eq-gamma-condition}. 
Another application of Lemma 
\ref{lemma-same-decay-estimate-for-L-explicit-constants}
to \eqref{eq-linear-eq-unique-conti}, 
with $C_0 = A_0 (C_*e^{t_0})^{k-1} A^k$, $C_1 = A_0 (C_*Ae^{t_0})^{k-1}$, and $\alpha = \frac{k_0+k-1}{2}$, 
yields \eqref{eq-uniqueness-induction} for such $k$. 

Take any $t\ge t_0$ large enough so that $ C_* A e^{t_0} e^{-\frac{t}{2}}<1$. 
For such $t$, letting $k \to \infty$ in \eqref{eq-uniqueness-induction}, we conclude $w(t,\theta) = 0$ 
for any $\theta \in \mathbb{S}^{n-1}$. Now, $w$ is a solution to the linear elliptic equation
$$\partial_{tt} w + \Delta_\theta w + \frac{1}{4}n(n+2)\xi^{\frac{4}{n-2}} w - \frac{1}{4}(n-2)^2 w - B w = 0 
\quad\text{on } [t_0,\infty)\times \mathbb{S}^{n-1},$$
and it vanishes in an open subset of $[t_0,\infty)\times \mathbb{S}^{n-1}$. 
By the standard unique continuation of solutions to linear elliptic equations, 
we conclude that $w = 0$ in $[t_0,\infty)\times \mathbb{S}^{n-1}$ and hence 
$\varphi = \widetilde \varphi$ in $[t_0,\infty)\times \mathbb{S}^{n-1}$.
\end{proof}

Theorem \ref{thrm-uniqueness-exponential-decay-sol} is interpreted as the unique continuation 
at infinity.
Theorem \ref{thrm-unique-expansion} follows from 
Theorem \ref{thrm-uniqueness-exponential-decay-sol} easily.


\section{Reparametrizations}\label{sec-Reparametrizations}

In this section, we reparameterize the index set in Definition \ref{def-index-set}
and the sequence of generalized polynomials in Definition \ref{def-sequence-P-determined-by-X}
by multi-indices of the infinite dimension. 

\smallskip

Define
\begin{align*}
\ind = \{\alpha=(\alpha_1, \alpha_2, \cdots);\, \alpha_j\in \Z_+, \alpha_j=0\text{ for all but finitely many }j, \alpha\neq0\}.
\end{align*}
For $i\ge 1$, denote by $e_i \in \ind$ whose $i$-th component is $1$ 
and all other components are $0$. Then, 
$$\ind = \Z_+\text{-}\text{span}\{ e_i;\,i\geq 1\} \setminus\{0\}.$$
For brevity, we write $\alpha = \{\alpha_i\}_{i\geq 1}$ and view it as a sequence. 
We emphasize that there are only finitely many nonzero components in each $\alpha \in \ind$. 
This is the reason we introduce a new notation $\ind$ instead of using $\Z_+^{\infty}\setminus\{0\}$.

For any $\alpha \in \ind$ and any real-valued sequence $x = \{x_i\}_{i\geq 1} \in \R^\infty$,  we write
\begin{align*}
|\alpha| = \sum_{i=1}^\infty \alpha_i, \quad
\langle \alpha, x \rangle = \sum_{i=1}^\infty \alpha_i x_i,\quad
\langle \alpha,\N\rangle = \sum_{i=1}^\infty i \alpha_i,
\end{align*}
where $\N$ is viewed as the increasing sequence of positive integers.
There are finitely many terms in these sums. 
For the sequence $\{\rho_i\}_{i\ge 1}$, 
we write 
$$\rho = \{\rho_i\}_{i \geq 1}.$$
Now, we have 
$\mathcal I=\{\langle\alpha,\rho\rangle;\, \alpha\in \ind\}$ for the index set  in Definition \ref{def-index-set}. 

At first glance, it seems more complicated to parameterize the index set 
by infinite dimensional multi-indices as in $\{\langle\alpha,\rho\rangle\}_{\alpha\in \ind}$
than by positive integers as in $\{\mu_\ell\}_{\ell\ge 1}$. 
It turns out that such a reparameterization has significant advantages. 

We now prove two crucial results in the form of multi-indices. 
These results play important roles in our study in later sections
and are the key elements that allow us to characterize exponential decay solutions of \eqref{eq-linearized-Yamabe-at-xi}. 
In the next result, the strict subadditivity in Proposition \ref{prop:SA} is essentially used. 

\begin{lemma}
\label{lemma-comparison-of-mu-and-rho}
There holds, for any $\alpha \in \ind$ with $\alpha_1 = 0$ and $|\alpha| \ge 2$,
\begin{align}\label{eq-estimate-mu-rho}
\langle \alpha,\rho \rangle - {\rho}_{\langle\alpha,\N\rangle} \geq \epsilon_*  |\alpha|,
\end{align}
where $\epsilon_*$ is a positive constant depending only on $n$ and $\xi$. 
\end{lemma}

\begin{proof} 
We will use Lemma \ref{lem:delta-Delta}(i) repeatedly. 
For $\delta_j=\rho_j-j$ in \eqref{eq:delta-Delta-def}, 
we have $\delta_1=0$, $0<\delta_2\le \delta_j< a$ for any $j\ge 2$, 
$\{\delta_j\}_{j\ge 1}$ is strictly increasing, and $\delta_j\to a$ as $j\to\infty$. 
To be consistent with the notations $\alpha$ and $\rho$, we write $\delta= \{\delta_j\}_{j\ge 1}$.

By $\rho_j=j+\delta_j$, we get   
\begin{equation}\label{eq:divisor-defect-form}
  \langle\alpha,\rho\rangle-\rho_{\langle\alpha,\N\rangle}
  =\langle\alpha,\delta\rangle-\delta_{\langle\alpha,\N\rangle}.
\end{equation}
Let $\alpha\in\ind$ be as given; namely, $\alpha_1 = 0$ and $|\alpha|\ge 2$. 
Then, $\langle\alpha,\delta\rangle=\alpha_2\delta_2+\alpha_3\delta_3+\cdots$. 
In the sum, each $\delta_j$ appears $\alpha_j$ times. 
By 
Proposition \ref{prop:SA}, we have 
$$\alpha_2\delta_2+\alpha_3\delta_3+\cdots>\delta_{2\alpha_2+3\alpha_3+\cdots},$$
or simply $\langle\alpha,\delta\rangle-\delta_{\langle\alpha,\N\rangle}  >0$.
It remains to estimate the lower bound of such a difference.

With $\delta_{j}\ge\delta_2>0$ for any $j\ge 2$ and $\delta_{\langle\alpha,\N\rangle}<a$, 
\eqref{eq:divisor-defect-form} implies 
\begin{equation*}
  \langle\alpha,\rho\rangle-\rho_{\langle\alpha,\N\rangle}
  \ge \delta_2|\alpha|-a.
\end{equation*}
Hence, if $|\alpha|\ge 2a/\delta_2$, we obtain
\begin{equation}\label{eq:large-N-bound}
  \langle\alpha,\rho\rangle-\rho_{\langle\alpha,\N\rangle}\ge 
  \frac{\delta_2}{2}|\alpha|.
\end{equation}

We next consider $\alpha \in \ind$ such that
\begin{equation}\label{eq:finite-range-N}
\alpha_1=0,\quad  2\le |\alpha|<\frac{2a}{\delta_2}.
\end{equation}
We claim that zero is not an accumulation point of the positive numbers
$\langle\alpha,\rho\rangle-\rho_{\langle\alpha,\N\rangle}$ for $\alpha \in \ind$ satisfying \eqref{eq:finite-range-N}. 
Suppose, to the contrary, 
that there exists a distinct sequence $\{\alpha^{(\ell)}\}_{\ell \ge 1}$, 
with $\alpha^{(\ell)}_1=0$ and $2\le |\alpha^{(\ell)}|<2a/\delta_2$,  such that
\begin{equation}\label{eq:contradict-seq}
  \langle\alpha^{(\ell)},\rho\rangle-\rho_{\langle\alpha^{(\ell)},\N\rangle}\to0.
\end{equation}
We first note that $\{\langle \alpha^{(\ell)}, \rho\rangle\}_{\ell \ge 1}$ is unbounded. 
This is because there would be only finitely many choices of $\alpha^{(\ell)}$ 
if $\{\langle \alpha^{(\ell)}, \rho\rangle\}_{\ell \ge 1}$ were bounded.
By passing to a subsequence, we assume $\langle \alpha^{(\ell)}, \rho\rangle\to\infty$. 
For each $\alpha^{(\ell)}$, take $j_\ell\ge 2$ such that 
$\alpha^{(\ell)}_{j_\ell}\neq 0$ and $\alpha^{(\ell)}_j=0$ for any $j>j_\ell$. 
In other words, $j_\ell$ is the last index $j$ with the nonvanishing $\alpha_j^{(\ell)}$ in $\alpha^{(\ell)}$. 
It is easy to check that $j_\ell\to\infty$. 
We now consider two cases. 
If $\alpha^{(\ell)}_{j_\ell} \ge 2$, 
we get
$$\langle \alpha^{(\ell)}, \delta \rangle \ge \alpha^{(\ell)}_{j_\ell} \delta_{j_\ell} \ge  \delta_2+\delta_{j_\ell}.$$
If $\alpha^{(\ell)}_{j_\ell} = 1$, 
there exists $k_\ell$, with $2\le k_\ell<j_\ell$, such that $\alpha^{(\ell)}_{k_\ell} \neq 0$, by $|\alpha^{(\ell)}| \ge 2$.
Hence,
$$\langle  \alpha^{(\ell)}, \delta \rangle 
\ge \alpha^{(\ell)}_{k_\ell} \delta_{k_\ell}+\alpha^{(\ell)}_{j_\ell} \delta_{j_\ell}
\ge \delta_2+\delta_{j_\ell}.$$
In both cases, with \eqref{eq:divisor-defect-form} for $\alpha=\alpha^{(\ell)}$, we have 
$$\langle \alpha^{(\ell)}, \rho\rangle - \rho_{\langle \alpha^{(\ell)}, \N\rangle} 
\ge \delta_2 + \delta_{j_\ell} - \delta_{\langle \alpha^{(\ell)},\N\rangle}.$$
Letting $\ell \rightarrow \infty$, we obtain
$$\liminf_{\ell\to\infty}\big(\langle \alpha^{(\ell)}, \rho\rangle - \rho_{\langle \alpha^{(\ell)}, \N\rangle}\big) \ge \delta_2,$$
contradicting \eqref{eq:contradict-seq}.
This proves the claim.

Consequently,  
there exists $\gamma_*>0$ such that, for any $\alpha \in \ind$ satisfying \eqref{eq:finite-range-N}, 
\begin{equation}\label{eq:small-N-bound}
\langle\alpha,\rho\rangle-\rho_{\langle\alpha,\N\rangle}\ge \gamma_*\ge\frac{\gamma_*\delta_2}{2a}|\alpha|.
\end{equation}
By combining \eqref{eq:large-N-bound} and \eqref{eq:small-N-bound}, 
we have \eqref{eq-estimate-mu-rho}, for every $\alpha$ with $\alpha_1=0$ and $|\alpha|\ge 2$.
\end{proof}

Note that $\{\rho_i\}_{i\ge 1}$ is strictly increasing by Lemma \ref{lemma-kernel-of-projection-ode}. 
Hence, for any $\alpha \in \ind$, with $\alpha_1 = 0$ and $|\alpha| \ge 2$,
and any $1\le i \le \langle\alpha,\N\rangle$, 
\begin{align}\label{eq-estimate-mu-rho-i}
\langle \alpha,\rho \rangle - {\rho}_{i} \geq \epsilon_*  |\alpha|.
\end{align}
The next example demonstrates that \eqref{eq-estimate-mu-rho-i} may not hold 
if nonzero $\alpha_1$ is allowed. 

\begin{example}\label{exa-small-divisors}
Let $\xi$ be the positive constant solution to \eqref{eq-ode-for-radial-singular-solution}
as in \eqref{eq-positive-constant-solu}. 
Recall that $\rho_1=1$. 
With $\alpha=\ell e_1$, we consider $\langle \alpha,\rho \rangle - {\rho}_{i}=\ell-\rho_i$ 
for $i\le \langle\alpha,\N\rangle=\ell$. 
Take $n=4j_0+2$ for some $j_0\ge 1$. 
By the explicit expressions of $\{\rho_i\}_{i\ge 1}$ given by \eqref{eq-i-Yamabe-linearization-constant-decay-rate} 
and a simple computation, we have $\ell^2-\rho_{\ell-2j_0}^2=4j_0(j_0+1)$ for any $\ell\ge 2j_0+1$.
Then, $\ell\neq\rho_{\ell-2j_0}$ for any $\ell\ge 2j_0+1$, but $\ell-\rho_{\ell-2j_0}\to0$ as $\ell\to\infty$. 
\end{example} 

In the next result, we examine the case that nonzero $\alpha_1$ is allowed.

\begin{lemma}
\label{lemma-<alpha,rho>=rho_i}
There exist at most finitely many pairs $(\alpha,i) \in \ind\times \N$ 
with $|\alpha|\ge 2$ and $1 \le i \le \langle\alpha,\N\rangle$ such that
\begin{align}
\label{eq-<alpha,rho>=rho_i}
\langle\alpha,\rho\rangle = \rho_i.
\end{align}
\end{lemma}

\begin{proof}  
We will use Lemma \ref{lem:delta-Delta}(i) repeatedly. 
For $\delta_j=\rho_j-j$ in \eqref{eq:delta-Delta-def}, 
we have $\delta_1=0$, $0<\delta_2\le \delta_j< a$ for any $j\ge 2$, 
$\{\delta_j\}_{j\ge 1}$ is strictly increasing, and $\delta_j\to a$ as $j\to\infty$. 
To be consistent with the notations $\alpha$ and $\rho$, we write $\delta= \{\delta_j\}_{j\ge 1}$.

If $i=1$, with $\rho_1=1$, we know that  
the only $\alpha$ satisfying \eqref{eq-<alpha,rho>=rho_i} 
is given by $\alpha=e_1$, contradicting $|\alpha|\ge 2$. 
In the following, we consider $i\ge 2$. 

We first consider the case $\alpha_j=0$ for any $j\ge 2$.
Then, $\alpha=\alpha_1e_1$ for some $\alpha_1\ge 2$ and 
\eqref{eq-<alpha,rho>=rho_i} reduces to $\alpha_1=\rho_i$.  
In particular, $\rho_i$ is an integer at least 2. 
Since $0<\delta_j<a$ for any $j\ge 2$ 
and $\{\delta_j\}$ is strictly increasing, there are at most finitely many integers in 
$\{\delta_j\}_{j\ge 2}$ and also in $\{\rho_j\}_{j\ge 2}$. 
Whenever $\rho_i\ge 2$ is an integer, we get $\alpha=\rho_ie_1$. 

We next consider the case $\alpha_j\ge 1$ for some $j\ge 2$.
Then, $|\alpha-\alpha_1e_1|=\sum_{j=2}^\infty \alpha_j\ge 1$. 
Substituting $\rho_j=j+\delta_j$ in \eqref{eq-<alpha,rho>=rho_i} yields 
$\langle\alpha,\N\rangle -i+\langle\alpha,\delta\rangle= \delta_i.$
We claim that there are only finitely many such pairs $(\alpha,i)$.

Suppose, to the contrary, that there exists a distinct sequence of pairs $\{(\alpha^{(\ell)}, i^{(\ell)})\}_{\ell\ge1}$, 
with $|\alpha^{(\ell)}|\ge 2$, $|\alpha^{(\ell)}-\alpha^{(\ell)}_1e_1|\ge 1$,
and $1 \le i^{(\ell)} \le \langle\alpha^{(\ell)},\N\rangle$, 
such that 
\begin{align}
\label{eq-<alpha,rho>=rho_i-delta-ell}
\langle\alpha^{(\ell)},\N\rangle -i^{(\ell)}+\langle\alpha^{(\ell)},\delta\rangle= \delta_{i^{(\ell)}}.
\end{align}
With $i^{(\ell)} \le \langle\alpha^{(\ell)},\N\rangle$,  $|\alpha^{(\ell)}-\alpha^{(\ell)}_1e_1|\ge 1$, 
and $\delta_2\le\delta_{j}<a$ for any $j\ge 2$, we get  
\begin{align}
\label{eq-range-alpha-i-ell-difference} 
 0\le \langle\alpha^{(\ell)},\N\rangle -i^{(\ell)}<a,
\end{align}
and $1\le |\alpha^{(\ell)}-\alpha^{(\ell)}_1e_1|<{a}/{\delta_2}.$
We first note that $\{\langle \alpha^{(\ell)}, \N\rangle\}_{\ell \ge 1}$ is unbounded.  
By passing to a subsequence, we assume $\langle \alpha^{(\ell)}, \N\rangle\to\infty$. 
Then, \eqref{eq-range-alpha-i-ell-difference} implies $i^{(\ell)}\to\infty$.
For each $\alpha^{(\ell)}$, take $j_\ell\ge 2$ such that 
$\alpha^{(\ell)}_{j_\ell}\neq 0$ and $\alpha^{(\ell)}_j=0$ for any $j>j_\ell$. 
It is easy to check that $j_\ell\to\infty$. 
With $\alpha^{(\ell)}_{j_\ell}\ge 1$, we 
write \eqref{eq-<alpha,rho>=rho_i-delta-ell} as 
\begin{align}
\label{eq-<alpha,rho>=rho_i-delta-ell-9}
\delta_{i^{(\ell)}}=\big(\langle\alpha^{(\ell)},\N\rangle -i^{(\ell)}\big)
+\langle\alpha^{(\ell)}-\alpha_1^{(\ell)}e_1-e_{j_\ell},\delta\rangle+\delta_{j_\ell}.
\end{align}
Note that the first difference in the right-hand side is an integer satisfying \eqref{eq-range-alpha-i-ell-difference}  
and the second term is a linear combination of $\delta_2, \delta_3, \cdots$ with 
nonnegative integers. By taking $\ell\to\infty$, we conclude 
$\langle\alpha^{(\ell)},\N\rangle -i^{(\ell)}=0$ and 
$\alpha^{(\ell)}=\alpha_1^{(\ell)}e_1+e_{j_\ell}$, for $\ell$ sufficiently large. 
Substituting these in \eqref{eq-<alpha,rho>=rho_i-delta-ell-9} 
yields $\delta_{i^{(\ell)}}=\delta_{j_\ell}$, 
implying $i^{(\ell)}=j_{\ell}$ for large $\ell$. 
Thus, $\alpha^{(\ell)}=\alpha^{(\ell)}_1e_1+e_{i^{(\ell)}}$.
By substituting this in $\langle\alpha^{(\ell)},\N\rangle -i^{(\ell)}=0$,  
we have $\alpha^{(\ell)}_1=0$ and hence $\alpha^{(\ell)}=e_{i^{(\ell)}}$, 
contradicting $|\alpha^{(\ell)}|\ge 2$. This finishes the proof of the claim. 

By combining the conclusions for both $\alpha=\alpha_1e_1$
and $|\alpha-\alpha_1e_1|\ge 1$, we have the desired result. 
\end{proof}

\smallskip 

In the rest of this section, we reparameterize the expansion \eqref{eq-formal-expansion} 
by multi-indices 
of infinite dimension and derive several properties that play important roles in our study.
One property is a simple characterization of the spherical degree of the coefficients. 

\begin{proposition}
\label{prop-formal-expansion-multi-index}
Let 
$\mathcal{X} =\{X_i\}_{i \geq 1}$ be a sequence of spherical harmonics with $X_i \in \mathscr{E}_i$, 
and $\{P_\ell\}_{\ell\geq 1}$ be the sequence of generalized polynomials determined 
by $\mathcal{X}$ as in Definition \ref{def-sequence-P-determined-by-X}. 
Define, for $i\ge 1$, 
$$P_{e_i}(t,\theta) = X_i(\theta)p_i^+(t),$$ 
and $P_\alpha$ inductively 
for $\alpha  \in \ind$ with $|\alpha| \geq 2$ by
\begin{equation}\label{eq-recursive-equation-multi-index}
\begin{aligned}
P_\alpha  = \mathcal L_{\langle\alpha,\rho\rangle}^{-1} \Big[\sum_{j = 2}^{|\alpha|} 
\sum_{\beta_1 + \cdots + \beta_j = \alpha} a_j P_{\beta_1} \cdots P_{\beta_j} \Big].
\end{aligned}
\end{equation}
Then,  for any $\ell\ge 1$,
\begin{equation}\label{eq-P-ell-expression} 
P_\ell = \sum_{\alpha, \langle \alpha, \rho \rangle = \mu_\ell } P_\alpha.\end{equation}
Moreover, for any $\alpha\in \ind$, 
$P_\alpha$ has a polynomial degree at most $|\alpha| - 1$ and a spherical degree at most $\langle \alpha,\N\rangle$.
\end{proposition}

We can say that $\{P_\alpha = P_\alpha(\mathcal{X})\}_{\alpha \in \ind}$ 
is the sequence of generalized polynomials determined by $\mathcal{X} =\{X_i\}_{i \geq 1}$.

\begin{proof}
For any $\ell\ge 1$, set 
$$\widetilde P_\ell = \sum_{\alpha, \langle \alpha, \rho \rangle = \mu_\ell } P_\alpha.$$
We now prove by induction that, for any $\ell \ge 1$,
\begin{align}
\label{eq-tilde P = P}
\widetilde P_\ell = P_\ell.
\end{align}
For $\ell = 1$, we have $\mu_1 = \rho_1$ and then
$\alpha = e_1$ is the only $\alpha \in \ind$ satisfying $\langle\alpha,\rho \rangle = \mu_1$.
Hence, $\widetilde P_1 = P_{e_1} = X_1(\theta) p_1^+(t)$. 
This proves \eqref{eq-tilde P = P} for $\ell = 1$ by \eqref{eq-P-1}. 

Take $\ell \ge 2$ and 
assume \eqref{eq-tilde P = P} holds for $1,\cdots, \ell - 1.$
If $\mu_\ell = \rho_i$ for some $i\ge 1$ and $\rho_i\notin \Z_+\text{-span}\{\rho_1,\cdots,\rho_{i-1}\}$,
then $\alpha = e_i$ is the only $\alpha \in \ind$ satisfying $\langle \alpha,\rho\rangle = \mu_\ell$.
We thus have \eqref{eq-tilde P = P} by Definition \ref{def-sequence-P-determined-by-X}(ii). 
For other cases, set
$$Q_\ell = \sum_{\substack{|\alpha|\ge 2\\ \langle\alpha,\rho\rangle = \mu_\ell}} P_\alpha.$$
Then, by \eqref{eq-recursive-equation-multi-index}, we have 
\begin{align*}
Q_\ell  =\mathcal L_{\mu_\ell}^{-1}\Big[  \sum_{\alpha, \langle \alpha, \rho \rangle = \mu_\ell } \sum_{j = 2}^{|\alpha|} 
\sum_{\beta_1 + \cdots + \beta_j = \alpha}   a_j P_{\beta_1}\cdots P_{\beta_j}\Big].
\end{align*}
Here, we sum over all $\beta_1,\cdots,\beta_j \in \ind$ 
such that $\beta_1 + \cdots + \beta_j = \alpha$ for some $\alpha \in \ind$ satisfying 
$\langle \alpha, \rho \rangle = \mu_\ell$. 
This is the same as the sum over all $\beta_1,\cdots,\beta_j \in \ind$ such that $\beta_1 + \cdots + \beta_j$ 
lies in the hyperplane $\{x \in \R^{\infty};\langle x,\rho \rangle = \mu_\ell\}.$ 
Equivalently, we sum over all $\beta_1,\cdots,\beta_j \in \ind$ such that 
$\langle \beta_1,\rho \rangle + \cdots + \langle \beta_j,\rho \rangle 
= \mu_\ell$. Note that each $\langle\beta_m,\rho\rangle$ is some $\mu_{k_m}$. 
Since $\langle \beta_1,\rho \rangle + \cdots + \langle \beta_j,\rho \rangle \geq j \rho_1=j$, 
we have $j \leq \ell$. 
Hence,
\begin{align*}
{Q}_\ell &=\mathcal L_{\mu_\ell}^{-1}\Big[\sum_{j=2}^\ell 
\sum_{\substack{\beta_1,\cdots,\beta_j\\ \langle \beta_1,\rho \rangle + \cdots + \langle \beta_j,\rho \rangle = \mu_\ell}} 
a_j P_{\beta_1}\cdots P_{\beta_j} \Big]\\
& = \mathcal L_{\mu_\ell}^{-1}\Big[\sum_{j=2}^\ell \sum_{\mu_{k_1} + \cdots + \mu_{k_j} = \mu_\ell} 
a_j\Big(\sum_{\beta_1, \langle \beta_1,\rho \rangle = \mu_{k_1} } 
P_{\beta_1} \Big)\cdots \Big( \sum_{\beta_j, \langle \beta_j,\rho \rangle = \mu_{k_j} } P_{\beta_j}\Big)\Big]\\
& = \mathcal L_{\mu_\ell}^{-1}\Big[\sum_{j=2}^\ell \sum_{\mu_{k_1} + \cdots + \mu_{k_j} = \mu_\ell} 
a_j {\widetilde P}_{k_1} \cdots {\widetilde P}_{k_j}\Big].
\end{align*}
The induction hypothesis yields $\widetilde P_{k_s}=P_{k_s}$ for $1\le s\le j$, as $k_s \le \ell-1$, and hence 
$$Q_\ell =  \mathcal  L_{\mu_\ell}^{-1} \Big[\sum_{j=2}^\ell \sum_{\mu_{k_1} + \cdots + \mu_{k_j} = \mu_\ell} 
a_j  P_{k_1} \cdots P_{k_j}\Big].$$
If $\mu_\ell \ne \rho_i$ for any $i$, then $\widetilde P_\ell = Q_\ell$.
Hence, by Definition \ref{def-sequence-P-determined-by-X}(i), 
\eqref{eq-tilde P = P} holds.
If  $\mu_\ell = \rho_i$ for some $i \geq 2$ and $\rho_i \in \Z_+$-span$\{\rho_1,\cdots,\rho_{i-1}\}$, 
then $e_i$ is the only $\alpha \in \ind$ with $|\alpha|=1$ satisfying $\langle\alpha,\rho\rangle = \mu_\ell$. 
Hence, 
$${\widetilde P}_\ell = P_{e_i} +  Q_\ell  = X_i(\theta)p_i^+(t)  +  Q_\ell.$$
We thus obtain \eqref{eq-tilde P = P} by Definition \ref{def-sequence-P-determined-by-X}(iii). 
This finishes the proof of \eqref{eq-tilde P = P} by induction.

We now prove the assertion on the polynomial degree and spherical degree by induction on $\abs{\alpha}$.
By construction, $P_{e_i}(t,\theta) = X_i(\theta)p_i^+(t)$ for $i \geq 1$. Note that $i = \langle e_i,\N\rangle$. 
Hence, the conclusion holds for $\abs{\alpha} = 1$. 
Assume $P_\alpha$ has the stated polynomial degree and spherical degree, 
for $|\alpha| = 1,\cdots,N$, 
and fix $\alpha\in \ind$ 
with $|\alpha| = N+1$.  
By the induction hypothesis and Lemma \ref{lemma-decomposition-of-spherical-harmonics}, the expression
$$\sum_{j=2}^{|\alpha|} \sum_{\beta_1 +\cdots + \beta_j=\alpha} a_j P_{\beta_1}\cdots P_{\beta_j}$$
is a generalized polynomial with $C^\infty(\mathbb{S}^{n-1})$-coefficients of a polynomial degree at most $|\alpha| - 2$ and 
a spherical degree at most $\langle\alpha,\N\rangle$. 
We point out that if $\beta_1 + \cdots + \beta_j = \alpha$ for some $j \geq 2$, 
then $\langle\beta_1,\N\rangle + \cdots + \langle\beta_j,\N\rangle = \langle\alpha,\N\rangle$ 
and $(|\beta_1| - 1) + \cdots + (|\beta_j| - 1) = |\alpha| - j \leq |\alpha| - 2$. 
By Lemma \ref{lemma-polynomial_solution} and \eqref{eq-recursive-equation-multi-index}, 
we obtain the desired conclusion.
\end{proof}

We now make a remark. Let $\mathcal{X} = \{X_i\}_{i\geq 1}$ be a sequence of spherical harmonics
with $X_i \in \mathscr{E}_i$. In the parameterization $\{P_\ell(\mathcal{X})\}_{\ell \geq 1}$ given by 
$\ell \geq 1$, $P_\ell(\mathcal{X})$ satisfies \eqref{eq-recursive-equation-Asymp-expans}. 
If $\mu_\ell \neq \rho_i$ for any $i \geq 1$, then \eqref{eq-recursive-equation-Asymp-expans} 
admits a unique generalized polynomial solution, and if $\mu_\ell = \rho_i$ for some $i \geq 1$, 
we can write $P_\ell = X_i p_i^+ +  Q_\ell $, for a unique generalized polynomial $Q_\ell$. 
In other words, the constraint for the parameterization $\{P_\ell(\mathcal{X})\}_{\ell \ge 1}$ is that 
if $\mu_\ell = \rho_i$ for some $i \geq 1$, then the free part of $P_\ell$ is given by $X_i(\theta) p_i^+(t)$, 
where $X_i$ is the prescribed spherical harmonic with $X_i \in \mathscr{E}_i$ and 
$\psi_i^+ = p_i^+ e^{-\rho_i t}$ is the fixed exponential decay basis element of $\mathrm{Ker} (L_i)$. 
On the other hand, for the parameterization $\{P_\alpha(\mathcal{X})\}_{\alpha \in \ind}$ 
given by $\alpha \in \ind$, we require $P_{e_i}(t,\theta) = X_i(\theta)p_i^+(t)$ for any $i \geq 1$, 
and $P_\alpha$ given by \eqref{eq-recursive-equation-multi-index} for $|\alpha| \geq 2$. 
Hence, all the free parts are given by $P_\alpha$ with $|\alpha|=1$ 
and the rest of $P_\alpha$ is determined by $P_\beta$ for $\beta < \alpha$,
or equivalently, $\beta_i\le \alpha_i$ for any $i\ge 1$ and $\beta_i< \alpha_i$ for some $i\ge 1$.

The simple characterization of the spherical degree of $P_\alpha$ has 
significant consequences.
The proof of Proposition \ref{prop-formal-expansion-multi-index} 
shows that $P_\alpha$ in \eqref{eq-recursive-equation-multi-index}
is the operator $\mathcal L_{\langle\alpha,\rho\rangle}^{-1}$ acting on 
a generalized polynomial of a polynomial degree at most $|\alpha| - 2$ and a spherical degree at most $\langle \alpha,\N\rangle$. 
In order to estimate $P_\alpha$, 
by \eqref{eq-def-L_0^-1}, \eqref{eq-def-L_i^-1}, and \eqref{eq-L_mu^-1}, 
we need to study whether the set 
\begin{align}\label{eq-difference-small-divisors}
\mathcal S=\{|\langle \alpha,\rho \rangle - \rho_i|;\, 
\alpha\in \ind, 1\le i\le \langle\alpha,\N\rangle,\langle \alpha,\rho \rangle \neq \rho_i\}\end{align} 
has a positive infimum. 
If this set has the zero infimum, 
we encounter the so-called {\it small divisors} when we attempt to estimate $P_\alpha$. 
Then, deriving effective estimates of $P_\alpha$ becomes extremely difficult, if not impossible.

By Lemma \ref{lemma-comparison-of-mu-and-rho}, or \eqref{eq-estimate-mu-rho-i} in particular, 
we conclude that 
the set $\mathcal S$ indeed has a positive infimum if we require, in addition, $\alpha_1=0$. 
On the other hand, Example \ref{exa-small-divisors} demonstrates that nonzero $\alpha_1$ may
lead to the zero infimum of $\mathcal S$. 
It becomes important to eliminate nonzero $\alpha_1$. 
In the next result, we present a simple criterion on the sequence $\mathcal X = \{X_i\}_{i\ge 1}$
to achieve this 
and also improve Proposition \ref{prop-formal-expansion-multi-index} 
with the help of Lemma \ref{lemma-comparison-of-mu-and-rho}.

\begin{cor}
\label{cor-expansion-from-optimal-term-no-poly}
Let $\mathcal{X} =\{X_i\}_{i\geq 1} $ be a sequence of spherical harmonics
with $X_i \in \mathscr{E}_i$ 
and $\{P_\alpha\}_{\alpha \in \ind}$ be the sequence of generalized polynomials 
determined by $\mathcal{X}$ as in Proposition \ref{prop-formal-expansion-multi-index}. 
If $X_1  = 0$, then $P_\alpha = 0$ for any $\alpha \in \ind$ with $\alpha_1 \neq 0$ 
and $P_\alpha$ has polynomial degree $0$ for any $\alpha \in \ind$ with $\alpha_1 =0$. 
\end{cor}

\begin{proof}
We first prove by induction on $|\alpha|$ that $P_\alpha = 0$ 
for any $\alpha \in \ind$ with $\alpha_1 \neq 0$. 
For $|\alpha| = 1$, if $\alpha_1 \neq 0$, 
then $\alpha = e_1$. 
By assumption, we have $P_\alpha = P_{e_1} = X_1p_1^+ = 0$. 
Now, assume that it holds for $|\alpha| = 1,\cdots,N$ for some $N \geq 1$. 
Fix $\alpha \in \ind$ with 
$|\alpha| = N+1$ and $\alpha_1 \neq 0$.
Assume that $\alpha = \beta_1 + \cdots + \beta_j$ 
for some $j\geq 2$, $\beta_1,\cdots,\beta_j \in \ind$. 
Since $\alpha_1 \neq 0$, 
there exists some $s$ such that $\beta_{s,1}\neq 0$. 
The induction hypothesis implies $P_{\beta_s} = 0$. 
Therefore, by \eqref{eq-recursive-equation-multi-index}, $P_\alpha = 0$. 
This completes the proof by induction.

Next, we prove by induction on $|\alpha|$ that $P_\alpha$ has polynomial degree $0$ 
for any $\alpha \in\ind$ with $\alpha_1 = 0$. 
By definition, $P_{e_i} = X_i p_i^+$ for any $i \geq 1$. 
Hence, $P_\alpha $ has polynomial degree $0$ for any $\alpha \in \ind$ with $\alpha_1 = 0$ and $|\alpha| = 1$. 
Now, assume $P_\alpha$ has polynomial degree $0$ for any $\alpha \in \ind$ with $\alpha_1 = 0$ and $|\alpha| \leq N$
for some $N\ge 1$. 
Fix some $\alpha \in \ind$ with $\alpha_1 = 0$ and $|\alpha| = N+1$. 
Then, the expression inside $\mathcal L_{\langle\alpha,\rho\rangle}^{-1}$ in
the right-hand side of \eqref{eq-recursive-equation-multi-index} is a  
generalized polynomial with $C^\infty(\mathbb{S}^{n-1})$-coefficients 
of polynomial degree $0$
and a spherical degree at most $\langle\alpha,\N\rangle$,
by Lemma \ref{lemma-decomposition-of-spherical-harmonics} and 
Proposition \ref{prop-formal-expansion-multi-index}. 
Lemma \ref{lemma-comparison-of-mu-and-rho} asserts $\langle\alpha,\rho\rangle > \rho_{\langle\alpha,\N\rangle}$ 
since $\alpha_1 = 0$ and $|\alpha|\ge2$. Hence, by Lemma \ref{lemma-polynomial_solution}(iii), 
$P_\alpha$ has polynomial degree $0$. 
This completes the proof by induction.
\end{proof}

In the proof of the second assertion in Corollary \ref{cor-expansion-from-optimal-term-no-poly},
we only used the positivity of the difference in \eqref{eq-estimate-mu-rho}.
The positive lower bound will be essentially used later on.
The first assertion does not require such positivity. 

\section{Convergence of the Expansions}
\label{sec-Convergence-of-the-Expansions}


In this section, we construct exponential decay solutions to \eqref{eq-linearized-Yamabe-at-xi}. 
To this end, we fix a sequence of admissible bases $\{\psi_i^+, \psi_i^-\}_{i\geq 0}$ 
for $\{\mathrm{Ker}(L_i)\}_{i\geq 0}$ as in Lemma \ref{lemma-admissible-sequence-of-bases}.  

Let $\mathcal{X} =\{X_i\}_{i \geq 1} $ be an arbitrary sequence of spherical harmonics with $X_i \in \mathscr{E}_i$ 
and $\{P_\alpha(\mathcal{X})\}_{\alpha \in \ind}$ be the sequence of 
generalized polynomials determined by $\mathcal{X}$ 
as in Proposition \ref{prop-formal-expansion-multi-index}. 
Consider the corresponding formal series 
\begin{align}
\label{eq-formal-expansion-multi-index}
\sum_{\alpha \in \ind} P_\alpha(\mathcal{X})(t,\theta) e^{- \langle\alpha,\rho\rangle t}.
\end{align}
We will prove it converges to 
a solution $\varphi$ 
of \eqref{eq-linearized-Yamabe-at-xi} 
under the condition $X_1 = 0$.

\subsection{Preparation}\label{subsec-Preparation}


The proof of the convergence of the series \eqref{eq-formal-expansion-multi-index} 
is based on the method of majorant. The next lemma provides an upper bound.

\begin{lemma}
\label{lemma-majorant-series}
Let $A$, $B$, and $M$ be given positive constants.
Define $A_{e_i}  = A^i$ for $i \geq 1$ 
and $A_\alpha = A_\alpha(A,B,M)$ inductively for $\alpha \in \ind$ with $|\alpha|\geq 2$ by
\begin{align}
\label{eq-A_alpha}
A_\alpha = B \sum_{j=2}^{|\alpha|} M^j \sum_{\beta_1 + \cdots + \beta_j = \alpha} A_{\beta_1}\cdots A_{\beta_j}.
\end{align}
Then, for any $\ell \ge 1$,
\begin{align}
\label{eq-majorant-series}
\sum_{\langle\alpha,\N\rangle = \ell } A_\alpha  \le D^\ell,
\end{align}
for some positive constant $D$ depending only on $A$, $B$, and $M$.
\end{lemma}

\begin{proof}
For each $\ell \ge 1$, set 
$$A_\ell  = \sum_{\langle \alpha,\N\rangle = \ell} A_\alpha.$$
Then, $A_1 = A_{e_1} = A$
and, for $\ell \ge 2$, we write
$$A_\ell = A_{e_\ell} + \sum_{\substack{|\alpha|\ge 2\\ \langle \alpha,\N\rangle = \ell}} A_\alpha  
= A^\ell + \sum_{\substack{|\alpha|\ge 2\\ \langle \alpha,\N\rangle = \ell}} A_\alpha.$$
By \eqref{eq-A_alpha}, we have
$$\sum_{\substack{|\alpha|\ge 2\\ \langle \alpha,\N\rangle = \ell}} A_\alpha 
= B\sum_{\substack{|\alpha|\ge 2\\\langle \alpha,\N\rangle = \ell}} \sum_{j=2}^{|\alpha|} M^j 
\sum_{\beta_1 + \cdots + \beta_j = \alpha} A_{\beta_1}\cdots A_{\beta_j}.$$
As in the proof of Proposition \ref{prop-formal-expansion-multi-index}, 
we can switch the order of summation and write
$$\sum_{\substack{|\alpha|\ge 2\\\langle \alpha,\N\rangle = \ell}} \sum_{j=2}^{|\alpha|}  
\sum_{\beta_1 + \cdots + \beta_j = \alpha} = \sum_{j=2}^\ell 
\sum_{\substack{\beta_1,\cdots,\beta_j\\
\langle\beta_1,\N\rangle + \cdots + \langle\beta_j,\N\rangle = \ell}}.$$
Then, a simple rearrangement yields
\begin{align*}
\sum_{\substack{|\alpha|\ge 2\\ \langle \alpha,\N\rangle = \ell}} A_\alpha 
&=B\sum_{j=2}^\ell M^j \sum_{k_1 +\cdots + k_j = \ell} \Big( \sum_{\beta_1,\langle\beta_1,\N\rangle = k_1} 
A_{\beta_1} \Big) \cdots \Big( \sum_{\beta_j,\langle\beta_j,\N\rangle = k_j} A_{\beta_j} \Big)\\
& = B \sum_{j = 2}^\ell M^j \sum_{k_1 +\cdots + k_j = \ell} A_{k_1}\cdots A_{k_j},
\end{align*}
and hence 
\begin{equation}\label{eq-recurrence-formula-scalar-majorant}
\begin{aligned}
A_\ell = A^\ell+B \sum_{j = 2}^\ell M^j \sum_{k_1 +\cdots + k_j = \ell} A_{k_1}\cdots A_{k_j}.
\end{aligned}
\end{equation}
Now, consider the formal series 
$$\psi(s) = \sum_{i=1}^\infty A_i s^i.$$
Note that the summation in the right-hand side of \eqref{eq-recurrence-formula-scalar-majorant} is the $\ell$-th coefficient of 
$$B \sum_{j=2}^\infty M^j \psi^j = B\frac{M^2 \psi^2}{1 - M\psi}. $$
Then, $\psi$ formally satisfies the algebraic equation
$$\psi(s) = \sum_{\ell=1}^\infty A^\ell s^\ell + \frac{B M^2\psi^2}{1 - M\psi}.$$
Since $\psi(0) = 0$, the algebraic equation has a unique solution $\psi$ that is an analytic function near $0$. 
In particular, there exists $D  = D(A,B,M)> 0$ such that
\eqref{eq-majorant-series} holds.
\end{proof}

Recall that we formally write $F(s)$ in \eqref{eq-RHS-of-linearized-Yamabe} as \eqref{eq-formal-expansion-F}, 
for some sequence $\{a_j\}_{j\geq 2}$ of smooth periodic functions on $\R$ with the same period as $\xi$.
In the next result, we prove an upper bound of $a_j$ for $j \geq 2$.

\begin{lemma}
\label{lemma-upper-bound-for-a_j}
Let $\{a_j\}_{j \geq 2}$ be the sequence of periodic functions as in \eqref{eq-formal-expansion-F}. 
Then, there exists a constant $M\ge 1$, 
depending only on $n$ and $\xi$, such that, 
for any $j \ge 2$ and $k \ge 0$, 
\begin{align}
\label{eq-upper-bound-for-a_j}
|a_j^{(k)}|_{L^\infty (\R)} \le M^{j+k} k!.
\end{align}
\end{lemma}

\begin{proof}
Let $G$ be the analytic function given by \eqref{eq-def-G-RHS-of-Yamabe}.
By a simple Taylor expansion of $G$ in $s$, we obtain
\[
a_j(t)
=
-\frac{1}{4}n(n-2)\binom{p}{j}\xi(t)^{p-j},
\]
where $p = \frac{n+2}{n-2}$.
Write $\xi^{p - j} = \xi^p \cdot (\xi^{-1})^j$.
Note that $\min_\R \xi > 0$ and $\xi^p$, $\xi^{-1}$ are periodic 
analytic functions on $\R$.
Fix $t_0 \in \R$ and write
$$\xi(t_0+t)^p = \sum_{k=0}^\infty b_k t^k\quad\text{and}\quad \xi^{-1}(t_0+t) = \sum_{k=0}^\infty c_k t^k,$$
for some sequences $\{b_k\}_{k \ge 0}$ and $\{c_k\}_{k \ge 0}$.
Then, 
for any $k \ge 0$,
we have
\begin{align}
\label{eq-estimate-b_k,c_k}
|b_k|,\, |c_k| \le C A^k,
\end{align}
for some positive constants $A$ and $C$
depending only on $n$ and $\xi$.
A simple multiplication of the power series yields
$$\xi(t_0+t)^{p-j} = \sum_{k=0}^\infty \Big( \sum_{k_0 + k_1 +\cdots + k_j = k} b_{k_0} c_{k_1}\cdots c_{k_j} \Big) t^k.$$
We point out that the expression in the parenthesis is 
$(\xi^{p-j})^{(k)}(t_0)/k!$.
With $t = 1/(2A)$ 
and \eqref{eq-estimate-b_k,c_k}, we have
\begin{align*}
\frac{1}{(2A)^k}\sum_{k_0 + k_1 +\cdots + k_j = k} |b_{k_0} c_{k_1}\cdots c_{k_j}| \le \Big(\sum_{k=0}^\infty \frac{|b_k|}{(2A)^k} \Big) 
\Big( \sum_{k=0}^\infty \frac{|c_k|}{(2A)^k} \Big)^j \le (2C)^{j+1}.
\end{align*}
Hence,
$$|(\xi^{p-j})^{(k)}|_{L^\infty(\R)} \le (2A)^k (2C)^{j+1}k!.$$
By the Stirling formula, for any $j \ge 2$, we have
$$\Big|\binom{p}{j} \Big| \le \widetilde A^j,$$
for some $\widetilde A >0$.
This finishes the proof of \eqref{eq-upper-bound-for-a_j}.
\end{proof}

Recall the Sobolev inner-product $\langle\cdot,\cdot\rangle_m$ and Sobolev norm $\|\cdot\|_m$ 
given by \eqref{eq-H^*1-inner-product} and \eqref{eq-H^m-inner-product-with-constant}. 
For a generalized polynomial $P$ with $C^\infty(\mathbb{S}^{n-1})$-coefficients of 
polynomial degree $0$, we define
\begin{align}
\label{eq-norm-degree-0-poly}
\|P\|_{m,\infty} = \sup_{t \in \R} \|P(t,\cdot)\|_m.
\end{align}
If $P$ is independent of $t$, then $\|P\|_{m,\infty} = \|P\|_m$.
It is straightforward to check that $\|\cdot\|_{m,\infty}$ is a norm on the space of all 
generalized polynomials with $C^\infty(\mathbb{S}^{n-1})$-coefficients of 
polynomial degree $0$ 
and, for any generalized polynomials $P_1,P_2$ with $C^\infty(\mathbb{S}^{n-1})$-coefficients 
of 
polynomial degree $0$,
\begin{align}
\label{eq-Banach-algebra-norm-for-degree-0-poly}
\|P_1 P_2\|_{m,\infty} \leq \|P_1\|_{m,\infty} \|P_2\|_{m,\infty}.
\end{align}
We fix an integer $m_*$ as in \eqref{eq-choice-of-m_*}.

For $i \ge 0$, 
we write $d_i = \dim \mathscr{E}_i$ 
and denote by $Y_{i1},\cdots,Y_{id_i}$ an orthonormal basis of $\mathscr{E}_i$ 
with respect to $\langle\cdot,\cdot\rangle_{m_*}$. Note that $\{Y_{ij};\,i\ge 0, 1\le j\le d_i\}$ 
forms an orthonormal basis of $(H^{m_*}(\mathbb{S}^{n-1}),\langle\cdot,\cdot\rangle_{m_*})$.
In particular, $\{Y_{ij}\}$ forms an orthogonal basis of $L^2(\mathbb{S}^{n-1})$.
In addition, it is straightforward to check that, 
for any $X \in \mathscr E_i$,
\begin{align}
\label{eq-comparison-norm_*-and-norm_*1}
c_0(1 + \lambda_i)^{1/2}\|X\|_{m_*} \le\|X\|_{{m_*}+1} \le C_0(1 + \lambda_i)^{1/2} \|X\|_{m_*},
\end{align}
for some positive constants $c_0$ and $C_0$ depending only on $n$.

We now prove a fundamental estimate. 
Lemma \ref{lemma-comparison-of-mu-and-rho} plays an essential role in its proof. 

\begin{lemma}
\label{lemma-estimate-L_mu_alpha^-1}
Assume $\alpha \in \ind$, with $\alpha_1=0$ and $|\alpha| \geq 2$,
and $R_\alpha$ is a generalized polynomial  with $C^\infty(\mathbb{S}^{n-1})$-coefficients of polynomial degree $0$ and  
a spherical degree at most $\langle\alpha,\N\rangle$. 
Then, 
\begin{align}
\label{eq-estimate-L_mu_alpha^-1-poly}
\| \mathcal L_{\langle\alpha,\rho\rangle}^{-1}[R_\alpha]\|_{m_*+1,\infty} \leq B \|R_\alpha\|_{m_*,\infty},
\end{align}
where $B$ is a positive constant depending only on $n$ and $\xi$.
\end{lemma}

\begin{proof} 
By the assumption, we write
\begin{align*}
    R_\alpha(t,\theta)  
    =\sum_{i=0}^{\langle\alpha,\N\rangle} \sum_{j=1}^{d_i} r_{ij}(t)  Y_{ij}(\theta) ,
\end{align*}
where $r_{ij}$ are smooth periodic functions on $\R$ with the same period as $\xi$,
for $i = 0,\cdots, \langle\alpha,\N\rangle$ and $j = 1,\cdots, d_i$.
By \eqref{eq-norm-degree-0-poly}, we have
$$\|R_\alpha\|_{m_*,\infty} = \sup_{t \in \R} \Big( \sum_{i=0}^{\langle\alpha,\N\rangle} \sum_{j=1}^{d_i} r_{ij}(t)^2 \Big)^{\frac{1}{2}}.$$

Now, by \eqref{eq-L_mu^-1}, we get
\begin{align*}
\mathcal L_{\langle\alpha,\rho\rangle}^{-1}[R_\alpha](t,\theta) 
= e^{ \langle\alpha,\rho\rangle t} \sum_{i=0}^{\langle\alpha,\N\rangle} \sum_{j=1}^{d_i} 
 {L}_i^{-1}[r_{ij}  e^{-\langle\alpha,\rho\rangle s}](t) Y_{ij}(\theta).
\end{align*}
We write
$$\mathcal L_{\langle\alpha,\rho\rangle}^{-1}[R_\alpha] = J_1 + J_2,$$
where
\begin{align*}
J_1 = e^{\langle\alpha,\rho\rangle t}L_0^{-1} [r_{01}  e^{-\langle\alpha,\rho\rangle s}](t) Y_{01}(\theta),
\end{align*}
and
\begin{align*}
J_2 = \sum_{i=1}^{\langle\alpha,\N\rangle} \sum_{j=1}^{d_i} 
e^{\langle\alpha,\rho\rangle t}L_i^{-1} [r_{ij}e^{-\langle\alpha,\rho\rangle s}](t) Y_{ij}(\theta).
\end{align*}
Note that 
$$\|\mathcal L_{\langle\alpha,\rho\rangle}^{-1}[R_\alpha]\|_{m_*+1,\infty} \leq \|J_1\|_{m_*+1,\infty}+ \|J_2\|_{m_*+1,\infty}.$$
We will estimate $J_1$ and $J_2$ separately.
By Lemma \ref{lemma-admissible-sequence-of-bases}, $\{p_i^+\}_{i\geq 0}$ and $\{p_i^-\}_{i\geq 0}$ 
are uniformly bounded, and $\{|W_i|\}_{i \geq 0}$ is uniformly bounded below. 
For simplicity, we write 
$$a \lesssim b\text{ if } a \leq C b\text{ for some }C>0,\text{ depending only on } n\text{ and }\{\psi_i^+,\psi_i^-\}_{i\geq 0}.$$

We now estimate $J_1$. By \eqref{eq-def-L_0^-1}, we have
\begin{align*}
L_0^{-1} [r_{01} e^{-\langle\alpha,\rho\rangle s}](t)  
&= \frac{p_0^+(t)}{W_0}\int_t^\infty p_0^-(s) r_{01}(s)  e^{-\langle\alpha,\rho\rangle s} ds\\
&\qquad- \frac{p_0^-(t)}{W_0}\int_t^\infty p_0^+(s) r_{01}(s)  e^{-\langle\alpha,\rho\rangle s} ds\\
&\qquad+\frac{c p_0^+(t)}{W_0} \int_{t}^\infty \int_s^\infty p_0^+(\tau) r_{01}(\tau) e^{-\langle\alpha,\rho\rangle \tau}  d\tau ds,     
\end{align*}
and then, by a straightforward estimate, 
\begin{align}
\label{eq-estimate-for-(I)}
\|J_1\|_{m_*+1,\infty} \lesssim  \frac{1}{\langle\alpha,\rho\rangle}|r_{01}|_{L^\infty(\R)} 
\leq  \frac{1}{\langle\alpha,\rho\rangle}\|R_\alpha\|_{m_*,\infty}.
\end{align}

Next, we estimate $J_2$. 
By Lemma \ref{lemma-comparison-of-mu-and-rho}
and the monotonicity of the sequence $\{\rho_i\}_{i\ge 1}$, 
we get $\langle\alpha,\rho\rangle > \rho_i$ for $i = 1,\cdots, \langle\alpha,\N\rangle$.
By \eqref{eq-def-L_i^-1}, we have, for $i =1,\cdots,\langle\alpha,\N\rangle$ and $j = 1,\cdots, d_i$,
\begin{equation}
\label{eq-expression-Y_ij-part} 
e^{\langle\alpha,\rho\rangle t}L_i^{-1} [r_{ij}  e^{-\langle\alpha,\rho\rangle s}](t) = K_{ij}+L_{ij},\end{equation}
where 
\begin{align*}
K_{ij}&= \frac{p_i^+(t)}{W_i} e^{ (\langle\alpha,\rho\rangle -\rho_i) t}
\int_{t}^\infty p_i^-(s) r_{ij}(s) e^{-(\langle\alpha,\rho\rangle - \rho_i) s} ds,\\
L_{ij}&=-  \frac{p_i^-(t)}{W_i} e^{ ( \langle\alpha,\rho\rangle + \rho_i) t}
\int_{t}^\infty p_i^+(s) r_{ij}(s)  e^{-(\langle\alpha,\rho\rangle + \rho_i) s} ds.
\end{align*}
In the following, for simplicity, we write
$$\sum_{i,j} = \sum_{i=1}^{\langle\alpha,\N\rangle}\sum_{j=1}^{d_i}.$$
Fix $t \in \R$.
For the first term in \eqref{eq-expression-Y_ij-part}, we have
\begin{align*}
\Big\| \sum_{i,j}K_{ij}(t)Y_{ij}(\cdot)\Big\|_{m_*+1}
&\lesssim \Big( \sum_{i,j} \Big(\int_{t}^\infty p_i^-(s) r_{ij}(s) e^{(\langle\alpha,\rho\rangle - \rho_i)(t-s)} ds\Big)^2 \Big)^\frac{1}{2}\\
&\lesssim
\int_t^\infty  \Big(\sum_{i,j} r_{ij}^2 (s) \Big)^{\frac{1}{2}}  e^{(\langle\alpha,\rho\rangle - \rho_{\langle\alpha,\N\rangle})(t-s)} ds,
\end{align*}
where we used \eqref{eq-comparison-norm_*-and-norm_*1} and 
the lower bound of $W_i/\rho_i$ in Lemma \ref{lemma-admissible-sequence-of-bases} in the first estimate,
and used the Minkowski integral inequality for the last estimate.
We point out that $t-s \leq 0$ and $\langle\alpha,\rho\rangle - \rho_{\langle\alpha,\N\rangle}$ 
is the smallest among 
$\big\{\langle\alpha,\rho\rangle - \rho_i;\,1 \leq i \leq \langle\alpha,\N\rangle \big\}.$ 
A straightforward computation yields
\begin{align*}
\Big\|\sum_{i,j} K_{ij}Y_{ij} 
\Big\|_{m_*+1,\infty}
\lesssim \frac{\|R_\alpha\|_{m_*,\infty}}{\langle\alpha,\rho\rangle - \rho_{\langle\alpha,\N\rangle}}.
\end{align*}
Similarly, for the second term in \eqref{eq-expression-Y_ij-part}, we have
\begin{align*}
\Big\| \sum_{i,j} L_{ij}Y_{ij}
\Big\|_{m_*+1,\infty}
\lesssim \frac{\|R_\alpha\|_{m_*,\infty}}{\langle\alpha,\rho\rangle + \rho_1}.
\end{align*}
Combining these two estimates with \eqref{eq-estimate-mu-rho}, we have 
\begin{align}
\label{eq-estimate-for-(II)}
\|J_2\|_{m_*+1,\infty}\lesssim \frac{1}{\langle\alpha,\rho\rangle - \rho_{\langle\alpha,\N\rangle}} \|R_\alpha\|_{m_*,\infty}
\le \frac{1}{\epsilon_*|\alpha|} \|R_\alpha\|_{m_*,\infty},
\end{align}
where $\epsilon_*$ is as in Lemma \ref{lemma-comparison-of-mu-and-rho}.
By \eqref{eq-estimate-for-(I)} and \eqref{eq-estimate-for-(II)},
we obtain
\eqref{eq-estimate-L_mu_alpha^-1-poly}.
\end{proof}

Lemma \ref{lemma-comparison-of-mu-and-rho} plays an essential role in the proof above. 
It asserts that the set $\mathcal S$ in \eqref{eq-difference-small-divisors} in our setting has a positive infimum. 
If $\mathcal S$ has a zero infimum, we encounter the so-called small divisors.
Lemma \ref{lemma-estimate-L_mu_alpha^-1} will be used in the proofs of Theorems
\ref{thrm-convergence-of-the-expansion-from-optimal-term} and \ref{thm:actual-solution-series-improved}.

\subsection{Construction of Solutions}\label{subsec-Constructions} 


Now, we are ready to prove the convergence of the formal series 
\eqref{eq-formal-expansion-multi-index} 
determined by a sequence $\mathcal{X} =\{X_i\}_{i \geq 1}$ of spherical harmonics
with $X_i \in \mathscr{E}_i$ such that $X_1  =  0$.


\begin{theorem}
\label{thrm-convergence-of-the-expansion-from-optimal-term}
Let $\mathcal{X} =\{X_i\}_{i\geq 1} $ be a sequence of spherical harmonics
with $X_i \in \mathscr{E}_i$ and $X_1 =0$, 
and $\{P_\alpha(\mathcal{X})\}_{\alpha \in \ind}$ be the sequence of 
generalized polynomials determined by $\mathcal{X}$ 
as in Proposition \ref{prop-formal-expansion-multi-index}.
Assume, for any $i\ge 2$, 
\begin{align}
\label{eq-exponential-growth-condition-expansion-from-optimal-term}
    \|X_i\|_{L^2(\mathbb{S}^{n-1})} \leq \widetilde A^i,
\end{align}
for some positive constant $\widetilde A$.
Then, there exists $t_0\ge 1$, depending only on $n$, $\xi$, and $\widetilde A$,
 such that the series \eqref{eq-formal-expansion-multi-index}
converges 
in the $C^k$-sense to a smooth function $\varphi$ 
on $[t_0,\infty)\times\mathbb{S}^{n-1}$, for any $k\ge 0$.
Moreover, $\varphi$ is a solution to \eqref{eq-linearized-Yamabe-at-xi} 
on $[t_0,\infty)\times \mathbb S^{n-1}$
with its asymptotic behavior determined by $\mathcal{X} =\{X_i\}_{i \geq 1}$
as in Definition \ref{def-determined-behavior}.
\end{theorem}

\begin{proof} 
The proof consists of four steps.
For simplicity, we write $P_\alpha = P_\alpha(\mathcal{X})$.
By Corollary \ref{cor-expansion-from-optimal-term-no-poly}, we have
$P_\alpha = 0$ for any $\alpha \in \ind$ with $\alpha_1\neq 0$
and $P_\alpha$ has polynomial degree 0 for any $\alpha \in \ind$ with $\alpha_1= 0$. 
We only need to consider $\alpha \in \ind$ with $\alpha_1= 0$. 

{\it Step 1. We derive a recursive estimate for $P_\alpha$.}
Consider $\alpha\in \ind$ with $\alpha_1=0$ and $|\alpha|\ge 2$. 
Set
\begin{align}
\label{eq-def-R_alpha}
R_\alpha=\sum_{j = 2}^{|\alpha|} \sum_{\beta_1 + \cdots + \beta_j = \alpha} a_j {P}_{\beta_1} \cdots {P}_{\beta_j}.
\end{align}
By Lemma \ref{lemma-decomposition-of-spherical-harmonics} 
and Corollary \ref{cor-expansion-from-optimal-term-no-poly}, 
$R_\alpha$ is a generalized polynomial with $C^\infty(\mathbb{S}^{n-1})$-coefficients 
of polynomial degree $0$ and a spherical degree at most $\langle\alpha,\N\rangle$. 
Then, \eqref{eq-recursive-equation-multi-index} implies 
$P_\alpha = \mathcal L_{\langle\alpha,\rho\rangle}^{-1} [R_\alpha].$ 
Lemma \ref{lemma-estimate-L_mu_alpha^-1} 
implies
$$\|{P}_\alpha\|_{m_*,\infty} \leq B \|R_\alpha\|_{m_*,\infty}.$$
We point out this is only a weaker version of \eqref{eq-estimate-L_mu_alpha^-1-poly}. 
Hence, by \eqref{eq-upper-bound-for-a_j}, \eqref{eq-Banach-algebra-norm-for-degree-0-poly},
and \eqref{eq-def-R_alpha}, we obtain
\begin{align}\label{eq-recursive-estimate-P-alpha-scalar-majorant}
\|{P}_\alpha\|_{m_*,\infty}  \le B   \sum_{j=2}^{|\alpha|} \sum_{\beta_1+\cdots+\beta_j = \alpha} 
M^j \|{P}_{\beta_1}\|_{m_*,\infty}\cdots \|{P}_{\beta_j}\|_{m_*,\infty},
\end{align}
for the positive constants $B$ and $M$ in \eqref{eq-estimate-L_mu_alpha^-1-poly} and \eqref{eq-upper-bound-for-a_j}, respectively. 

{\it Step 2. We prove that the series \eqref{eq-formal-expansion-multi-index} 
converges uniformly and absolutely on $[t_0,\infty)\times\mathbb{S}^{n-1}$,
for some $t_0\ge 1$ depending only on $n$, $\xi$, and $\widetilde A$.}
By a simple comparison of the $\|\cdot\|_{m_*}$-norm and the
$\|\cdot\|_{L^2(\mathbb S^{n-1})}$-norm for spherical harmonics,
we have, for any $X \in \mathscr{E}_i$, 
$$\|X\|_{m_*} \le C\lambda_i^{m_*/2} \|X\|_{L^2(\mathbb{S}^{n-1})} \leq \widetilde C^i \|X\|_{L^2(\mathbb{S}^{n-1})},$$
for some positive constants $C$ and $\widetilde C$ depending only on $n$.
By Lemma \ref{lemma-admissible-sequence-of-bases}(ii)
and \eqref{eq-exponential-growth-condition-expansion-from-optimal-term}, we have, for any $i \geq 2$,
\begin{align}
\label{eq-exponential-growth-condition-H^m-norm-expansion-from-2nd-term}
\|P_{e_i}\|_{m_*,\infty} =  |p_i^+|_{L^\infty(\R)} \|X_i\|_{m_*} \leq  A^i,
\end{align}
for some constant $ A> 0$ 
depending only on $n$, $\xi$, and $\widetilde A$.

Now, let $A$ be as in \eqref{eq-exponential-growth-condition-H^m-norm-expansion-from-2nd-term},
$B$ and $M$ be as in \eqref{eq-recursive-estimate-P-alpha-scalar-majorant},
and $A_\alpha = A_\alpha(A,B,M)$ be the constant in Lemma \ref{lemma-majorant-series}.
Take any $\alpha \in \ind$ with $\alpha_1=0$. 
Then, by \eqref{eq-A_alpha}, \eqref{eq-recursive-estimate-P-alpha-scalar-majorant}, 
\eqref{eq-exponential-growth-condition-H^m-norm-expansion-from-2nd-term}, and a simple induction,
we get
\begin{align}\label{eq-P_alpha-estimate-last}
\|{P}_\alpha\|_{m_*,\infty} \leq A_\alpha.
\end{align} 
By the Sobolev embedding and  $\rho_i > i$ for any $i \ge 2$, we have, 
for any $(t,\theta) \in [t_0,\infty)\times \mathbb S^{n-1}$,
\begin{align*}
|P_\alpha(t,\theta) e^{-\langle\alpha,\rho\rangle t}|
\leq CA_\alpha e^{-\langle \alpha,\N\rangle t},
\end{align*}
for some constant $C\ge 1$ depending only on $n$.
Let $D$ be the constant, depending only on $A$, $B$, and $M$, in Lemma \ref{lemma-majorant-series}.
By \eqref{eq-majorant-series},
we obtain,
for any $\ell \ge 1$ and $(t,\theta) \in [t_0,\infty)\times \mathbb S^{n-1}$,
\begin{align}
\label{eq-upper-bound-P_alpha e^-alpha,rho-sum-ell}
 \sum_{\langle\alpha,\N\rangle = \ell}|P_\alpha(t,\theta) e^{-\langle\alpha,\rho\rangle t}|
\le C\sum_{\langle\alpha,\N\rangle = \ell} A_\alpha  e^{-\langle \alpha,\N\rangle t}
\le C(D  e^{-t})^\ell.
\end{align}
Take $t_0$ large so that $D e^{-t} \le 1/2$ for any $t \ge t_0$.
Then, the series \eqref{eq-formal-expansion-multi-index} 
converges uniformly and absolutely on $[t_0,\infty)\times\mathbb S^{n-1}$.

{\it Step 3. We improve the convergence to the $C^k$-sense, for any $k \ge 1$.}
Let $A_\alpha$, $D$, and $t_0$ be the constants in Step 2.
Set
\begin{align}
\label{eq-psi_alpha=P_alpha exp}
\psi_\alpha(t,\theta)= P_\alpha(t,\theta)e^{-\langle\alpha,\rho\rangle t}.
\end{align}
We first derive an upper bound of $t$-derivatives of $\psi_\alpha$ of any order.
To this end, define, for $\ell\ge2$ and $r\ge0$,
\begin{align}
\label{eq-def-B-l-r}
B_{\ell,r}
=
\frac{1}{r!}
\sum_{\substack{\langle\alpha,\N\rangle = \ell}}
\sup_{t\ge t_0}
\|(\partial_t^r \psi_\alpha)(t,\cdot)\|_{m_*}.
\end{align}
Note that the sum is over $\alpha \in \ind$ with $\alpha_1 = 0$.
Since $\langle\alpha,\N\rangle \ge 2$ for such $\alpha$, 
the index of $\ell$ starts from $2$.
We will prove that
\begin{align}
\label{eq-exp-upper-bound-B-l-r}
B_{\ell,r} \le K^{\ell+r},
\end{align}
for some positive constant $K$ depending only on $n$, $\xi$, and
$\widetilde A$.
For simplicity, denote by $C$ (different from line to line) a general
constant depending only on $n$ and $\xi$.

By \eqref{eq-recursive-equation-multi-index},
we have,  for any $\alpha \in \ind$, 
\begin{equation}
\label{eq-recursive-equation-f-alpha}
\begin{aligned}
\partial_{tt} \psi_\alpha + \Delta_\theta \psi_\alpha - a^2 \psi_\alpha + V(t) \psi_\alpha
=\sum_{j=2}^{|\alpha|} \sum_{\beta_1 +\cdots + \beta_j=\alpha}
a_j \psi_{\beta_1}\cdots \psi_{\beta_j},
\end{aligned}
\end{equation}
where $a = (n-2)/2$ and
$V = \frac{1}{4} n(n+2) \xi^{\frac{4}{n-2}}$.
Fix $r \ge 0$. Applying $\partial_t^r$ to
\eqref{eq-recursive-equation-f-alpha}, we get
\begin{align}
\label{eq-inductive-recursive-PDE-time-derivatives}
\partial_t^{r+2} \psi_\alpha
=&
-\Delta_\theta\partial_t^r \psi_\alpha
+a^2\partial_t^r \psi_\alpha
-\partial_t^r(V\psi_\alpha)+
\sum_{j = 2}^{|\alpha|}
\partial_t^r
\Big(
a_j
\sum_{\beta_1+\cdots+\beta_j=\alpha}
\psi_{\beta_1}\cdots \psi_{\beta_j}
\Big).
\end{align}
Note that $V$ depends only on $t$. Hence,
\[
\frac{1}{r!}\partial_t^r(V\psi_\alpha)
=
\sum_{j=0}^r
\frac{V^{(j)}}{j!}
\frac{\partial_t^{r-j}\psi_\alpha}{(r-j)!}.
\]
Similarly, we get
\begin{align*}
\frac{1}{r!}
\partial_t^r(a_j \psi_{\beta_1}\cdots \psi_{\beta_j})
=
\sum_{\substack{p+r_1+\cdots+r_j=r}}
\frac{1}{p!} a_j^{(p)}
\prod_{\nu=1}^j
\frac{1}{r_\nu!}
\partial_t^{r_\nu}\psi_{\beta_\nu}.
\end{align*}
On the other hand, since $P_\alpha$ has a spherical degree at most
$\langle\alpha,\N\rangle$, we have
\[
\|(\Delta_\theta \partial_t^r \psi_\alpha)(t,\cdot)\|_{m_*}
\le 
\lambda_{\langle\alpha,\N\rangle}
\|(\partial_t^r \psi_\alpha)(t,\cdot)\|_{m_*},
\]
where we used the orthogonality for spherical harmonics of different degrees.
Since $V$ is an analytic periodic function on $\R$, we get, for any $j \ge 0$,
$$|V^{(j)}|_{L^\infty(\R)} \le C^{j+1} j!,$$
for some positive constant $C$ depending only on $n$ and $\xi$.
Dividing \eqref{eq-inductive-recursive-PDE-time-derivatives} by $r!$
and then summing over $\langle\alpha,\N\rangle = \ell$
with $\alpha_1 = 0$,
we have
\begin{equation}
\label{eq-recurrence-B-l-r}
\begin{aligned}
(r+2)(r+1)B_{\ell,r+2}
&\le
\big(\lambda_\ell+a^2\big)B_{\ell,r}
+
C\sum_{j=0}^r C^j B_{\ell,r-j}
+\mathcal N,
\end{aligned}
\end{equation}
where
\begin{align}
\label{eq-complicated-nonlinearity-time-derivatives}
\mathcal N
=
\sum_{j=2}^{[\ell/2]}
M^j
\sum_{p=0}^r M^p
\sum_{\substack{\ell_1+\cdots+\ell_j=\ell\\\ell_i\ge2}}
\sum_{\substack{r_1+\cdots+r_j=r-p\\r_i\ge0}}
\prod_{\nu=1}^j
B_{\ell_\nu,r_\nu}.
\end{align}

For any $\ell \ge 2$, define $\widehat B_{\ell,0}=B_{\ell,0},
\widehat B_{\ell,1}=B_{\ell,1}$, 
and inductively for $r\ge0$ by 
\begin{align}
\label{eq-def-majorant-B-l-r}
(r+2)(r+1)\widehat B_{\ell,r+2}
=
\big(\lambda_\ell+a^2\big)\widehat B_{\ell,r}
+
C\sum_{j=0}^r C^j \widehat B_{\ell,r-j}
+\widehat{\mathcal N},
\end{align}
where $\widehat{\mathcal N}$ is given by
\eqref{eq-complicated-nonlinearity-time-derivatives}
with $B$ replaced by $\widehat B$.
By \eqref{eq-recurrence-B-l-r} and a simple induction, 
we have, for any $r\ge0$ and $\ell\ge2$,
\begin{align}
\label{eq-B-less-than-majorant-B-time}
B_{\ell,r}\le \widehat B_{\ell,r}.
\end{align}

We proceed formally. 
Consider the formal series
\[
\mathscr B(s_1,s_2)
=
\sum_{\ell\ge2}\sum_{r\ge0}
\widehat B_{\ell,r}s_1^\ell s_2^r 
\quad\text{near }0\in\mathbb R^2.
\]
By $\lambda_\ell=\ell(\ell+n-2)$, it is straightforward to check that
\eqref{eq-def-majorant-B-l-r} formally implies
\begin{equation}
\label{eq-majorant-PDE-B-time}
\begin{aligned}
\partial_{s_2s_2}\mathscr B
=
s_1^2\partial_{s_1s_1}\mathscr B
+
(n-1)s_1\partial_{s_1}\mathscr B
+
a^2\mathscr B
+
\frac{C}{1-Cs_2}\mathscr B
+
\frac{1}{1-Ms_2}
\frac{M^2\mathscr B^2}{1-M\mathscr B},
\end{aligned}
\end{equation}
with the initial values given by 
\begin{align}
\label{eq-Cauchy-data-B-time}
\mathscr B(s_1,0)
=
\sum_{\ell\ge2}B_{\ell,0}s_1^\ell,
\qquad
\partial_{s_2}\mathscr B(s_1,0)
=
\sum_{\ell\ge2}B_{\ell,1}s_1^\ell.
\end{align}

We now prove that the initial values in
\eqref{eq-Cauchy-data-B-time} define analytic functions near $s_1=0$.
By Step 2, we have
\[
B_{\ell,0}\le D^\ell,
\]
where $D$ is the constant in Lemma \ref{lemma-majorant-series}.
For $r=1$, we write
\[
\partial_t \psi_\alpha
=
-\int_t^\infty \partial_{tt}\psi_\alpha(s,\cdot)\,ds.
\]
By \eqref{eq-recursive-equation-f-alpha}, \eqref{eq-P_alpha-estimate-last},
\eqref{eq-A_alpha}, and $B \ge 1$,
we have,
for any $t\ge t_0$,
\begin{align*}
\|\partial_t \psi_\alpha(t,\cdot)\|_{m_*}
\le
\int_t^\infty
\|\partial_{tt}\psi_\alpha(s,\cdot)\|_{m_*}\,ds
\le
C\lambda_{\langle\alpha,\N\rangle}
A_\alpha
\int_t^\infty
e^{-\langle\alpha,\rho\rangle s}\,ds.
\end{align*}
Taking supremum over $t\ge t_0$ and summing over
$\langle\alpha,\N\rangle=\ell$, we obtain
\[
B_{\ell,1}\le C\ell D^\ell.
\]
Therefore, the initial values in \eqref{eq-Cauchy-data-B-time} are
analytic near $s_1=0$.

Now, by the Cauchy--Kowalevskaya theorem, the Cauchy problem
\eqref{eq-majorant-PDE-B-time} has a unique analytic solution with the
initial condition \eqref{eq-Cauchy-data-B-time} in a neighborhood of
$0\in\mathbb R^2$.
Note that \eqref{eq-majorant-PDE-B-time} can be written as a
non-characteristic first-order system for
$(\mathscr B,\partial_{s_1} \mathscr B, \partial_{s_2}\mathscr B)$. 
By the analyticity, we get, for any $r\ge0$ and $\ell\ge2$,
\[
\widehat B_{\ell,r}\le K^{\ell+r},
\]
for some constant $K>0$.
Then, \eqref{eq-exp-upper-bound-B-l-r} follows from
\eqref{eq-B-less-than-majorant-B-time}.

Finally, we write
$P_\alpha=e^{\langle\alpha,\rho\rangle t}\psi_\alpha$.
By \eqref{eq-def-B-l-r}, \eqref{eq-exp-upper-bound-B-l-r},
$\rho_i/i\to 1$ as $i\to\infty$,
and a simple estimate over one period of $P_\alpha$, we have
\begin{align}
\label{eq-exp-upper-bound-P-alpha-time-derivatives}
\sum_{\langle\alpha,\N\rangle=\ell}
\|\partial_t^rP_\alpha\|_{m_*,\infty}
\le
K_1^{\ell+r} r!,
\end{align}
where $K_1$ is a positive constant depending only on
$n$, $\xi$, and $\widetilde A$.

Now, for any $h \in C^\infty(\mathbb S^{n-1})$ and $k \ge 0$,
by the Sobolev embedding and standard estimates, we have
$$|\nabla_\theta^k h|_{L^\infty (\mathbb S^{n-1})} 
\le C_k \big\{\|h\|_{m_*} + \|\Delta_\theta^k h\|_{m_*}\big\},$$
for some constant $C_k$ depending only on $n$ and $k$.
Then, for any $p,k\ge0$, since $P_\alpha$ has a spherical degree
at most $\langle\alpha,\N\rangle$, 
\eqref{eq-exp-upper-bound-P-alpha-time-derivatives} yields
\[
\sum_{\langle\alpha,\N\rangle=\ell}
|\partial_t^p\nabla_\theta^kP_\alpha|_{L^\infty(\R\times\mathbb S^{n-1})}
\le
2C_k\lambda_\ell^{k}K_1^{\ell+p} p!.
\]
Then, we obtain
\[
\sum_{\langle\alpha,\N\rangle=\ell}
|\partial_t^p\nabla_\theta^k
(P_\alpha(t,\theta)e^{-\langle\alpha,\rho\rangle t})|
\le
C_{p,k}( 2K_1e^{-t})^\ell.
\]
Therefore, for $t_0$ sufficiently large such that
$K_1e^{-t}<1/4$ for any $t\ge t_0$, the series 
\eqref{eq-formal-expansion-multi-index} 
converges in the
$C^k$-sense on $[t_0,\infty)\times\mathbb S^{n-1}$, for any $k\ge0$.

{\it Step 4. We prove that the asymptotic behavior of $\varphi$ is determined by $\mathcal X$
as in Definition \ref{def-determined-behavior}.}
Define $P_\ell$ by \eqref{eq-P-ell-expression}, for any $\ell\ge 1$.
We now fix an $\ell\ge 2$. Then,
$$\varphi - \sum_{j=1}^\ell P_j(t,\theta) e^{-\mu_j t} = \sum_{\alpha,\langle\alpha,\rho\rangle \ge \mu_{\ell+1}} 
P_\alpha(t,\theta) e^{-\langle\alpha,\rho\rangle t}.$$
We claim that, for any $0<\eta < \mu_{\ell+1} - \mu_{\ell}$,
\begin{align}
\label{eq-exponential-estimate-tail-term}
\Big|\sum_{\alpha,\langle\alpha,\rho\rangle \ge \mu_{\ell+1}} P_\alpha(t,\theta) e^{-\langle\alpha,\rho\rangle t}\Big| 
= O(e^{-(\mu_{\ell+1} - \eta) t}).
\end{align}
Choose $K>0$ such that $\frac{\mu_{\ell+1}}{K} < \eta$, 
i.e., $\mu_{\ell+1}-\eta < \frac{K-1}{K} \mu_{\ell+1}$. 
Then, for any $\langle\alpha,\rho\rangle \geq \mu_{\ell+1}$, 
$$\frac{\langle\alpha,\rho\rangle}{K} < \langle\alpha,\rho\rangle - \mu_{\ell+1} +\eta.$$
Hence, for any $(t,\theta) \in [t_0,\infty)\times \mathbb S^{n-1}$,
\begin{align*}
\Big|\sum_{\alpha,\langle\alpha,\rho\rangle \ge \mu_{\ell+1}} P_\alpha(t,\theta) e^{-\langle\alpha,\rho\rangle t}\Big| 
\leq  e^{-(\mu_{\ell+1} - \eta) t}\sum_{\alpha,\langle\alpha,\rho\rangle \ge \mu_{\ell+1}} 
|P_\alpha(t,\theta)| e^{-\frac{\langle\alpha,\rho\rangle}{K}t}.
\end{align*}
Let $D$ be the constant, depending only on $A$, $B$, and $M$, in Lemma \ref{lemma-majorant-series}.
Take a positive constant $t_{0,K}\ge t_0$ such that $D e^{-t/K} \le 1/2$
for any $t \ge t_{0,K}$.
By proceeding similarly as in Step 2, we conclude that the last series converges uniformly on 
$[t_{0,K},\infty)\times\mathbb S^{n-1}$.
This implies \eqref{eq-exponential-estimate-tail-term}.
In other words, \eqref{eq-higher-order-expansion-inequality-mu} holds 
for 
any $\mu\in (\mu_{\ell}, \mu_{\ell+1})$. 
Then, \eqref{eq-higher-order-expansion-inequality} follows for any $\ell\ge 2$. 
\end{proof}

\subsection{Representations by Convergent Series}\label{subsec-analyticity-estimates}
In this subsection, we 
consider solutions $\varphi$ 
of \eqref{eq-linearized-Yamabe-at-xi} 
on $[t_0,\infty)\times\mathbb{S}^{n-1}$, 
satisfying \eqref{eq-exp-convergence}
for some constants $\gamma\in (1,2]$ and $C_*>0$.
We will improve Theorem \ref{thrm-higher-order-expansion}
for such $\varphi$ and prove that $\varphi$ is given by a convergent series.
Note that solutions constructed in Theorem \ref{thrm-convergence-of-the-expansion-from-optimal-term} 
satisfy \eqref{eq-exp-convergence} 
with $\gamma=2$.

Denote by 
$\Pi_i: L^2(\mathbb S^{n-1}) \rightarrow \mathscr E_i$ 
the orthogonal projection onto $\mathscr E_i$.
Recall the weighted H\"older spaces on $[t_0,\infty)\times\mathbb S^{n-1}$.
For $t_0\ge 0$, $k\in\mathbb Z_+$, $\alpha\in(0,1)$, and
$\mu>0$, define
\[
|v|_{C^k_\mu([t_0,\infty)\times\mathbb S^{n-1})}
=
\sum_{j=0}^k
\sup_{(t,\theta)\in[t_0,\infty)\times\mathbb S^{n-1}}
e^{\mu t}|\nabla^jv(t,\theta)|,
\]
and
\[
\begin{aligned}
|v|_{C^{k,\alpha}_\mu([t_0,\infty)\times\mathbb S^{n-1})}
=
|v|_{C^k_\mu([t_0,\infty)\times\mathbb S^{n-1})}
+
\sup_{t\ge t_0+1}
e^{\mu t}
[\nabla^k v]_{C^\alpha([t-1,t+1]\times\mathbb S^{n-1})}.
\end{aligned}
\]
We denote by
$C^{k,\alpha}_\mu([t_0,\infty)\times\mathbb S^{n-1})$
the space of functions
$v\in C^k([t_0,\infty)\times\mathbb S^{n-1})$
with $|v|_{C^{k,\alpha}_\mu([t_0,\infty)\times\mathbb S^{n-1})} < \infty$.
We first prove the following technical uniqueness result.
For simplicity, we write $\Pi_{\ge 2} = \operatorname{Id} - \Pi_0 - \Pi_1$.

\begin{proposition}
\label{prop-uniqueness-of-Dirichlet}
Let 
$\varphi, \widetilde\varphi$ be exponential decay solutions of \eqref{eq-linearized-Yamabe-at-xi} 
on $[t_0,\infty) \times \mathbb{S}^{n-1}$, for some $t_0 \ge 0$.
Assume $\varphi,\widetilde\varphi = O(e^{-\mu t})$ for some $1 < \mu < \rho_2$ and
$\Pi_{\ge 2}\varphi(t_0,\cdot)
= \Pi_{\ge 2} \widetilde \varphi(t_0,\cdot)$ 
on $\mathbb S^{n-1}$.
Then, there exists a positive constant $\epsilon_0$, 
depending only on $n$, $\xi$, and $\mu$,
such that if
\begin{align}
\label{eq-smallness-uniqueness-Dirichlet}
|\varphi|_{L^\infty([t_0,\infty)\times\mathbb S^{n-1})} + |\widetilde\varphi|_{L^\infty([t_0,\infty)\times\mathbb S^{n-1})} <\epsilon_0,
\end{align}
then $\varphi = \widetilde \varphi$ in $[t_0,\infty) \times \mathbb{S}^{n-1}$.
\end{proposition}

\begin{proof}
We divide the proof into two steps.

\textit{Step 1. We prove an estimate for the 
Dirichlet problem for $\mathcal L$ with the zero boundary condition.}
Assume $f\in C^{0,\alpha}_\mu([t_0,\infty)\times\mathbb S^{n-1})$ and let
$v\in C^{2,\alpha}_\mu([t_0,\infty)\times\mathbb S^{n-1})$ satisfy
\begin{align*}
\mathcal L v&=f\quad\text{in } (t_0,\infty)\times\mathbb S^{n-1},\\
v&=0\quad\text{on } \{t_0\}\times\mathbb S^{n-1}.
\end{align*}
We claim
\begin{equation}
\label{eq-improved-Dirichlet-C0-estimate}
|v|_{C^0_\mu([t_0,\infty)\times\mathbb S^{n-1})}
\le
C
|f|_{C^0_\mu([t_0,\infty)\times\mathbb S^{n-1})},
\end{equation}
where $C$ is a positive constant depending only on $n$, $\xi$, $\mu$, and $\alpha$,
independent of $t_0$.

The proof is essentially the same as that of Lemma 2.6\cite{HAN2021127},
adapted from \cite{Mazzeo1996ConstantSC}.
We only give a brief sketch. Suppose that
\eqref{eq-improved-Dirichlet-C0-estimate} fails. The normalization,
translation, and compactness argument in the proof of 
Lemma 2.6\cite{HAN2021127} then gives a nonzero solution $\widetilde v$ of
$$
\mathcal L_{\xi(\cdot+\tau)}\widetilde v=0 \quad\text{on }[a,\infty)\times\mathbb S^{n-1},
$$
for some $a\in[-\infty,0)$, where $\mathcal L_{\xi(\cdot + \tau)}$ 
is the linearized operator $\mathcal L$ in \eqref{eq-linear-operator-of-Yamabe}
with $\xi$ replaced by $\xi(\cdot + \tau)$.
Moreover, $\widetilde v$ satisfies, 
for any $(t,\theta)$ in the domain of $\widetilde v$,
\begin{equation}
\label{eq-improved-Dirichlet-limit-bound}
|\widetilde v(t,\theta)|\le e^{-\mu t}.
\end{equation}
If $a>-\infty$, then $\widetilde v(a,\cdot)=0$.
The uniqueness in Lemma 2.4\cite{HAN2021127} implies that $\widetilde v=0$, which
is a contradiction.
In the case $a=-\infty$, $\widetilde v$ is defined on
$\mathbb R\times\mathbb S^{n-1}$. 
For each $i \ge 0$, 
let $d_i = \dim\mathscr E_i$ and $\{Y_{i1},\cdots, Y_{id_i}\}$ 
be an orthonormal basis of $\mathscr E_i$ in $L^2(\mathbb S^{n-1})$.
Write 
$$\Pi_i \widetilde v(t,\cdot) = \sum_{j=1}^{d_i} \widetilde v_{ij}(t) Y_{ij},$$
for some $\widetilde v_{ij} \in C^\infty(\R)$.
Then, $L_i \widetilde v_{ij} = 0$ on $\R$, and $|\widetilde v_{ij}(t)| \le Ce^{-\mu t}$
by \eqref{eq-improved-Dirichlet-limit-bound}, for any $t \in \R$,
By a simple analysis of linear ODEs,
we obtain $\widetilde v_{ij} = 0$ for any $i \ge 0$ and $j = 1,\cdots, d_i$.
Thus $\widetilde v=0$, again a contradiction. This proves
\eqref{eq-improved-Dirichlet-C0-estimate}.

\textit{Step 2. We prove $\varphi = \widetilde \varphi$ under the assumption \eqref{eq-smallness-uniqueness-Dirichlet}.}
Set
$
w=\varphi-\widetilde\varphi.
$
Then, 
\begin{equation}
\label{eq-difference-equation-uniqueness}
\mathcal L w=Bw,
\end{equation}
where $B$ is given by \eqref{eq-expression-B-difference-linear-eq}.
The mean value theorem implies 
\begin{equation}
\label{eq-B-smallness-uniqueness}
|B|_{L^\infty([t_0,\infty)\times\mathbb S^{n-1})}
\le C\big\{|\varphi|_{L^\infty([t_0,\infty)\times\mathbb S^{n-1})}
+|\widetilde\varphi|_{L^\infty([t_0,\infty)\times\mathbb S^{n-1})}\big\},
\end{equation}
for some constant $C$ depending only on $n$ and $\xi$.
By the same argument as in the proof of 
Theorem \ref{thrm-first-decay-C^m-norm}, we have $w \in C^k_\mu([t_0,\infty)\times\mathbb S^{n-1})$ for any $k \ge 0$.
Decompose $w$ as
$$
w=w_0 +w_1+w_\perp,
\quad\text{with }w_0 = \Pi_0w,\,w_1=\Pi_1w,\text{ and }
w_\perp=\Pi_{\ge 2}w.
$$
Since $\mathcal L$ commutes with spherical projections,
\eqref{eq-difference-equation-uniqueness} gives
\begin{align}\label{eq-difference-equation-decomposition-uniqueness}
\mathcal L w_0=\Pi_0(Bw),\quad
\mathcal L w_1=\Pi_1(Bw),\quad\text{and}\quad
\mathcal L w_\perp=\Pi_{\ge 2}(Bw).
\end{align}
By the assumption on the initial values,
we get $w_\perp(t_0,\cdot)=0$.
Hence, \eqref{eq-improved-Dirichlet-C0-estimate} yields
\begin{equation}\label{eq-complement-uniqueness-estimate}
\begin{aligned}
|w_\perp|_{C^0_\mu([t_0,\infty)\times\mathbb S^{n-1})}
&\le
C|Bw|_{C^0_\mu([t_0,\infty)\times\mathbb S^{n-1})}\\
&\le
C|B|_{L^\infty([t_0,\infty)\times\mathbb S^{n-1})}|w|_{C^0_\mu([t_0,\infty)\times\mathbb S^{n-1})}.
\end{aligned}
\end{equation}
We next estimate $w_0$ and $w_1$. Write 
$$\Pi_0 = \langle\cdot,Y_{01}\rangle_{L^2(\mathbb S^{n-1})} Y_{01},\quad\Pi_1 
= \sum_{j=1}^n \langle\cdot,Y_{1j}\rangle_{L^2(\mathbb S^{n-1})} Y_{1j}.$$
Then, the equation for $w_0$ and $w_1$
in \eqref{eq-difference-equation-decomposition-uniqueness} 
is equivalent to
$$L_0 w_{01} = f_0,\quad L_1 w_{1j} = f_j\quad\text{for }j = 1,\cdots, n,$$
where
\begin{align*}
w_{01}(t) = \langle w(t,\cdot),Y_{01}\rangle_{L^2(\mathbb S^{n-1})},\quad w_{1j}(t) = \langle w(t,\cdot),Y_{1j}\rangle_{L^2(\mathbb S^{n-1})},
\end{align*}
and $f_0$, $f_j$ are the corresponding projections of $Bw$.
Since $\mu >\rho_1 = 1$, we have $w_{01} = L_0^{-1} f_0$ and
$w_{1j} = L_1^{-1} f_j$.
A simple estimate yields
\begin{equation}
\label{eq-first-mode-uniqueness-estimate}
|w_0|_{C^0_\mu([t_0,\infty)\times\mathbb S^{n-1})} +|w_1|_{C^0_\mu([t_0,\infty)\times\mathbb S^{n-1})}
\le
C|B|_{L^\infty([t_0,\infty)\times\mathbb S^{n-1})}|w|_{C^0_\mu([t_0,\infty)\times\mathbb S^{n-1})}.
\end{equation}
Combining \eqref{eq-complement-uniqueness-estimate} and
\eqref{eq-first-mode-uniqueness-estimate}, we conclude that
$$
|w|_{C^0_\mu([t_0,\infty)\times\mathbb S^{n-1})}
\le
C|B|_{L^\infty([t_0,\infty)\times\mathbb S^{n-1})}
|w|_{C^0_\mu([t_0,\infty)\times\mathbb S^{n-1})},
$$
for some positive constant $C$ depending only on $n$, $\xi$, and $\mu$.
By \eqref{eq-B-smallness-uniqueness}, we may choose sufficiently small
$\epsilon_0>0$, depending only on $n$, $\xi$, and $\mu$,  
so that \eqref{eq-smallness-uniqueness-Dirichlet} implies
$$
C|B|_{L^\infty([t_0,\infty)\times\mathbb S^{n-1})}<1.
$$
It follows that $w=0$, and hence $\varphi=\widetilde\varphi$
on $[t_0,\infty)\times\mathbb S^{n-1}$.
\end{proof}

For the next result, we fix a sequence of admissible bases $\{\psi_i^+, \psi_i^-\}_{i\geq 0}$ 
for $\{\mathrm{Ker}(L_i)\}_{i\geq 0}$ as in Lemma \ref{lemma-admissible-sequence-of-bases},
with $p_i^+(0) = 1$ for any $i \ge 1$.

\begin{theorem}\label{thm:actual-solution-series-improved}
Assume $\varphi$ is a smooth solution to
\eqref{eq-linearized-Yamabe-at-xi} on
$[t_0,\infty)\times\mathbb S^{n-1}$
satisfying \eqref{eq-exp-convergence} 
for some
$\gamma\in(1,2]$ and $C_*>0$.
Then, there exists a sequence
$\mathcal X=\{X_i\}_{i\ge2}$ of spherical harmonics with
$X_i\in\mathscr E_i$ such that
\begin{align}
\label{eq-actual-solution-series-improved}
\varphi=\sum_{\ell=2}^\infty P_\ell(\mathcal X)e^{-\mu_\ell t},
\end{align}
where $\{P_\ell(\mathcal X)\}_{\ell\ge2}$ is the sequence of
generalized polynomials determined by $\mathcal X$ as in
Definition \ref{def-sequence-P-determined-by-X},
and the series converges in the $C^k$-sense on
$[t_1,\infty)\times\mathbb S^{n-1}$, for some $t_1\ge t_0$
and any $k\ge0$.
Moreover, each $P_\ell$ has polynomial degree $0$ for $\ell\ge2$,
and \eqref{eq-exponential-growth-condition-expansion-from-optimal-term}
holds for any $i\ge2$ and some positive constant $\widetilde A$.
\end{theorem}

\begin{proof} 
Let $\mathcal{X} = \{X_i\}_{i\geq 1}$ be the sequence of spherical harmonics 
with $X_i \in \mathscr{E}_i$
and $\{P_\ell(\mathcal{X})\}_{\ell\ge 1}$ be the sequence of generalized polynomials 
determined by $\mathcal{X}$ as in Theorem \ref{thrm-higher-order-expansion}.  
In particular, \eqref{eq-higher-order-expansion-inequality} holds. 
The condition \eqref{eq-exp-convergence} with $\gamma\in (1,2]$ implies $X_1=0$. 
By Corollary \ref{cor-expansion-from-optimal-term-no-poly}, 
we conclude that each $P_\ell(\mathcal{X})$  has polynomial degree $0$. 
Once \eqref{eq-exponential-growth-condition-expansion-from-optimal-term}
is established, \eqref{eq-P-ell-expression} and Theorem
\ref{thrm-convergence-of-the-expansion-from-optimal-term}
imply that the series in \eqref{eq-actual-solution-series-improved}
converges to a smooth solution with the same asymptotic behavior as $\varphi$.
The equality in \eqref{eq-actual-solution-series-improved} then follows from
Theorem \ref{thrm-uniqueness-exponential-decay-sol}.
We next prove \eqref{eq-exponential-growth-condition-expansion-from-optimal-term} in four steps.

{\it Step 1. We introduce successive sequences.} 
Let $T$ be the common period of $\xi$ and $\{p_i^+\}_{i\ge 1}$.
For simplicity, we denote by $C\ge 1$ (different from line to line) a 
general constant depending only on $n$ and $\xi$.
Fix a large integer $n_0$ and define
$$\widetilde \varphi(t,\theta)=\varphi(t + n_0T,\theta),$$
and, for any $i \ge 2$,
$$g_i=\Pi_i\widetilde\varphi(0,\cdot) = \Pi_i \varphi(n_0T,\cdot).$$
Note that, for any $h\in H^{m_*+1}(\mathbb S^{n-1})$,
by \eqref{eq-comparison-norm_*-and-norm_*1} and
the Cauchy inequality, 
\begin{align}
\label{eq:real-slice-sum}
 \sum_{i\ge2}\norm{\Pi_i h}_{m_*}
\le
 \Big(\sum_{i\ge2}(1+\lambda_i)^{-1}\Big)^{1/2}
 \Big(\sum_{i\ge2}(1+\lambda_i)
 \norm{\Pi_i h}_{m_*}^2\Big)^{1/2}
 \le C\norm{h}_{m_*+1}.
\end{align}
The uniform bound of $\{p_i^+\}_{i\ge 1}$ in
Lemma \ref{lemma-admissible-sequence-of-bases} yields
\begin{equation}
\label{eq:real-small-initial-sum}
 \sum_{i\ge2} \|p_i^+ g_i\|_{m_*, \infty}
 \le C \sum_{i\ge 2} \|\Pi_i \varphi(n_0T,\cdot)\|_{m_*} \le C\|\varphi(n_0T,\cdot)\|_{m_*+1}.
\end{equation}
By Theorem \ref{thrm-first-decay-C^m-norm}, we have
\begin{align}
\label{eq-initial-sum-upper-bound-small}
\|\varphi(n_0T,\cdot)\|_{m_*+1} \rightarrow 0 \quad\text{as } n_0\rightarrow \infty.
\end{align}

For a family $\mathcal Q = \{Q_\alpha\}_{\alpha \in \ind}$ 
of generalized polynomials, define its formal recurrence by
$R_{e_i}(\mathcal Q) = 0$ for any $i \ge 1$ and, for any $|\alpha| \ge 2$,
\begin{equation}\label{eq:real-Galpha}
 R_\alpha(\mathcal Q)
 =\sum_{j=2}^{|\alpha|}a_j
 \sum_{\beta_1+\cdots+\beta_j=\alpha}
 Q_{\beta_1}\cdots Q_{\beta_j}.
\end{equation}
Start with $\mathcal Q^{(0)} = \{Q^{(0)}_\alpha\}_{\alpha \in \ind}$ given by
$Q_\alpha^{(0)} = 0$ if $\alpha_1 \ne 0$ and
\begin{align}
\label{eq-P_alpha^0}
Q_\alpha^{(0)} = 
\begin{cases}
p_i^+ g_i & \text{if } \alpha = e_i,\,i\ge 2,\\
0 & \text{if } |\alpha| \ge 2,\,\alpha_1 = 0.
\end{cases}
\end{align}
Define formally $\mathcal Q^{(m)} = \{Q_\alpha^{(m)}\}_{\alpha \in \ind}$
inductively by $Q_\alpha^{(m+1)} = 0$ if $\alpha_1 \ne 0$,
for $|\alpha|\ge 2$ with $\alpha_1 = 0$,
\begin{align}
\label{eq:real-coefficient-iteration-large-|alpha|}
Q_\alpha^{(m+1)} = \mathcal L^{-1}_{\langle\alpha,\rho\rangle} \big[ R_\alpha(\mathcal Q^{(m)}) \big],
\end{align}
and then, for $\alpha = e_i$ with $i \ge 2$,
\begin{equation}\label{eq:real-coefficient-iteration-|alpha|=1}
 \begin{aligned}
Q_\alpha^{(m+1)} = 
p_i^+\Big[
 g_i-\sum_{|\beta|\ge2}\Pi_i \big(Q_\beta^{(m+1)}(0,\cdot)\big)\Big] .
 \end{aligned}
\end{equation}
We point out that \eqref{eq:real-coefficient-iteration-|alpha|=1} formally ensures
\begin{equation}\label{eq:real-exact-slice}
\Pi_{\ge 2}\sum_{\alpha\in\ind} Q_\alpha^{(m)}(0,\cdot)
 =\Pi_{\ge 2}\widetilde\varphi(0,\cdot).
\end{equation}
Moreover, by a simple induction as in the proof of
Proposition \ref{prop-formal-expansion-multi-index}
and Corollary \ref{cor-expansion-from-optimal-term-no-poly},
$Q_\alpha^{(m)}$ has polynomial degree $0$ and a spherical degree  
at most $\langle \alpha,\N\rangle$, for any $\alpha \in \ind$ and $m\ge 0$.
We should emphasize that $g_1 = \Pi_1 \widetilde\varphi(0,\cdot)$ may not be zero,
but is not used in the construction of $Q_\alpha^{(m)}$.

{\it Step 2. We prove that $\mathcal Q^{(m)}$ is well-defined, for each $m \ge 0$.}
For simplicity, we set
\[
 K_m=\sum_{\alpha \in \ind} \|Q_\alpha^{(m)}\|_{m_*, \infty}. 
\]
We point out that the summation
is over $\alpha \in \ind$ with $\alpha_1 = 0$.
In particular, by \eqref{eq-P_alpha^0},
\begin{align}
\label{eq-K_0-initial-sum}
K_0 = \sum_{i \ge 2} \|p_i^+g_i\|_{m_*, \infty}.
\end{align}
First, note that, by \eqref{eq:real-Galpha}, 
\begin{equation}\label{eq:real-forcing-sum}
 \begin{aligned}
\sum_{|\alpha| \geq 2} \|R_\alpha( \mathcal Q^{(m)})\|_{m_*, \infty} 
&\leq  \sum_{\ell \geq 2} M^{\ell} \sum_{|\alpha| \geq \ell} \sum_{\beta_1+\cdots+\beta_{\ell}=\alpha}  
\|Q_{\beta_1}^{(m)}\|_{m_*, \infty}\cdots \| Q_{\beta_\ell}^{(m)}\|_{m_*, \infty}\\
& =\sum_{\ell \geq 2} M^{\ell} K_m^{\ell},
\end{aligned}
\end{equation}
where $M$ is the constant in \eqref{eq-upper-bound-for-a_j}.
On the other hand, for $|\beta| \ge 2$ with $\beta_1 = 0$,
by \eqref{eq:real-slice-sum}, \eqref{eq:real-coefficient-iteration-large-|alpha|},
and \eqref{eq-estimate-L_mu_alpha^-1-poly}, we have 
\begin{equation}
\label{eq:real-boundary-correction}
\begin{aligned}
\sum_{i\ge2}
 \big\|p^+_i\Pi_i\big(Q_\beta^{(m+1)}(0,\cdot)\big)\big\|_{m_*, \infty} &=
 \sum_{i\ge2}
 \big\|p^+_i\Pi_i\big(\mathcal L^{-1}_{\langle\beta,\rho\rangle} \big[ R_\beta(\mathcal Q^{(m)})\big](0,\cdot)\big)\big\|_{m_*, \infty}\\
 &\le C\big\|\mathcal L^{-1}_{\langle\beta,\rho\rangle} \big[R_\beta(\mathcal Q^{(m)})\big](0,\cdot)\big\|_{m_*+1}\\
 &\le C\|R_\beta(\mathcal Q^{(m)})\|_{m_*, \infty}.
\end{aligned}
\end{equation}
Moreover, by \eqref{eq:real-coefficient-iteration-|alpha|=1} and \eqref{eq-K_0-initial-sum},
we get
$$\sum_{i\ge 2} \|Q_{e_i}^{(m+1)}\|_{m_*, \infty} \le K_0 
+ \sum_{i\ge 2} \sum_{|\beta| \ge 2} \big\| p_i^+\Pi_i \big(Q_\beta^{(m+1)}(0,\cdot)\big)\big\|_{m_*, \infty},$$
and a simple estimate of \eqref{eq:real-coefficient-iteration-large-|alpha|}
using \eqref{eq-estimate-L_mu_alpha^-1-poly} yields
$$\sum_{|\alpha| \ge 2} \|Q_\alpha^{(m+1)}\|_{m_*, \infty} 
\le C \sum_{|\alpha| \ge 2} \|R_\alpha(\mathcal Q^{(m)})\|_{m_*, \infty}.$$
Combining these two estimates with \eqref{eq:real-forcing-sum} and \eqref{eq:real-boundary-correction},
we have
\begin{equation}\label{eq:real-scalar-iteration}
 K_{m+1}\le K_0+
 C\sum_{\ell\ge2}M^\ell K_m^\ell.
\end{equation}
Choose $\varepsilon>0$ small such that
\begin{align}
\label{eq-choice-of-epsilon-regularity-infty}
 6CM^2\varepsilon\le\frac12,
\end{align}
and then, by \eqref{eq-K_0-initial-sum}, \eqref{eq:real-small-initial-sum},
and \eqref{eq-initial-sum-upper-bound-small}, choose $n_0$ large so that $K_0\le\varepsilon/2$.
A simple induction based on \eqref{eq:real-scalar-iteration} yields,  for any $m\ge 0$, 
\begin{align}
\label{eq-K_m-small}
K_m\le\varepsilon.
\end{align}
In particular, $\mathcal Q^{(m)}$ is well-defined for any $m \ge 0$.

{\it Step 3. We prove that $\{Q_\alpha^{(m)}\}_{m\ge 0}$ forms a Cauchy sequence with respect to $\|\cdot\|_{m_*, \infty}$.}
For any $m \ge 0$, define
\begin{equation}\label{eq:real-successive-difference-sum}
 D_m=\sum_{\alpha \in \ind}
 \|Q_\alpha^{(m+1)}-Q_\alpha^{(m)}\|_{m_*, \infty}.
\end{equation}
For $m\ge 1$ and $\abs{\alpha}\ge2$ with $\alpha_1 = 0$,
\eqref{eq:real-coefficient-iteration-large-|alpha|} yields
\[
 Q_\alpha^{(m+1)}-Q_\alpha^{(m)}
 =\mathcal L^{-1}_{\langle\alpha,\rho\rangle}\big[ R_\alpha (\mathcal Q^{(m)}) - R_\alpha(\mathcal Q^{(m-1)})\big].
\]
Thus, by \eqref{eq-estimate-L_mu_alpha^-1-poly}, we have
\begin{align}
\label{eq:real-nonlinear-difference-bound}
\sum_{|\alpha|\ge2}
 \|Q_\alpha^{(m+1)}-Q_\alpha^{(m)}\|_{m_*+1, \infty}\le
 C\sum_{|\alpha|\ge2}
 \|R_\alpha (\mathcal Q^{(m)}) - R_\alpha(\mathcal Q^{(m-1)})\|_{m_*, \infty}.
\end{align}
On the other hand, for $\alpha = e_i$, $i \ge 2$, 
\eqref{eq:real-coefficient-iteration-|alpha|=1} yields
\begin{align*}
 Q_{e_i}^{(m+1)}-Q_{e_i}^{(m)}
 &=-p_i^+\sum_{|\alpha|\ge2}
 \Pi_i\bigl(Q_\alpha^{(m+1)}(0,\cdot)-Q_\alpha^{(m)}(0,\cdot)\bigr).
\end{align*}
Hence, by \eqref{eq:real-slice-sum}, we obtain
\begin{align}
 \label{eq:real-elementary-difference-bound}
\sum_{i\ge2}
 \|Q_{e_i}^{(m+1)}-Q_{e_i}^{(m)}\|_{m_*, \infty}\le
 C\sum_{|\alpha|\ge2}
 \|Q_\alpha^{(m+1)}-Q_\alpha^{(m)}\|_{m_*+1, \infty}.
\end{align}
Combining \eqref{eq:real-nonlinear-difference-bound} and \eqref{eq:real-elementary-difference-bound},
a simple substitution in \eqref{eq:real-successive-difference-sum} implies
\begin{equation}\label{eq:real-difference-from-forcing}
 D_m\le C
 \sum_{|\alpha|\ge2}
 \|R_\alpha (\mathcal Q^{(m)}) - R_\alpha(\mathcal Q^{(m-1)})\|_{m_*, \infty}.
\end{equation}
Recall the following elementary identity:
\begin{align*}
 \prod_{j=1}^\ell A_j-\prod_{j=1}^\ell B_j
 =
 \sum_{r=1}^\ell
 \Big(\prod_{j<r}A_j\Big)(A_r-B_r)
 \Big(\prod_{j>r}B_j\Big).
\end{align*}
With
$A_j=Q_{\beta_j}^{(m)}$ and $B_j=Q_{\beta_j}^{(m-1)}$,
by \eqref{eq-Banach-algebra-norm-for-degree-0-poly},
\eqref{eq:real-Galpha}, \eqref{eq-K_m-small},
and a simple rearrangement, we obtain
\begin{equation}
\label{eq:real-forcing-difference-bound}
\begin{aligned}
\sum_{|\alpha|\ge2}
\|R_\alpha (\mathcal Q^{(m)}) - R_\alpha(\mathcal Q^{(m-1)})\|_{m_*, \infty}
&\le \sum_{\ell\ge2}M^\ell
\sum_{r=1}^{\ell}K_m^{r-1}D_{m-1}K_{m-1}^{\ell-r}\\
&\le D_{m-1}
\sum_{\ell\ge2}\ell M^\ell\varepsilon^{\ell-1}.
\end{aligned}
\end{equation}
Moreover, by \eqref{eq-choice-of-epsilon-regularity-infty}, we have $M\epsilon < 1/2$,
and hence
\begin{equation}
\label{eq:real-geometric-Lipschitz-sum}
\begin{aligned}
\sum_{\ell\ge2}\ell M^\ell\varepsilon^{\ell-1}
=M^2\varepsilon
\sum_{\ell\ge2}\ell(M\varepsilon)^{\ell-2}
\le M^2\varepsilon
   \sum_{\ell\ge2}\ell\,2^{-(\ell-2)}
 =6M^2\varepsilon.
\end{aligned}
\end{equation}
Combining \eqref{eq:real-difference-from-forcing},
\eqref{eq:real-forcing-difference-bound},
and \eqref{eq:real-geometric-Lipschitz-sum},
our choice of $\epsilon$ in \eqref{eq-choice-of-epsilon-regularity-infty} 
yields, for any $m \ge 1$,
\[
 D_m\le6CM^2\varepsilon D_{m-1}
 \le\frac12D_{m-1}.
\]
This implies $\{Q_\alpha^{(m)}\}_{m\ge 0}$ is a Cauchy sequence
with respect to $\|\cdot\|_{m_*, \infty}$ and 
it converges to some function $P_\alpha$ on $\R\times\mathbb S^{n-1}$.
Moreover,
\begin{equation}\label{eq:real-limit-summability}
 \sum_{\alpha\in \ind}\|P_\alpha\|_{m_*, \infty}
 \le\varepsilon.
\end{equation}

Write $\mathcal P = \{P_\alpha\}_{\alpha \in \ind}$.
Note that \eqref{eq:real-forcing-difference-bound} and \eqref{eq:real-geometric-Lipschitz-sum}
also show that $\{R_\alpha(\mathcal Q^{(m)})\}_{m\ge 0}$ is a Cauchy sequence 
with respect to $\|\cdot\|_{m_*, \infty}$ and hence converges.
In fact,
\begin{align}
\label{eq-R_alpha(P^m)-convergence}
R_\alpha(\mathcal Q^{(m)}) \rightarrow R_\alpha(\mathcal P) \quad\text{in }\|\cdot\|_{m_*, \infty}
\text{ as } m\rightarrow \infty.
\end{align}
By letting $m\rightarrow \infty$, we obtain
$P_\alpha = 0$ if $\alpha_1 \ne 0$;
for any $|\alpha|\ge 2$ with $\alpha_1 = 0$,
$$P_\alpha = \mathcal L^{-1}_{\langle\alpha,\rho\rangle} \big[ R_\alpha(\mathcal P) \big],$$
by \eqref{eq-R_alpha(P^m)-convergence} and \eqref{eq-estimate-L_mu_alpha^-1-poly}; 
and, for $\alpha = e_i$ with $i \ge 2$,
$$P_\alpha = p_i^+\Big[
 g_i-\sum_{|\beta|\ge2}\Pi_i\big(P_\beta(0,\cdot)\big)\Big].$$
We point out that $P_{e_i}/p_i^+ \in \mathscr E_i$ for any $i \ge 2$,
and $P_\alpha$ satisfies \eqref{eq-recursive-equation-multi-index} 
for any $|\alpha|\ge 2$.
In particular, by a simple induction, $P_\alpha$ is a generalized polynomial of polynomial degree $0$.

{\it Step 4. We prove $\{X_i\}_{i\ge 2}$ satisfies an analyticity type estimate.} Define
\begin{equation}\label{eq:real-w-series}
 \widehat\varphi=\sum_{\alpha \in \ind} P_\alpha e^{-\langle \alpha,\rho\rangle t}.
\end{equation}
Fix $\gamma \in (0,1)$ and 
let $\psi_\alpha$ be as in \eqref{eq-psi_alpha=P_alpha exp}.
By the Sobolev embedding, 
a $C^1$-estimate of \eqref{eq-recursive-equation-f-alpha}
on $[0,\infty)\times\mathbb S^{n-1}$,
the choice of $m_*$ in \eqref{eq-choice-of-m_*},
and \eqref{eq:real-limit-summability}, we have
$$\sum_{\alpha \in \ind} |\psi_\alpha|_{C^1([0,\infty)\times\mathbb S^{n-1})} \le C \epsilon,$$
where $C$ is a constant depending only on $n$.
Then, by \eqref{eq-upper-bound-for-a_j} and a simple rearrangement, we obtain
$$\sum_{|\alpha| \ge 2} \Big| \sum_{j=2}^{|\alpha|} \sum_{\beta_1 +\cdots + \beta_j=\alpha}
a_j \psi_{\beta_1}\cdots \psi_{\beta_j} \Big|_{C^1([0,\infty)\times\mathbb S^{n-1})} \le (1 + M) \sum_{\ell\ge2} (CM\epsilon)^\ell.$$
By shrinking $\epsilon$ if necessary, 
we have the convergence of the last sum.
Then, by a $C^{2,\gamma}$-estimate on $[0,\infty)\times\mathbb S^{n-1}$,
we conclude that the series in \eqref{eq:real-w-series} converges in 
$C^{2,\gamma}([0,\infty)\times\mathbb S^{n-1})$,
and that $\widehat\varphi$ is a solution of \eqref{eq-linearized-Yamabe-at-xi}
on $[0,\infty)\times\mathbb S^{n-1}$.
We point out that the choice of $m_*$ in \eqref{eq-choice-of-m_*} implies
$|\cdot|_{C^{2,\gamma}(\mathbb S^{n-1})} \le C \|\cdot\|_{m_*}$.
Now, by letting $m\rightarrow\infty$ in \eqref{eq:real-exact-slice},
we obtain $\Pi_{\ge 2}\widehat\varphi(0,\cdot)=\Pi_{\ge 2}\widetilde\varphi(0,\cdot).$
Moreover, by the assumption $X_1 = 0$, 
the construction of $\{P_\alpha\}_{\alpha \in \ind}$,
and $\rho_2 >2$, we get
$\widehat\varphi, \widetilde\varphi = O(e^{-2 t}).$
We may further enlarge $n_0$ and shrink $\epsilon$, if necessary, so that \eqref{eq-smallness-uniqueness-Dirichlet}
holds for $\widetilde\varphi$ and $\widehat \varphi$ with
the constant $\epsilon_0$ corresponding to $\mu = 2$.
Then, by Proposition \ref{prop-uniqueness-of-Dirichlet},
we get $\widehat\varphi = \widetilde\varphi$ on 
$[0,\infty)\times\mathbb S^{n-1}$.
In particular, $\{Y_i = P_{e_i}/p_i^+\}_{i\ge 2}$ is 
the sequence of spherical harmonics with $Y_i \in \mathscr E_i$ associated
with $\widetilde\varphi$.
By a simple comparison of the expansions of $\varphi$ and $\widetilde \varphi$,
it is straightforward to check that $X_i = e^{n_0 T \rho_i} Y_i$ for any $i \ge 2$.
Then, by \eqref{eq:real-limit-summability}, we conclude
$$\|X_i\|_{L^2(\mathbb S^{n-1})} \le e^{n_0 T\rho_i} \|Y_i\|_{m_*} 
\le C \epsilon e^{n_0 T\rho_i}.$$
We obtain the desired result by noting $\rho_i/i \to 1$ as $i\to\infty$.  
\end{proof}



\subsection{The Collection of Solutions}\label{subsec-solution-space}
We now briefly study the collection of solutions $\varphi$ to \eqref{eq-linearized-Yamabe-at-xi} 
satisfying \eqref{eq-exp-convergence} with $\gamma\in (1,2]$. 

Let $H$ be a harmonic function in $B_{r_0}$ with $H(0)=0$ for some $r_0>0$. We can write, for any $r\theta\in B_{r_0},$ 
\begin{align}\label{eq-expansion-harmonic}H(r\theta) = \sum_{i=1}^\infty r^i X_i(\theta),\end{align}
for a sequence of spherical harmonics $\mathcal{X} = \{X_i\}_{i\geq 1}$ with $X_i\in \mathscr{E}_i$. 
The interior estimates for harmonic functions imply, for any $i\ge 1$,
\begin{align}\label{eq-estimate-basic-L2}\|X_i \|_{L^2(\mathbb{S}^{n-1})} \leq \widetilde C^i\|H\|_{L^2(B_{r_0})},\end{align}
for some positive constant $\widetilde C$ depending only on $n$ and $r_0$. 
In particular, \eqref{eq-exponential-growth-condition-expansion-from-optimal-term} 
holds for any $i\ge 1$ and some $\widetilde A>0$.
On the other hand, let $\mathcal{X} = \{X_i\}_{i\geq 1}$ be 
a sequence of spherical harmonics with $X_i\in \mathscr{E}_i$. 
If \eqref{eq-exponential-growth-condition-expansion-from-optimal-term} 
holds for any $i\ge 1$ and some $\widetilde A>0$, then $H$ given by \eqref{eq-expansion-harmonic} 
defines a harmonic function in $B_{r_0}$ with $H(0)=0$ for some $r_0>0$. 
In summary, any harmonic function $H$ near the origin with $H(0)=0$ can be identified 
with a sequence of spherical harmonics $\mathcal{X} = \{X_i\}_{i\geq 1}$ with $X_i\in \mathscr{E}_i$
satisfying \eqref{eq-exponential-growth-condition-expansion-from-optimal-term} 
for any $i\ge 1$ and some $\widetilde A>0$. 
It is obvious that $\nabla H(0)=0$ if and only if $X_1=0$. 

Take an arbitrary harmonic function $H$ near the origin with $H(0)=0$ and $\nabla H(0)=0$. 
Through the identification of $H$ and a sequence of spherical harmonics $\mathcal{X} = \{X_i\}_{i\geq 1}$ with $X_1=0$ 
as in \eqref{eq-expansion-harmonic}, by Theorem \ref{thrm-convergence-of-the-expansion-from-optimal-term}
and Theorem \ref{thrm-uniqueness-exponential-decay-sol}, 
there exists a unique smooth solution $\varphi(\mathcal{X})$ to \eqref{eq-linearized-Yamabe-at-xi} 
on some $[t_0,\infty)\times \mathbb S^{n-1}$ with 
its asymptotic behavior determined by $\mathcal{X} =\{X_i\}_{i \geq 1}$
as in Definition \ref{def-determined-behavior}.
This process induces a map $\Phi$ by 
\begin{align}\label{eq-definition-map-Phi}
\Phi(H) = \varphi(\mathcal{X}).
\end{align}
We call $\Phi(H)$ the solution 
of \eqref{eq-linearized-Yamabe-at-xi} whose asymptotic behavior is determined by $H$.
It is easy to see that $\varphi(\mathcal{X})$ satisfies \eqref{eq-exp-convergence} 
for $\gamma=2$. 
By combining with Theorem \ref{thm:actual-solution-series-improved}, we conclude that 
any smooth $\varphi$ to \eqref{eq-linearized-Yamabe-at-xi} on $[t_0,\infty)\times\mathbb{S}^{n-1}$
satisfying \eqref{eq-exp-convergence} 
for some $\gamma\in (1,2]$ and $C_*>0$ 
can be written as $\varphi=\Phi(H)$ for some harmonic function $H$ near the origin with $H(0)=0$ and $\nabla H(0)=0$. 
We can prove that $\Phi$ is a $C^1$-map between two appropriately chosen function spaces. We omit details.


\section{Conformal Transforms}\label{sec-Conformal-Transforms}

In Section \ref{sec-Convergence-of-the-Expansions},
we studied solutions of \eqref{eq-linearized-Yamabe-at-xi} 
under an enhanced exponential decay condition. 
In this section, we employ a class of conformal transforms to eliminate such a condition. 
More precisely, with appropriate conformal transforms, we reduce the general case 
to the case with such a condition. 
The conformal structure of the equation \eqref{eq-linearized-Yamabe-at-xi} plays an essential role.

\subsection{Conformal Transforms and Formal Expansions}\label{subsec-conformal-changes} 
Fix $b \in \R^n$ and $t_0 \ge 0$. Define the operator $\mathcal C_b$ acting on 
functions $v$ defined on $[t_0,\infty)\times \mathbb S^{n-1}$ by
\begin{align}
\label{eq-conformal-change-by-a-operator}
(\mathcal C_b v)(t,\theta) = 
|\theta - b e^{-t}|^{\frac{2-n}{2}} 
v \Big( t + \ln |\theta - b e^{-t}|, \frac{\theta - b e^{-t}}{|\theta - b e^{-t}|} \Big).
\end{align}
Note that $\mathcal C_b v$ is well-defined as long as 
$|b|e^{-t} <1$ and $t + \ln|\theta - b e^{-t}| \ge t_0$.
We call $\mathcal C_b$ the {\it conformal transform determined by $b$}.
It is obvious that $\mathcal C_0v=v$.

Such conformal transforms were applied to radial solutions of \eqref{eq-Yamabe-cylindrical}
in the derivation of bases of the projection of the linearized equation onto the first eigenspace, 
as in \cite{Korevaar1998RefinedAF}. 
In this section, we apply them to arbitrary solutions and also formal solutions.

If $v$ is a solution of \eqref{eq-Yamabe-cylindrical}, 
then $\mathcal C_b v$ is also a solution of \eqref{eq-Yamabe-cylindrical}.
It is straightforward to check that, for any $b, c \in \R^n$,
\begin{equation}\label{eq:Cbgroup}
 \mathcal{C}_b\mathcal{C}_cv=\mathcal{C}_{b+c}v.
\end{equation}
We now derive a simple expansion.

\begin{lemma}\label{lemma-formula-transformation-poly}
Let $b \in \R^n$, $\mu\ge 0$ be a constant, and $R$ be a 
generalized polynomial with $C^\infty(\mathbb{S}^{n-1})$-coefficients 
of polynomial degree $i$ and spherical degree $j$. 
Then, for any $\ell\ge 1$, 
\begin{align*}
\mathcal C_b(Re^{-\mu t})(t,\theta)=R(t,\theta)e^{-\mu t}
+\sum_{k=1}^\ell R_k(t, \theta)e^{-(\mu+k)t}+o(e^{-(\mu+\ell)t}),
\end{align*}
where $R_k$ is a generalized polynomial with $C^\infty(\mathbb{S}^{n-1})$-coefficients 
of a polynomial degree at most $i$ and a spherical degree at most $k+j$, 
for any $k\ge 1$.
\end{lemma}

\begin{proof}  Set $s=e^{-t}$ and write 
$$\mathcal C_b(Re^{-\mu t})(t,\theta)=|\theta - b s|^{\frac{2-n}{2}-\mu} 
R \Big( t + \ln |\theta - b s|, \frac{\theta - b s}{|\theta - b s|} \Big)e^{-\mu t}.$$
We will expand the right-hand side as a power series of $s$ 
and study its coefficients as functions of $t$ and $\theta$. 
For $s=0$, the right-hand side reduces to $Re^{-\mu t}$. 

By the linearity of $\mathcal C_b$, we simply consider $R(t, \theta)=t^ir(t)Y_j(\theta)$, 
for some $i\ge 0$, a periodic function $r$ with the same period as $\xi$, and $Y_j\in \mathscr E_j$. 
Then, 
$$\mathcal C_b(Re^{-\mu t})(t,\theta)=|\theta - b s|^{\frac{2-n}{2}-\mu} 
\,(t + \ln |\theta - b s|)^i\, r( t + \ln |\theta - b s|)\, Y_j\Big(\frac{\theta - b s}{|\theta - b s|} \Big)e^{-\mu t}.$$
First, we have 
\begin{align}\label{eq-identity-norm-difference} 
|\theta - b s|=\big(1-2\langle b, \theta\rangle s+|b|^2s^2\big)^{1/2}.\end{align}
It is obvious that $\langle b, \theta\rangle\in \mathscr E_1$ and $|b|^2\in \mathscr E_0$.
Hence, expanding any of the first three factors as power series of $s$, 
the coefficient of $s^k$ is a generalized polynomial of a spherical degree at most $k$. 
Next, we write  $Y_j=H_j|_{\mathbb S^{n-1}}$ for some homogeneous harmonic polynomial $H_j$ of degree $j$. Then, 
$$Y_j\Big(\frac{\theta - b s}{|\theta - b s|} \Big)=|\theta - b s|^{-j}H_j(\theta - b s).$$
The Taylor expansion in the direction $b$ yields
\begin{equation}\label{eq:Hk-shift-expansion}
 H_j(\theta-sb)
 =\sum_{m=0}^j\frac{1}{m!}(-1)^m(b\cdot\nabla)^mH_j(\theta)s^m.
\end{equation}
The coefficient of $s^m$ has a spherical degree $j-m\le j$. 
A straightforward multiplication yields the desired result for $R(t, \theta)=t^ir(t)Y_j(\theta)$.

For later reference, we observe that the coefficient of $s^m$ is a function of degree $m$ in $b$, 
as shown in \eqref{eq-identity-norm-difference} and \eqref{eq:Hk-shift-expansion}, for any $m\ge 0$.
\end{proof} 

In the following, we study asymptotic expansions of conformal changes of solutions. 
Throughout this section, we take $p_1^+=P_{\Lambda_1}^+$ as in \eqref{eq:floquet-p-pm-i=1}.

Let $\varphi$ be an exponential decay solution of \eqref{eq-linearized-Yamabe-at-xi} on $[t_0, \infty)\times \mathbb S^{n-1}$. 
By Theorem \ref{thrm-higher-order-expansion}, 
there exists a sequence $\mathcal{X} = \{X_i\}_{i\geq 1}$ of spherical harmonics 
with $X_i \in \mathscr{E}_i$
such that, for any $m \ge 1$
and any $(t,\theta) \in [t_0,\infty)\times\mathbb{S}^{n-1}$, 
\begin{align} \label{eq-higher-order-expansion-identity}
\varphi = \sum_{\ell=1}^m P_\ell(\mathcal{X}) e^{-\mu_\ell t} + R_m,
\end{align}
where $\{P_\ell(\mathcal{X})\}_{\ell \geq 1}$ is the sequence of generalized polynomials 
determined by $\mathcal{X}$ as in Definition \ref{def-sequence-P-determined-by-X}
and $R_m=O(t^m e^{-\mu_{m+1} t})$. 
It is obvious that $\xi+\varphi$ is a solution of \eqref{eq-Yamabe-cylindrical}. 
By the linearity of $\mathcal C_b$, we obtain  
\begin{align}\label{eq-formal-solution}
\mathcal C_b(\xi+\varphi)=
\mathcal C_b(\xi)+\sum_{\ell = 1}^m \mathcal C_b(P_\ell(\mathcal X) e^{-\mu_\ell t})+ \mathcal C_b(R_m).\end{align}
By Lemma \ref{lemma-formula-transformation-poly}, we have, for any $m\ge 1$, 
\begin{align}
\label{eq-conformal-expansion-by-b-zero-order}
\mathcal C_b(\xi) (t, \theta)= \xi(t) + \sum_{k = 1}^m P_{0,k}(t,\theta) e^{- k t}+o(e^{-mt}),
\end{align} 
for some sequence $\{P_{0,k}\}_{k \ge 1}$ of generalized polynomials 
with $C^\infty(\mathbb{S}^{n-1})$-coefficients of 
polynomial degree $0$. 
Note that $P_{0,1}$ is given by
\begin{align}
\label{eq-P_0,1}
P_{0,1}(t,\theta) = \langle b,\theta\rangle p_1^+(t).
\end{align}
Similarly, for each $\ell \ge 1$, we have 
\begin{equation}
\label{eq-conformal-expansion-by-b-ell-order}
\begin{aligned}
\mathcal C_b(P_\ell e^{-\mu_\ell t}) (t, \theta)
= P_\ell(\mathcal X)(t,\theta) e^{-\mu_\ell t} + \sum_{k=1}^m P_{\ell,k}(t,\theta) e^{-(\mu_\ell + k) t}+o(e^{-(\mu_\ell + m)t}),
\end{aligned}
\end{equation}
for some sequence $\{P_{\ell,k}\}_{k\ge 1}$ of generalized polynomials 
with $C^\infty(\mathbb{S}^{n-1})$-coefficients. 
Note that $\rho_1 = 1 < k$ for any $k \ge 2$ and $\rho_1 < \mu_j + k$ for any $j,k\ge 1$.
By combining \eqref{eq-formal-solution}, 
\eqref{eq-conformal-expansion-by-b-zero-order}, and 
\eqref{eq-conformal-expansion-by-b-ell-order}, 
we obtain, for any $m\ge 1$,   
\begin{align}\label{eq-conformal-formal-series} 
\mathcal C_b(\xi+\varphi)=
\xi+\sum_{\ell = 1}^m \widetilde P_\ell  e^{-\mu_\ell t}+o(e^{-\mu_mt}),\end{align}
where, for any $\ell\ge 1$,  
\begin{align}\label{eq-transformed-sequence-P_ell-1}
\widetilde P_\ell = P_\ell(\mathcal X) +  \sum_{\substack{k, k = \mu_\ell}} P_{0,k} 
+ \sum_{\substack{(j,k),\,\mu_j + k = \mu_\ell}} P_{j,k}.
\end{align} 
We point out that $ \widetilde P_\ell $ is a generalized polynomial 
with $C^\infty(\mathbb{S}^{n-1})$-coefficients and that the summations above consist of finitely many terms, 
determined only by $\xi$, $b$, and $P_1(\mathcal X),\cdots, P_{\ell-1}(\mathcal X)$.
In fact, each $\widetilde P_\ell$ can be expressed explicitly in terms of $\mathcal{X} = \{X_i\}_{i\geq 1}$. 
For example, by \eqref{eq-P_0,1} and \eqref{eq-transformed-sequence-P_ell-1} with $\ell = 1$, 
we have
\begin{align*}
\widetilde P_1   = P_1(\mathcal X) + \langle b,\theta\rangle p_1^+ 
 = (X_1 + \langle b,\theta\rangle) p_1^+.
\end{align*}

We note that $\mathcal C_b(\xi+\varphi)$ is a solution of \eqref{eq-Yamabe-cylindrical} 
with the same leading term $\xi$ by \eqref{eq-conformal-formal-series}, and hence $\mathcal C_b(\xi+\varphi)-\xi$ is
an exponential decay solution of \eqref{eq-linearized-Yamabe-at-xi}. 
By Theorem \ref{thrm-higher-order-expansion} applied to $\mathcal C_b(\xi+\varphi)-\xi$, 
there exists a sequence $\widetilde{\mathcal{X}} = \{\widetilde X_i\}_{i\geq 1}$ of spherical harmonics 
with $\widetilde X_i \in \mathscr{E}_i$ such that $\{\widetilde P_\ell\}_{\ell \ge 1}$ is
determined by $\widetilde{\mathcal{X}}$ as in Definition \ref{def-sequence-P-determined-by-X}, i.e., 
$\widetilde P_\ell=P_\ell(\widetilde{\mathcal{X}})$ for any $\ell\ge 1$. 
We now characterize $\widetilde{\mathcal{X}}$. 

\begin{lemma}
\label{lemma-formal-seq-spherical-harmonics-actual} 
Let $b \in \R^n$, $\varphi$ be an exponential decay solution of \eqref{eq-linearized-Yamabe-at-xi} 
 on $[t_0, \infty)\times \mathbb S^{n-1}$, 
$\mathcal X$ and $\{P_\ell(\mathcal X)\}_{\ell\ge 1}$ 
be as in \eqref{eq-higher-order-expansion-identity}, 
and $\{\widetilde P_\ell\}_{\ell\ge 1}$ be given by \eqref{eq-transformed-sequence-P_ell-1}. 
Define
\begin{align}
\label{eq-C_bX_1}
(\mathcal C_b \mathcal X)_1 =  X_1 + \langle b, \theta \rangle,
\end{align}
and, for $i\ge 2$,  
\begin{align}
\label{eq-C_bX_i-tilde}
(\mathcal C_b \mathcal X)_i = \frac 1 {p_i^+} 
\big[ \widetilde  P_\ell - \mathcal L_{\rho_i}^{-1} \mathcal L_{\rho_i}\widetilde P_\ell \big],
\end{align}
where $\ell$ satisfies $\mu_\ell = \rho_i$, and $\mathcal L_{\rho_i}$ is given by \eqref{eq-linear-equation-mu-polynomial-nonconstant}
and $\mathcal L_{\rho_i}^{-1}$ is given by \eqref{eq-L_mu^-1} with $\mu = \rho_i$.
Then, $(\mathcal C_b\mathcal X)_i \in \mathscr E_i$ for any $i\ge 1$ and
$\{\widetilde P_\ell\}_{\ell \ge 1}$ is determined by $\mathcal C_b\mathcal X=\{(\mathcal C_b\mathcal X)_i\}_{i\ge 1}$
as in Definition \ref{def-sequence-P-determined-by-X}. 
\end{lemma}

\begin{proof} 
First, 
it is obvious that 
$(\mathcal C_b\mathcal X)_1 \in \mathscr E_1$ by \eqref{eq-C_bX_1}. 
Next, we take any $i \ge 2$ and proceed to prove $(\mathcal C_b\mathcal X)_i \in \mathscr E_i$. 
Note that there is a unique $\ell$ satisfying $\mu_\ell = \rho_i$. 
By the explicit expression in \eqref{eq-C_bX_i-tilde}, we get 
$\mathcal L_{\rho_i}({p_i^+}(\mathcal C_b\mathcal X)_i)=0$, 
and thus $\mathcal L({p_i^+}(\mathcal C_b\mathcal X)_ie^{-\rho_i t})=0$.  
By the proof of Lemma \ref{lemma-polynomial_solution}, 
we have $(\mathcal C_b\mathcal X)_i \in \mathscr E_i$. 
Last, since $\{\widetilde P_\ell\}$ is the sequence of 
generalized polynomials associated with a solution of 
\eqref{eq-linearized-Yamabe-at-xi},
\eqref{eq-recursive-equation-Asymp-expans} holds for $\widetilde P_\ell$.
By a simple induction, 
we conclude that $\{\widetilde P_\ell\}_{\ell \ge 1}$ is the sequence of 
generalized polynomials determined by $\mathcal C_b\mathcal X$
as in Definition \ref{def-sequence-P-determined-by-X}.
Therefore, $\widetilde P_\ell=P_\ell(\mathcal C_b\mathcal X)$. 
We point out that 
\eqref{eq-C_bX_i-tilde} is inspired 
by Definition \ref{def-sequence-P-determined-by-X}(ii)-(iii). 
The positivity of $p_i^+$ as in Lemma \ref{lemma-kernel-of-projection-ode} is needed here. 
\end{proof} 

We next prove that Lemma \ref{lemma-formal-seq-spherical-harmonics-actual} 
still holds if $\varphi$ is a formal solution of \eqref{eq-linearized-Yamabe-at-xi}.

\begin{lemma}
\label{lemma-formal-seq-spherical-harmonics}
Let $b \in \R^n$, $\mathcal X = \{X_i\}_{i\ge 1}$ be a sequence of spherical harmonics with $X_i \in \mathscr E_i$, 
$\{P_\ell (\mathcal X)\}_{\ell \ge 1}$ be the sequence of 
generalized polynomials determined by $\mathcal X$
as in Definition \ref{def-sequence-P-determined-by-X}, 
$\{\widetilde P_\ell\}_{\ell\ge 1}$ be given by \eqref{eq-transformed-sequence-P_ell-1}, 
and $\mathcal C_b\mathcal X = \{(\mathcal C_b \mathcal X)_i\}_{i\ge 1}$ be given by \eqref{eq-C_bX_1}
and \eqref{eq-C_bX_i-tilde}. 
Then, $(\mathcal C_b\mathcal X)_i \in \mathscr E_i$ for any $i\ge 1$ and
$\{\widetilde P_\ell\}_{\ell \ge 1}$ is determined by $\mathcal C_b\mathcal X=\{(\mathcal C_b\mathcal X)_i\}_{i\ge 1}$
as in Definition \ref{def-sequence-P-determined-by-X}.
\end{lemma}

\begin{proof} Fix $\ell \ge 2$ and note that 
the recursive equations \eqref{eq-recursive-equation-Asymp-expans} for $P_2(\mathcal X),\cdots, P_\ell (\mathcal X)$
are independent of $\{P_j(\mathcal X)\}_{j \ge \ell+1}$.
Since $\{P_\ell (\mathcal X)\}_{\ell \ge 1}$ is the sequence of 
generalized polynomials determined by $\mathcal X$
as in Definition \ref{def-sequence-P-determined-by-X}, 
it is straightforward to verify that 
the finite series $\widehat v=\xi + \sum_{j=1}^\ell P_j(\mathcal X) e^{-\mu_j t}$ is 
an approximate solution of \eqref{eq-Yamabe-cylindrical} as in \cite{HAN2021127}. 
By Theorem 3.1\cite{HAN2021127}, 
there exists a solution $v$ of \eqref{eq-Yamabe-cylindrical} 
such that $\varphi = v - \xi$ satisfies \eqref{eq-higher-order-expansion-inequality}, i.e., for any $\mu \in (\mu_\ell,\mu_{\ell+1})$,
\begin{align*}
v = \xi + \sum_{j=1}^\ell P_j(\mathcal X) e^{-\mu_j t} + O\big(e^{-\mu t}\big).
\end{align*}
By the definition of $\mathcal C_b$ in \eqref{eq-conformal-change-by-a-operator} and
the expansions in \eqref{eq-conformal-expansion-by-b-zero-order} and \eqref{eq-conformal-expansion-by-b-ell-order},
it is straightforward to check that, for any $\mu \in (\mu_\ell,\mu_{\ell+1})$,
\begin{align}
\label{eq-formal-expansion-C_bv}
\mathcal C_b v = \xi + \sum_{j=1}^\ell \widetilde P_j e^{-\mu_j t} 
+ O\big(e^{-\mu t}\big).
\end{align}
We emphasize that if we expand $v$ and $\mathcal C_b v$ further, then the coefficients of $e^{-\mu_j t}$ 
may not be $P_j(\mathcal X)$ and $\widetilde P_j$ for $j\ge \ell+1$, respectively. 

Note that $\mathcal C_b v$ is also a solution to \eqref{eq-Yamabe-cylindrical} 
with the same leading term $\xi$.
Then, \eqref{eq-higher-order-expansion-inequality} holds for $\mathcal C_b v - \xi$ 
and some sequence $\{\widehat P_\ell\}_{\ell \ge 1}$ of generalized polynomials 
with $C^\infty(\mathbb{S}^{n-1})$-coefficients. 
By a simple comparison based on \eqref{eq-formal-expansion-C_bv} and 
\eqref{eq-higher-order-expansion-inequality} for $\mathcal C_b v - \xi$, 
we conclude that $\widehat P_j = \widetilde P_j$ for $1\le j\le \ell$.
In other words, $\widetilde P_j$ for $1\le j\le \ell$ 
are the coefficients of the first $\ell$ terms in the expansion of $\mathcal C_b v - \xi$.
The proof of Lemma \ref{lemma-formal-seq-spherical-harmonics-actual} can be repeated.
\end{proof}

From now on, we write, 
for any $\ell \ge 1$,
\begin{align}
\label{eq-transformed-sequence-P_ell}
P_\ell(\mathcal C_b\mathcal X) = P_\ell(\mathcal X) +  \sum_{\substack{k, k = \mu_\ell}} P_{0,k} 
 + \sum_{\substack{(j,k),\,\mu_j + k = \mu_\ell}} P_{j,k},
\end{align}
where $\{P_{0,k}\}_{k\ge 1}$ and $\{P_{j,k}\}_{j,k\ge 1}$ are given by \eqref{eq-conformal-expansion-by-b-zero-order} 
and \eqref{eq-conformal-expansion-by-b-ell-order}, respectively.
For any $i\ge 2$, with $\mu_\ell=\rho_i$ 
and $P_\ell(\mathcal C_b\mathcal X)$ given by \eqref{eq-transformed-sequence-P_ell}, we can write
\begin{align}
\label{eq-C_bX_i}
(\mathcal C_b \mathcal X)_i = \frac 1 {p_i^+} 
\big[  P_\ell(\mathcal C_b\mathcal X) - \mathcal L_{\rho_i}^{-1} \mathcal L_{\rho_i}[ P_\ell(\mathcal C_b\mathcal X)] \big].
\end{align}
For each $\ell \ge 1$, 
we can decompose $P_\ell(\mathcal X)$ as in \eqref{eq-P-ell-expression}.
For each $\alpha \in \ind$, we write, for any $m \ge 1$,
\begin{equation}
\label{eq-conformal-expansion-P_alpha}
\begin{aligned}
\mathcal C_b(P_\alpha e^{-\langle\alpha,\rho\rangle t})
= P_\alpha e^{-\langle\alpha,\rho\rangle t} + \sum_{k=1}^m P_{\alpha,k} e^{-(\langle\alpha,\rho\rangle + k) t} + o(e^{-(\langle\alpha,\rho\rangle + m) t}),
\end{aligned}
\end{equation}
for some sequence $\{P_{\alpha,k}\}_{k\ge 1}$ of generalized polynomials 
with $C^\infty(\mathbb{S}^{n-1})$-coefficients. 
By Lemma \ref{lemma-formula-transformation-poly}, $P_{\alpha,k}$ has a spherical degree at most $\langle \alpha,\N\rangle + k$.
By \eqref{eq-transformed-sequence-P_ell} 
and the linearity of $\mathcal C_b$, we conclude that
\begin{align}
\label{eq-decomposition-C_bP_ell-multiindex}
P_\ell(\mathcal C_b\mathcal X) = \sum_{\alpha, \langle \alpha, \rho \rangle = \mu_\ell } P_\alpha(\mathcal X) 
 +  \sum_{\substack{k, k = \mu_\ell}} P_{0,k} + \sum_{\substack{(\alpha,k),
\langle\alpha,\rho\rangle + k = \mu_\ell}} P_{\alpha,k}.
\end{align}

\subsection{Properties}\label{subsec-Properties} 

In this subsection, we study properties of $\mathcal C_b \mathcal X$. 
It is obvious that $\mathcal C_0 \mathcal X = \mathcal X$.
First, we prove an important transformation law of $\mathcal C_b \mathcal X$.

\begin{lemma}
\label{lemma-group-property-of-C_bP}
Let $b,c \in \R^n$ and $\mathcal X = \{X_i\}_{i\ge 1}$ be a sequence of spherical harmonics with $X_i \in \mathscr E_i$.
Then, 
$\mathcal C_b \mathcal C_c \mathcal X = \mathcal C_{b+c} \mathcal X$.
\end{lemma}

\begin{proof}
As in the proof of Lemmas \ref{lemma-formal-seq-spherical-harmonics-actual} and \ref{lemma-formal-seq-spherical-harmonics},
we assume that $\{P_\ell(\mathcal X)\}_{\ell \ge 1}$ is the sequence of 
generalized polynomials associated with a solution $v$ to \eqref{eq-Yamabe-cylindrical}
with the leading term $\xi$.
Then, $\{P_\ell(\mathcal C_b\mathcal C_c\mathcal X)\}_{\ell \ge 1}$ 
and $\{P_\ell(\mathcal C_{b+c}\mathcal X)\}_{\ell \ge 1}$ are the sequences of 
generalized polynomials associated with $\mathcal C_b\mathcal C_c v$ and $\mathcal C_{b+c} v$, respectively.
By \eqref{eq:Cbgroup}, we have 
$P_\ell(\mathcal C_b\mathcal C_c\mathcal X) =P_\ell (\mathcal C_{b+c}\mathcal X)$ 
and then $(\mathcal C_b\mathcal C_c\mathcal X)_i =(\mathcal C_{b+c}\mathcal X)_i$ for any $i\ge 1$.
\end{proof}

In the next result, we prove that there are no positive powers of $t$ 
in the asymptotic expansions of solutions to \eqref{eq-linearized-Yamabe-at-xi}.

\begin{proposition}
\label{prop-no-t-power}
Let $\mathcal X = \{X_i\}_{i\ge 1}$ be a sequence of spherical harmonics with $X_i \in \mathscr E_i$,
$\{P_\ell(\mathcal X)\}_{\ell \ge 1}$ be the sequence of 
generalized polynomials determined by $\mathcal X$
as in Definition \ref{def-sequence-P-determined-by-X},
and $\{P_\alpha(\mathcal X)\}_{\alpha \in \ind}$ be the sequence of
generalized polynomials determined by $\mathcal X$
as in Proposition \ref{prop-formal-expansion-multi-index}.
Then, $P_\ell (\mathcal X)$ and $P_\alpha(\mathcal X)$ have
polynomial degree $0$ for any $\ell \ge 1$ and $\alpha \in \ind$.
\end{proposition}

\begin{proof}
We first consider $P_\ell(\mathcal X)$.
If $X_1 = 0$, then $P_\ell(\mathcal X)$ has polynomial degree $0$ for any $\ell \ge 1$,
by \eqref{eq-P-ell-expression} and
Corollary \ref{cor-expansion-from-optimal-term-no-poly}.
Now, assume that $X_1(\theta) = \langle b,\theta \rangle$ for some $b \in \R^n\setminus \{0\}$.
Consider the sequence of spherical harmonics 
$\mathcal C_{-b}\mathcal X = \{(\mathcal C_{-b}\mathcal X)_i\}_{i\ge 1}$
as in Lemma \ref{lemma-formal-seq-spherical-harmonics}. 
Then, $(\mathcal C_{-b} \mathcal X)_1 = 0$ by \eqref{eq-C_bX_1}.
Hence, $ P_\ell(\mathcal C_{-b}\mathcal X)$ has polynomial degree $0$ for any $\ell \ge 1$.
By Lemma \ref{lemma-group-property-of-C_bP}, 
$P_{\ell} (\mathcal X) =  P_\ell(\mathcal C_b \mathcal C_{-b}\mathcal X)$ has polynomial degree $0$.
We point out that $P_{0,k}$ in \eqref{eq-conformal-expansion-by-b-zero-order} 
has polynomial degree $0$ for any  $k \ge 1$
and that for any $\ell \ge 1$, if $P_\ell(\mathcal X)$ has polynomial degree $0$, 
then $P_{\ell,k}$ in \eqref{eq-conformal-expansion-by-b-ell-order} also has polynomial degree $0$ for any $k \ge 1$.

We next consider $P_\alpha(\mathcal X)$.
For a sequence $z=\{z_i\}_{i\ge1}$ of real numbers, set
$z\mathcal X=\{z_iX_i\}_{i\ge1}$.
By \eqref{eq-recursive-equation-multi-index} and 
an induction,
it is straightforward to check that, for any $\alpha \in \ind$,
\begin{align}
\label{eq-P_alpha-scaling-property}
P_\alpha(z\mathcal X)=z^\alpha P_\alpha(\mathcal X),
\end{align}
where $z^\alpha = \prod_{i=1}^\infty {z_i}^{\alpha_i}$.
We point out that the product is finite since $\alpha_i =0$ for all but
finitely many $i$.
Fix $\ell\ge2$. 
Note that there are finitely many $\alpha \in \ind$ satisfying
$\langle\alpha,\rho\rangle = \mu_\ell$.
For any such $\alpha \in \ind$, we can choose an integer $N$ and write
\[
P_\alpha(\mathcal X)
=
\sum_{q=0}^N t^q P_{\alpha,q}, 
\]
where $P_{\alpha,q}$ is a generalized polynomial of polynomial degree $0$.
By \eqref{eq-P-ell-expression} and \eqref{eq-P_alpha-scaling-property},
we have
\begin{align}
\label{eq-decompos-P_ell-scaling}
P_\ell(z\mathcal X)
= \sum_{\langle\alpha,\rho\rangle=\mu_\ell}
z^\alpha P_\alpha(\mathcal X) =
\sum_{q=0}^N \Big(
\sum_{\langle\alpha,\rho\rangle=\mu_\ell}
z^\alpha P_{\alpha,q}\Big) t^q.
\end{align}
By the first part of the proof, $P_\ell(z\mathcal X)$ has polynomial
degree $0$ for every $z$. 
Thus, the coefficient of $t^q$ in \eqref{eq-decompos-P_ell-scaling} vanishes
if $q \ge 1$.
Since the monomials $z^\alpha$ are linearly independent,
we obtain $P_{\alpha,q}=0$
for every $q\ge1$ and every $\alpha \in \ind$ satisfying
$\langle\alpha,\rho\rangle=\mu_\ell$.
Since $\ell\ge2$ is arbitrary and $P_1=P_{e_1}$ has polynomial degree
$0$, we conclude that $P_\alpha(\mathcal X)$ has polynomial degree $0$
for every $\alpha\in\ind$.
\end{proof}

Next, we study structures of transformed sequences of spherical harmonics 
and prove that the spherical harmonics of degree $i$ remain unchanged 
for sufficiently large $i$. 
Lemma \ref{lemma-<alpha,rho>=rho_i} will be essentially used. 

\begin{proposition}
\label{prop-preservation-of-tail-conformal-change}
Let $b \in \R^n$, $\mathcal X = \{X_i\}_{i\ge 1}$ be a sequence of spherical harmonics with $X_i \in \mathscr E_i$,
and $\mathcal C_b \mathcal X = \{(\mathcal C_b \mathcal X)_i\}_{i\ge 1}$ be the sequence of spherical harmonics as in 
Lemma \ref{lemma-formal-seq-spherical-harmonics}.
Then, there exists an integer $I_*\ge 2$ depending only on $n$ and $\xi$ such that 
$X_i = (\mathcal C_b \mathcal X)_i$ for any $i \ge I_*$. 
\end{proposition}
 
\begin{proof}
First, as in \eqref{eq-conformal-expansion-by-b-zero-order},
we write, for any $m \ge 2$, 
$$ \mathcal C_b(\xi)(t, \theta) 
= \xi(t) + \langle b,\theta\rangle p_1^+(t) e^{-t} + \sum_{k = 2}^m P_{0,k}(t,\theta) e^{- k t} + o(e^{-mt}), $$
for some sequence $\{P_{0,k}\}_{k\ge 2}$ of generalized polynomials with polynomial degree $0$.
By Lemma \ref{lemma-formula-transformation-poly},
$P_{0,k}$ has a spherical degree at most $k$.

Now, 
let $\{P_\ell(\mathcal X)\}_{\ell\ge 1}$ and $\{P_\alpha(\mathcal X)\}_{\alpha \in \ind}$ be the sequences of generalized
polynomials with $C^\infty(\mathbb S^{n-1})$-coefficients 
determined by $\mathcal X$ as in Definition \ref{def-sequence-P-determined-by-X} 
and 
Proposition \ref{prop-formal-expansion-multi-index}, respectively.
Fix $i \ge 2$ and choose the unique integer $\ell\ge 2$ satisfying $\mu_\ell = \rho_i$.
By \eqref{eq-C_bX_i}, we need to compute 
$(\operatorname{Id} - \mathcal L_{\rho_i}^{-1} \mathcal L_{\rho_i})[ P_\ell(\mathcal C_b\mathcal X)].$
Since $\operatorname{Id} - \mathcal L_{\rho_i}^{-1} \mathcal L_{\rho_i}$ is a linear operator, 
we can analyze each term in the right-hand side of \eqref{eq-decomposition-C_bP_ell-multiindex} separately.

For the first term, we write
$$\sum_{\alpha, \langle \alpha, \rho \rangle = \rho_i } P_\alpha(\mathcal X) 
= X_ip_i^+ + \sum_{|\alpha|\ge 2, \langle \alpha, \rho \rangle = \rho_i} P_\alpha(\mathcal X).$$
Take any $\alpha \in \ind$ with $|\alpha|\ge 2$ and $\langle\alpha,\rho\rangle = \rho_i$. 
By \eqref{eq-recursive-equation-multi-index}, we get
$P_\alpha(\mathcal X) = \mathcal L_{\rho_i}^{-1}[R_\alpha],$
for some generalized polynomial $R_\alpha$.
Hence, $\mathcal L_{\rho_i}^{-1} \mathcal L_{\rho_i}[P_\alpha(\mathcal X)] = P_\alpha(\mathcal X)$ for such $\alpha$.
On the other hand, we have $\mathcal L_{\rho_i}^{-1} \mathcal L_{\rho_i}[X_i p_i^+] = \mathcal L_{\rho_i}^{-1}[0] = 0.$
Therefore, 
\begin{align}
\label{eq-computation-of-projection-1st-term-conformal-change}
(\operatorname{Id} - \mathcal L_{\rho_i}^{-1} \mathcal L_{\rho_i})
\Big[ \sum_{\alpha, \langle \alpha, \rho \rangle = \rho_i } P_\alpha(\mathcal X)\Big] = X_i p_i^+.
\end{align}

For the last term in \eqref{eq-decomposition-C_bP_ell-multiindex},
we claim that, for any pair $(\alpha,k) \in \ind\times\N$ 
satisfying $\langle\alpha,\rho\rangle + k = \rho_i$ and $i >\langle\alpha,\N\rangle + k$,
\begin{align}
\label{eq-computation-of-projection-3rd-term-vanishing-conformal-change}
(\operatorname{Id} - \mathcal L_{\rho_i}^{-1} \mathcal L_{\rho_i})[P_{\alpha,k}] = 0.
\end{align}
Fix such a pair $(\alpha,k)$. 
By Lemma \ref{lemma-formula-transformation-poly}, $P_{\alpha,k}$ has a 
spherical degree at most $\langle\alpha,\N\rangle + k$.
Since $i > \langle\alpha,\N\rangle + k$, 
we have $\rho_i >  \rho_r$ for any $r \le \langle\alpha,\N\rangle + k$.
We obtain \eqref{eq-computation-of-projection-3rd-term-vanishing-conformal-change} 
by \eqref{eq-identity-inverse-L-mu}.

Combining \eqref{eq-C_bX_i}, \eqref{eq-computation-of-projection-1st-term-conformal-change},
and \eqref{eq-computation-of-projection-3rd-term-vanishing-conformal-change},
we obtain
\begin{equation}
\label{eq-computation-of-proj-operator-last-conformal-change}
(\mathcal C_b\mathcal X)_i
 = X_i  +  \frac{1}{p_i^+} 
 \!\!\sum_{\substack{k, k = \rho_i}} 
\!\! (\operatorname{Id} - \mathcal L_{\rho_i}^{-1} \mathcal L_{\rho_i}) P_{0,k} 
 +  \frac{1}{p_i^+}
\!\!\sum_{\substack{(\alpha,k),
\langle\alpha,\rho\rangle + k = \rho_i,\\i \le\langle\alpha,\N\rangle + k}} 
\!\!\!(\operatorname{Id} - \mathcal L_{\rho_i}^{-1} \mathcal L_{\rho_i}) P_{\alpha,k}.
\end{equation}
Therefore, $X_i \neq (\mathcal C_b \mathcal X)_i$ possibly holds only when $\rho_i = k$ for some $k \ge 2$,
or $\rho_i = \langle\alpha,\rho\rangle + k$ for some $(\alpha,k) \in \ind\times\N$ with $ i \le\langle\alpha,\N\rangle + k$
and $k\ge 1$.
We point out that $i \le k$ in the first case, by Lemma \ref{lem:delta-Delta}(i).
In both cases, \eqref{eq-<alpha,rho>=rho_i} holds for some $\widehat \alpha \in \ind$
with $\widehat \alpha_1 \ge 1$ and $|\widehat\alpha|\ge 2$, and $1 \le i \le \langle \widehat\alpha,\N\rangle$; namely, 
$\langle \widehat\alpha,\rho\rangle = \rho_i.$
Indeed, $\widehat \alpha = k e_1$ in the first case and $\widehat \alpha = \alpha + k e_1$ in the second case.
By Lemma \ref{lemma-<alpha,rho>=rho_i}, there exists an integer $I_*\ge 1$ depending only on $n$ and $\xi$ such that
\eqref{eq-<alpha,rho>=rho_i} has no solutions for any $i \ge I_*$.
Therefore, $X_i = (\mathcal C_b \mathcal X)_i$ for any $i \ge I_*$.
\end{proof}

Next, we derive analyticity type estimates for $P_{0,k}$ in 
\eqref{eq-conformal-expansion-by-b-zero-order}
and $P_{\alpha,k}$ in
\eqref{eq-conformal-expansion-P_alpha}.

\begin{proposition}
\label{prop-convergence-of-conformal-expansion}
Let $b \in \R^n\setminus\{0\}$ and $\mathcal X = \{X_i\}_{i\ge 1}$ be a sequence of spherical harmonics
with $X_i \in \mathscr E_i$ and $X_1 = 0$
such that \eqref{eq-exponential-growth-condition-expansion-from-optimal-term} holds for any $i \ge 2$
and some $\widetilde A>0$.
Assume
$\{P_{0,k}\}_{k\ge 1}$ is the sequence of generalized polynomials in
\eqref{eq-conformal-expansion-by-b-zero-order}, 
and $\{P_{\alpha,k}\}_{\alpha \in \ind, k\ge 1}$ is the sequence of
generalized polynomials in \eqref{eq-conformal-expansion-P_alpha}. 
Then, for any $m \ge 0$, $k \ge 1$, and $\ell \ge 2$,
\begin{align}
\label{eq-exp-upper-bound-conformal-expans-all-derivatives-1}
|\nabla_{(t,\theta)}^m P_{0,k}|_{L^\infty(\R \times \mathbb S^{n-1})} \le C_mK^{k},
\end{align}
and
\begin{align}
\label{eq-exp-upper-bound-conformal-expans-all-derivatives-2}
\sum_{\substack{\langle\alpha,\N\rangle = \ell\\\alpha_1 = 0}} 
|\nabla_{(t,\theta)}^m P_{\alpha,k}|_{L^\infty(\R\times\mathbb S^{n-1})} \le C_m K^{ \ell + k},
\end{align}
where $C_m$ is a positive constant depending on $n$, $\xi$, $b$, 
$\widetilde A$, and $m$,
and $K$ is a positive constant depending only on $n$, $\xi$, $b$, and 
$\widetilde A$.
\end{proposition}

\begin{proof}
For simplicity, denote by $C$ (different from line to line) a general 
constant depending only on $n$, $\xi$, $b$, and $\widetilde A$.
Set $s = e^{-t}$.
By \eqref{eq-conformal-change-by-a-operator}
and \eqref{eq-conformal-expansion-by-b-zero-order},
$P_{0,k}$ is obtained by the Taylor expansion in $s$ of
$|\theta - sb|^{-a} \xi(t + \ln |\theta - sb|),$
with $a = (n-2)/2$.
Write,
for each $r \ge 1$, 
\[
(\ln|\theta-sb|)^r
=
\sum_{j=r}^\infty L_{r,j}(\theta)s^j.
\]
Then, $L_{r,j}$ has a spherical degree at most $j$ and,
for any $j \ge r \ge 1$,
\begin{align}
\label{eq-upper-bound-L-r-j}
|L_{r,j}|_{L^\infty(\mathbb S^{n-1})}
\le
K_0^j,
\end{align}
for some positive constant $K_0$ depending only on $b$.
For simplicity, we set $L_{0,0} = 1$, $L_{0,j} = 0$ for any $j \ge 1$,
and $L_{r,j} = 0$ for any $r \ge 1$ and $j < r$.
It is straightforward to check that
\begin{align}
\label{eq-P_0,k-explicit}
P_{0,k}(t,\theta)
=
\sum_{r=1}^k
\frac{1}{r!}
(\partial_t-a)^r\xi(t)
L_{r,k}(\theta).
\end{align}
Now, the analyticity of $\xi$ implies, for any $k \ge 0$,
\begin{align}
\label{eq-analytic-estimate-xi-conformal-conv}
|\partial_t^k \xi|_{L^\infty (\R)} \le C^{k+1} k!.
\end{align}
Combining with \eqref{eq-upper-bound-L-r-j},
we have
$$|P_{0,k}|_{L^\infty(\R \times\mathbb S^{n-1})} \le C e^a k (C K_0)^k.$$
For the $t$-derivatives, the same argument applies, 
since the $t$-derivatives act only on $\xi$ in 
\eqref{eq-P_0,k-explicit}. Thus, by shifting the number of 
$t$-derivatives in \eqref{eq-analytic-estimate-xi-conformal-conv},
we obtain, for any $p\ge0$,
\begin{align*}
|\partial_t^pP_{0,k}
|_{L^\infty(\R\times\mathbb S^{n-1})}
\le
e^aC^{p+1}k (k+p)^p   (C K_0)^k.
\end{align*}
On the other hand, since $P_{0,k}$ has a spherical degree at most $k$, 
we can proceed as in the last paragraph of Step 3 in the proof of 
Theorem \ref{thrm-convergence-of-the-expansion-from-optimal-term}
to obtain, for any $p,q \ge 0$,
\begin{align*}
|\nabla_\theta^q\partial_t^pP_{0,k}
|_{L^\infty(\R\times\mathbb S^{n-1})}
\le
e^a C_q \lambda_k^q C^{p+1} k (k+p)^p  (C K_0)^k,
\end{align*}
where $C_q$ is a constant depending only on $n$ and $q$.
This finishes the proof of 
\eqref{eq-exp-upper-bound-conformal-expans-all-derivatives-1}.

We next estimate $P_{\alpha,k}$.
By Proposition \ref{prop-no-t-power}, $P_\alpha$ has polynomial degree $0$.
By \eqref{eq-conformal-change-by-a-operator}
and \eqref{eq-conformal-expansion-P_alpha}, 
we need to study the Taylor expansion in $s$ at $0$ of
\begin{align}
\label{eq-conformal-piece-P_alpha}
|\theta - sb|^{-a - \langle\alpha,\rho\rangle} P_\alpha\Big(
t+\ln|\theta - sb|,
\frac{\theta-sb}{|\theta-sb|}
\Big).
\end{align}
Since $P_\alpha$ has a spherical degree at most
$\langle\alpha,\N\rangle$, we can write
\begin{align*}
P_\alpha(t,\theta)
=
\sum_{i=0}^{\langle\alpha,\N\rangle}
P_{\alpha,i}(t,\theta),
\end{align*}
for some $P_{\alpha,i} (t,\cdot) \in \mathscr E_i$ for any $t \in \R$.
For each $i$, let $H_{\alpha,i}(t,\cdot)$ be the unique homogeneous harmonic
polynomial of degree $i$ such that
$H_{\alpha,i}(t,\cdot)|_{\mathbb S^{n-1}}=P_{\alpha,i}(t,\cdot)$.
Then,
\begin{align}
\label{eq-rewrite-P-alpha-harmonic-extension}P_\alpha\Big(
t+\ln|\theta - sb|,
\frac{\theta-sb}{|\theta-sb|}
\Big)=
\sum_{i=0}^{\langle\alpha,\N\rangle}
|\theta-sb|^{-i}
H_{\alpha,i}
\big(
t+\ln|\theta - sb|,
\theta-sb
\big).
\end{align}

Fix $0 \le i \le \langle\alpha,\N\rangle$.
We first derive an upper bound for the derivatives of
$H_{\alpha,i}$.
By the orthogonality of the spherical harmonics of different degrees
with respect to $\langle\cdot,\cdot\rangle_{m_*}$,
\eqref{eq-exp-upper-bound-P-alpha-time-derivatives}, and the Sobolev
embedding, we have, for any $r\ge0$ and $\ell \ge 2$,
\begin{align}
\label{eq-upper-bound-P-alpha-i-time-derivatives}
\sum_{\substack{\langle\alpha,\N\rangle = \ell\\\alpha_1 = 0}}
|\partial_t^rP_{\alpha,i}
|_{L^\infty(\R\times\mathbb S^{n-1})}
\le
CK_1^{\ell+r}r!,
\end{align}
where $K_1$ is a positive constant depending only on $n$, $\xi$, and 
$\widetilde A$.
Then, by interior estimates for homogeneous harmonic polynomials,
\eqref{eq-upper-bound-P-alpha-i-time-derivatives} yields,
for any $\beta \in \Z_+^n$,
\begin{align}
\label{eq-upper-bound-harmonic-extension-P-alpha}
\sum_{\substack{\langle\alpha,\N\rangle = \ell\\\alpha_1 = 0}}
|\partial_x^\beta\partial_t^rH_{\alpha,i}
|_{L^\infty(\R\times B_{3/2})}
\le
K_2^{\ell+r}\beta!r!,
\end{align}
for some positive constant $K_2$ depending only on
$n$, $\xi$, and $\widetilde A$.
We point out that $H_{\alpha,i}(t,\cdot)$ is a homogeneous polynomial of
degree $i\le \langle\alpha,\N\rangle$.
Now, a simple Taylor expansion yields
\begin{align*}
H_{\alpha,i}
\big(
t+\ln|\theta-sb|,
\theta-sb
\big)
&=
\sum_{p,r\ge0}
\frac{1}{p!r!}
\partial_t^r(-b\cdot\nabla_x)^p
H_{\alpha,i}(t,\theta)
\big(\ln|\theta-sb|\big)^r s^p
\\
&=
\sum_{k=0}^\infty
Q_{\alpha,i,k}(t,\theta)s^k,
\end{align*}
where
\begin{align}
\label{eq-def-Q-alpha-i-k}
Q_{\alpha,i,k}
=
\sum_{p=0}^k
\sum_{r=0}^{k-p}
\frac{1}{p!r!}
\partial_t^r(-b\cdot\nabla_x)^p
H_{\alpha,i}(t,\theta)
L_{r,k-p}(\theta).
\end{align}
It follows from \eqref{eq-upper-bound-L-r-j},
\eqref{eq-upper-bound-harmonic-extension-P-alpha}, and
\eqref{eq-def-Q-alpha-i-k} that, for any $\ell \ge 2$ and $k \ge 0$,
\begin{align}
\label{eq-upper-bound-Q-alpha-i-k}
\sum_{\substack{\langle\alpha,\N\rangle = \ell\\\alpha_1 = 0}}
\sum_{i=0}^{\ell}
|Q_{\alpha,i,k}|_{L^\infty(\R\times\mathbb S^{n-1})}
\le
C K_3^{k+\ell},
\end{align}
for some positive constant $K_3$ depending only on
$n$, $\xi$, $b$, and $\widetilde A$.

Next, by $|\theta - sb| = \exp \{\ln |\theta - sb|\}$
and a simple Taylor expansion, write
\begin{align*}
|\theta-sb|^{-a-\langle\alpha,\rho\rangle-i}
=
\sum_{j=0}^\infty E_{\alpha,i,j}(\theta)s^j,
\end{align*}
where
\[
E_{\alpha,i,j}
=
\sum_{r=0}^j
\frac{1}{r!}(-a-\langle\alpha,\rho\rangle-i)^r
L_{r,j}.
\]
By \eqref{eq-upper-bound-L-r-j}, we have
\begin{align*}
|E_{\alpha,i,j}|_{L^\infty(\mathbb S^{n-1})}
&\le
K_0^j
\sum_{r=0}^j
\frac{1}{r!} (a+\langle\alpha,\rho\rangle+i)^r
\le
K_0^j
e^{a+\langle\alpha,\rho\rangle+i}.
\end{align*}
Since $i\le\langle\alpha,\N\rangle$ and
$\langle\alpha,\rho\rangle
\le
C\langle\alpha,\N\rangle$,
we obtain
\begin{align}
\label{eq-upper-bound-E-alpha-i-j}
|E_{\alpha,i,j}|_{L^\infty(\mathbb S^{n-1})}
\le
C K_4^{j+\langle\alpha,\N\rangle},
\end{align}
for some positive constant $K_4$ depending only on
$n$, $\xi$, and $b$.

By \eqref{eq-rewrite-P-alpha-harmonic-extension}, the coefficient of
$s^k$ in \eqref{eq-conformal-piece-P_alpha} is given by
\[
P_{\alpha,k}
=
\sum_{i=0}^{\langle\alpha,\N\rangle}
\sum_{j=0}^k
E_{\alpha,i,j}Q_{\alpha,i,k-j}.
\]
Hence, by \eqref{eq-upper-bound-Q-alpha-i-k},
\eqref{eq-upper-bound-E-alpha-i-j}, and a simple summation, we obtain
\begin{align*}
\sum_{\substack{\langle\alpha,\N\rangle = \ell\\\alpha_1 = 0}}
|P_{\alpha,k}|_{L^\infty(\R\times\mathbb S^{n-1})}
\le
C K^{k+\ell},
\end{align*}
for some positive constant $K$ depending only on
$n$, $\xi$, $b$, and $\widetilde A$.
We now proceed as in the argument for $P_{0,k}$ above.
For the $t$-derivatives, the same argument applies, 
since the $t$-derivatives act only on $H_{\alpha,i}$ in 
\eqref{eq-def-Q-alpha-i-k}. Thus, by shifting the number of 
$t$-derivatives in \eqref{eq-upper-bound-harmonic-extension-P-alpha},
we obtain, for any $p\ge0$,
\begin{align*}
\sum_{\substack{\langle\alpha,\N\rangle = \ell\\\alpha_1 = 0}}
|\partial_t^pP_{\alpha,k}
|_{L^\infty(\R\times\mathbb S^{n-1})}
\le
C_p(2K)^{k+\ell},
\end{align*}
for some positive constant $C_p$ depending on 
$n$, $\xi$, $b$, $\widetilde A$, and $p$.
On the other hand, since $P_{\alpha,k}$ has a spherical degree at most 
$\langle\alpha,\N\rangle + k$,
we get, for any $p,q \ge 0$,
\begin{align*}
\sum_{\substack{\langle\alpha,\N\rangle = \ell\\\alpha_1 = 0}}
|\nabla_\theta^q\partial_t^pP_{\alpha,k}
|_{L^\infty(\R\times\mathbb S^{n-1})}
\le
C_{p,q}(4K)^{k+\ell},
\end{align*}
for some positive constant $C_{p,q}$ depending on
$n$, $\xi$, $b$, $\widetilde A$, $p$, and $q$.
This finishes the proof of \eqref{eq-exp-upper-bound-conformal-expans-all-derivatives-2}.
\end{proof}

We now study another property of $\mathcal C_b \mathcal X$. 
Fix the constant $I_*$ 
as in Proposition \ref{prop-preservation-of-tail-conformal-change}. 
Then, $X_i = (\mathcal C_b \mathcal X)_i$ for any $i \ge I_*$.
For any $1 \le i \le I_*-1$,
we write
\begin{align*}
(\mathcal C_b \mathcal X)_i = X_i + \epsilon_i(\mathcal X, b),
\end{align*}
for some $\epsilon_i(\mathcal X,b)$ depending on $\mathcal X$ and $b$.
By \eqref{eq-C_bX_1}, we have
$\epsilon_1(\mathcal X,b) = \langle b,\theta\rangle.$
Hence, $\epsilon_1$ is independent of $\mathcal X$ and depends only on $b$.
In fact, $\epsilon_1:\R^n \rightarrow \mathscr E_1$ is a linear map.
For $2 \le i \le I_*-1$, we can prove easily 
that $\epsilon_i$ actually depends only on $X_1,\cdots, X_{i-1}$ and $b$. 
Moreover, we can prove that, for any $2 \le i \le I_*-1$,  
\begin{align}\label{prop-epsilon_i-polynomial} \epsilon_i  
\text{ is a smooth map from }(\mathscr E_1\oplus\cdots\oplus \mathscr E_{i-1})\times \R^n \text{ to }\mathscr E_i.\end{align}
By choosing bases of $\mathscr E_1, \cdots, \mathscr E_i$, 
we can view $\epsilon_i$ as a map between two Euclidean spaces.
We will not provide a proof of \eqref{prop-epsilon_i-polynomial} since it is not needed.

\subsection{Classifications of Solutions}\label{subsec-classifications}

We now prove the main existence result.

\begin{theorem}
\label{thrm-convergence-of-the-expansion-from-optimal-term-improved}
Let $\mathcal{X} =\{X_i\}_{i\geq 1} $ be a sequence of spherical harmonics
with $X_i \in \mathscr{E}_i$ 
and 
$\{P_\ell(\mathcal{X})\}_{\ell\ge 1}$ be the sequence of 
generalized polynomials determined by $\mathcal{X}$ 
as in 
Definition \ref{def-sequence-P-determined-by-X}. 
Suppose, 
for any $i \geq 1$,
\eqref{eq-exponential-growth-condition-expansion-from-optimal-term}
holds for some constant $\widetilde A>0$.
Then, there exists $t_0\ge 1$, depending only on $n$, $\xi$, and $\widetilde A$, such that the series 
\begin{align}
\label{eq-formal-expansion-improved}
\sum_{\ell = 1}^\infty P_\ell (\mathcal{X})(t,\theta) e^{-\mu_\ell t}
\end{align}
converges 
in the $C^k$-sense to a smooth function $\varphi$ 
on $[t_0,\infty)\times\mathbb{S}^{n-1}$, for any $k\ge 0$.
Moreover, $\varphi$ is a solution to \eqref{eq-linearized-Yamabe-at-xi} 
on $[t_0,\infty)\times \mathbb S^{n-1}$ with 
its asymptotic behavior determined by $\mathcal{X} =\{X_i\}_{i \geq 1}$
as in Definition \ref{def-determined-behavior}.
\end{theorem}

\begin{proof} If $X_1 = 0$, the desired result follows 
from \eqref{eq-P-ell-expression} and
Theorem \ref{thrm-convergence-of-the-expansion-from-optimal-term}. 
Now, assume that $X_1(\theta) = \langle b,\theta \rangle$ for some $b \in \R^n\setminus \{0\}$.
Consider the sequence of spherical harmonics 
$\mathcal C_{-b}\mathcal X = \{(\mathcal C_{-b}\mathcal X)_i\}_{i\ge 1}$
as in Lemma \ref{lemma-formal-seq-spherical-harmonics}. 
Then, $(\mathcal C_{-b} \mathcal X)_1 = 0$ by \eqref{eq-C_bX_1}.
Moreover, Proposition \ref{prop-preservation-of-tail-conformal-change} 
implies $(\mathcal C_{-b} \mathcal X)_i = X_i$ for any $i \ge I_*$. 
Hence, by \eqref{eq-C_bX_i} and 
\eqref{eq-exponential-growth-condition-expansion-from-optimal-term},
we get
$$\|(\mathcal C_{-b} \mathcal X)_i\|_{L^2 (\mathbb S^{n-1})} \le (\widetilde A')^{i},$$
for some constant $\widetilde A'$, depending only on $n$, $\xi$, and $\widetilde A$.

Lemma \ref{lemma-group-property-of-C_bP} yields  
$\mathcal X = \mathcal C_b(\mathcal C_{-b} \mathcal X)$.
Hence, by \eqref{eq-decomposition-C_bP_ell-multiindex} with $\mathcal X$
replaced by $\mathcal C_{-b} \mathcal X$, we get
\begin{align}
\label{eq-decomposition-P_ell-with-linear-term}
P_\ell(\mathcal X) = \sum_{\alpha, \langle \alpha, \rho \rangle = \mu_\ell } P_\alpha(\mathcal C_{-b} \mathcal X) 
 +  \sum_{\substack{k, k = \mu_\ell}} P_{0,k}(b) + \sum_{\substack{(\alpha,k),
\langle\alpha,\rho\rangle + k = \mu_\ell}} P_{\alpha,k} (b,\mathcal C_{-b}\mathcal X),
\end{align}
where $\{P_\alpha(\mathcal C_{-b} \mathcal X)\}_{\alpha \in \ind}$ is the
sequence of generalized polynomials determined by $\mathcal C_{-b}\mathcal X$
as in Proposition \ref{prop-formal-expansion-multi-index},
$\{P_{0,k}(b)\}_{k\ge 1}$ is the sequence of generalized polynomials 
determined by $b$ as in \eqref{eq-conformal-expansion-by-b-zero-order},
and $\{P_{\alpha,k}(b,\mathcal C_{-b} \mathcal X)\}_{\alpha \in \ind,k\ge 1}$
is the sequence of generalized polynomials determined by $b$ and 
$\mathcal C_{-b} \mathcal X$ as in \eqref{eq-conformal-expansion-P_alpha}.
Since $(\mathcal C_{-b}\mathcal X)_1 = 0$, 
by Theorem \ref{thrm-convergence-of-the-expansion-from-optimal-term}
and Proposition \ref{prop-convergence-of-conformal-expansion},
the three series
$$\sum_{\alpha\in \ind} P_\alpha(\mathcal C_{-b} \mathcal X) e^{-\langle\alpha,\rho\rangle t},\quad 
  \sum_{k=1}^\infty P_{0,k}(b) e^{-k t},\quad
 \sum_{\alpha \in \ind,k\ge 1} P_{\alpha,k} (b,\mathcal C_{-b}\mathcal X) e^{-(\langle\alpha,\rho\rangle + k)t}$$
converge in the $C^k$-sense, and their sum defines a 
smooth function $\varphi$ on $[t_0,\infty)\times\mathbb S^{n-1}$,
for some $t_0 \ge 1$ depending only on $n$, $\xi$, and $\widetilde A$.
By \eqref{eq-decomposition-P_ell-with-linear-term},
we obtain the convergence of the series in
\eqref{eq-formal-expansion-improved}.

Now, by \eqref{eq-upper-bound-P_alpha e^-alpha,rho-sum-ell},
\eqref{eq-exp-upper-bound-conformal-expans-all-derivatives-1},
and \eqref{eq-exp-upper-bound-conformal-expans-all-derivatives-2},
we can enlarge $t_0$ if necessary so that
$$\sum_{\ell=1}^\infty |P_\ell(\mathcal X) e^{-\mu_\ell t}|_{L^\infty([t_0,\infty)\times\mathbb S^{n-1})} \le \frac{1}{2M},$$
where $M$ is the constant in \eqref{eq-upper-bound-for-a_j}.
In particular, we can write
\begin{align*}
\sum_{j=2}^\infty a_j\varphi^j
&=
\sum_{j=2}^\infty a_j
\sum_{k_1,\ldots,k_j\ge1}
P_{k_1}(\mathcal X)e^{-\mu_{k_1}t} \cdots P_{k_j}(\mathcal X)e^{-\mu_{k_j}t}\\
&= \sum_{\ell = 2}^\infty \Big( \sum_{j = 2}^\ell \sum_{\mu_{k_1} + \cdots + \mu_{k_j} = \mu_\ell} 
a_j P_{k_1}(\mathcal X) \cdots P_{k_j} (\mathcal X) \Big) e^{-\mu_\ell t},
\end{align*}
where the last equality is justified by the 
uniform and absolute convergence of the series.
Together with the recursive equation 
\eqref{eq-recursive-equation-Asymp-expans} and the $C^2$-convergence of
\eqref{eq-formal-expansion-improved},
we conclude that $\varphi$ is a smooth solution of 
\eqref{eq-linearized-Yamabe-at-xi}.
\end{proof}

Theorem \ref{thrm-correspondence-harmonic-to-solutions-Yamabe} follows from 
Theorem \ref{thrm-convergence-of-the-expansion-from-optimal-term-improved}.
We simply note that any harmonic function $H$ near the origin with $H(0)=0$   
can be identified with a sequence 
of spherical harmonics $\mathcal{X} = \{X_i\}_{i\geq 1}$ with $X_i\in \mathscr{E}_i$ as in \eqref{eq-expansion-harmonic}, 
satisfying 
\eqref{eq-exponential-growth-condition-expansion-from-optimal-term}.
The convergence in Theorem \ref{thrm-correspondence-harmonic-to-solutions-Yamabe} is also uniform with respect to $|x|^k\nabla_x^k$ 
in $B_\epsilon\setminus\{0\}$, for any $k\ge 1$. 

We now prove the main convergence result in this paper. 

\begin{theorem}\label{thrm-higher-order-expansion-improved}
Assume $\varphi$ is a solution of \eqref{eq-linearized-Yamabe-at-xi} on $[t_0,\infty)\times\mathbb{S}^{n-1}$
satisfying \eqref{eq-exp-convergence} for some $\gamma, C_*>0$.
Then, there exists a sequence $\mathcal{X} = \{X_i\}_{i\geq 1}$ of spherical harmonics 
with $X_i \in \mathscr{E}_i$
such that
\begin{align}\label{eq-higher-order-expansion-improved}
\varphi = \sum_{\ell = 1}^\infty P_\ell(\mathcal X) e^{-\mu_\ell t},
\end{align}
where $\{P_\ell(\mathcal{X})\}_{\ell \geq 1}$ is the sequence of generalized polynomials 
determined by $\mathcal{X}$ as in Definition \ref{def-sequence-P-determined-by-X},  
and the series converges in the $C^k$-sense on 
$[t_1,\infty)\times \mathbb S^{n-1}$, for some $t_1 \ge t_0$ 
and any $k \ge 0$.
Moreover, each $P_\ell$ has polynomial degree $0$ for $\ell\ge1$,
and \eqref{eq-exponential-growth-condition-expansion-from-optimal-term}
holds for any $i\ge1$ and some positive constant $\widetilde A$.
\end{theorem}

\begin{proof}
Let $\mathcal{X} = \{X_i\}_{i\geq 1}$ be the sequence of spherical harmonics 
with $X_i \in \mathscr{E}_i$
and $\{P_\ell(\mathcal{X})\}_{\ell\ge 1}$ be the sequence of generalized polynomials 
determined by $\mathcal{X}$ as in Theorem \ref{thrm-higher-order-expansion}.  
If $X_1 = 0$, then Theorem \ref{thm:actual-solution-series-improved} implies the desired result. 

Now, assume that $X_1 = \langle b,\theta\rangle$ 
for some $b \in \mathbb R^n \setminus \{0\}$.
Set
$\widetilde \varphi = \mathcal C_{-b} (\xi + \varphi) - \xi.$
Then, $\widetilde \varphi$ is a solution of \eqref{eq-linearized-Yamabe-at-xi}.
By Lemma \ref{lemma-formal-seq-spherical-harmonics-actual},
$\mathcal C_{-b} \mathcal X$ is the sequence of spherical harmonics associated
with $\widetilde \varphi$.
In addition, $(\mathcal C_{-b} \mathcal X)_1 = 0$ by \eqref{eq-C_bX_1}.
By Theorem \ref{thm:actual-solution-series-improved},
for any $i \ge 2$,
$$\|(\mathcal C_{-b} \mathcal X)_i\|_{L^2 (\mathbb S^{n-1})} \le (\widetilde A')^{i},$$
for some constant $\widetilde A'$.
By Lemma \ref{lemma-group-property-of-C_bP},
we can write $ \mathcal X = \mathcal C_{b} (\mathcal C_{-b} \mathcal X)$.
Hence,
we conclude from Proposition \ref{prop-preservation-of-tail-conformal-change}
that \eqref{eq-exponential-growth-condition-expansion-from-optimal-term}
holds for $\mathcal X = \{X_i\}_{i\ge 1}$,
for any $i \ge 1$ and some constant $\widetilde A>0$.
Therefore, Theorem \ref{thrm-convergence-of-the-expansion-from-optimal-term-improved}
implies that the series in \eqref{eq-higher-order-expansion-improved} converges
in the $C^k$-sense to a smooth solution $\widehat \varphi$ on 
$[t_1,\infty)\times\mathbb S^{n-1}$, for some $t_1 \ge t_0$.
Moreover, the asymptotic behavior of $\widehat \varphi$ is
determined by $\mathcal{X} =\{X_i\}_{i \geq 1}$
as in Definition \ref{def-determined-behavior}.
In particular, $\varphi$ and $\widehat \varphi$ have the same asymptotic
behavior as $t \rightarrow \infty$.
By Theorem \ref{thrm-uniqueness-exponential-decay-sol}, 
$\varphi = \widehat \varphi$ on $[t_1,\infty) \times\mathbb S^{n-1}$.
This finishes the proof.
\end{proof}

Theorem \ref{thrm-HHW-intro-expansion} follows from 
Theorem \ref{thrm-higher-order-expansion-improved}.
The convergence in Theorem \ref{thrm-HHW-intro-expansion} is also uniform with respect to $|x|^k\nabla_x^k$ 
in $B_\epsilon\setminus\{0\}$, for any $k\ge 1$.

By Theorems \ref{thrm-convergence-of-the-expansion-from-optimal-term-improved} and \ref{thrm-higher-order-expansion-improved}, 
the map $\Phi$ in \eqref{eq-definition-map-Phi} establishes a one-to-one and onto correspondence between 
the collection of harmonic functions $H$ near the origin with $H(0)=0$ 
and the collection of smooth solutions $\varphi$ to \eqref{eq-linearized-Yamabe-at-xi} 
satisfying \eqref{eq-exp-convergence} 
for some $\gamma, C_*>0$.
With \eqref{prop-epsilon_i-polynomial}, 
we can prove that $\Phi$ is a $C^1$-map between two appropriately chosen function spaces. We omit details.



\end{document}